\documentclass{article}

\usepackage{amsmath}
\usepackage{amsthm}
\usepackage{amsfonts}
\usepackage{amsbsy}
\usepackage{amssymb}
\usepackage[initials]{amsrefs}
\newtheorem{theorem}{Theorem}
\newtheorem{proposition}{Proposition}
\newtheorem{corollary}{Corollary}

\newtheorem{lemma}{Lemma}

\newcommand{\R}{{\mathbb R}}
\newcommand{\Z}{{\mathbb Z}}
\newcommand{\set}[2]{ \left\{ #1 \ \left| \ #2 \right. \right\}}
\newcommand{\supp}{\mathop{\mathrm{supp}}}
\newcommand{\dist}{\operatorname{dist}}

\newcommand{\dH}{d {\mathcal H}}
\renewcommand{\H}{{\mathcal H}}
\newcommand{\li}[1]{{}^{#1}\Sigma}
\newcommand{\ri}[1]{\Sigma^{#1}}
\newcommand{\vis}{\operatorname{Vis}}
\newcommand{\ang}[1]{\left<#1\right>}

\title{Testing conditions for multilinear Radon-Brascamp-Lieb inequalities}
\author{Philip T. Gressman\footnote{Partially supported by NSF grants DMS-1764143 and DMS-2054602.}}
\date{\today}

\begin{document}
\maketitle
\begin{abstract}
This paper establishes a necessary and sufficient condition for $L^p$-boundedness of a class of multilinear functionals which includes both the Brascamp-Lieb inequalities and generalized Radon transforms associated to algebraic incidence relations. The testing condition involves bounding the average of an inverse power of certain Jacobian-type quantities along fibers of associated projections and covers many widely-studied special cases, including convolution with measures on nondegenerate hypersurfaces or on nondegenerate curves. The heart of the proof is based on Guth's visibility lemma \cite{guth2010} in one direction and on a careful analysis of Knapp-type examples in the other. Various applications are discussed which demonstrate new and subtle interplay between curvature and transversality and establish nontrivial mixed-norm $L^p$-improving inequalities in the model case of convolution with affine hypersurface measure on the paraboloid.
\end{abstract}

\tableofcontents

\section{Introduction}
\subsection{Background and main results}

Radon-like transforms are objects of extensive study in harmonic analysis, appearing in connection with singular integral theory (both single and multiparameter), microlocal analysis,  and Fourier-theoretic settings \cites{gs1984,ps1986I,ps1986II,cnsw1999,ss2011,ss2012,ss2013,street2012,gs1994,is1996,gs1998,lt2001,gs2002,ggip2015}. Such objects also find application in a wide variety of theoretical and applied problems even beyond the more well-known setting of medical imaging \cites{ms1999,rv2005,bhht2009,patel2018}. This paper introduces a class of multilinear inequalities which combine both Radon-like transforms and linear and nonlinear Brascamp-Lieb inequalities, which are, in their own right, tools of immense importance in modern Fourier analysis, particularly in decoupling theory (see \cites{bcct2008,bcct2010,bbcf2017,bcw2005,bb2010,bbg2013II,bbbcf2020,duncan2021,bdg2016}). The main result is a geometric characterization of the boundedness of these multilinear objects on products of Lebesgue spaces along a certain scaling line of exponents. The nature of the result is rather different than existing results concerning extremizers or quasi-extremals for Radon-like transforms \cites{stovall2009II,christ2011,biswas2019} and instead reduces the problem to what can be viewed as an unusual new category of uniform sublevel set inequalities.

The literature contains many useful ways to to describe and study Radon-like transforms (e.g., projections, fibrations, and vector fields 
\cites{tw2003,stovall2009,cdss2020}); the formulation that seems most helpful for present purposes is to work primarily with defining functions and incidence relations. To that end,
suppose $\Omega \subset \R^{n} \times \R^{n'}$ is an open set. Let $\pi : \Omega \rightarrow \R^{k}$ for some $k \leq \min\{n,n'\}$ and suppose that $\pi$ is smooth.  The zero set of $\pi$ will represent incidence pairs $(x,y) \in \R^{n} \times \R^{n'}$ such that the associated Radon-like transform, when evaluated at $x$, integrates functions over a submanifold passing through $y$. For this perspective to be applied in a straightforward way, it is necessary for the defining function $\pi$ to be nonsingular in a certain sense.
For any vectors $v_1,\ldots,v_k$ in $\R^n$, let $d_x \pi |_{(x,y)} (v_1,\ldots,v_k)$ be defined to equal the determinant of the $k \times k$ matrix whose $i,j$-entry is $\sum_{\ell=1}^n v_j^{\ell} \frac{\partial \pi^i}{\partial x^\ell}$
(throughout this paper, when some $v \in \R^{\ell}$ must be expressed in standard coordinates, superscript notation $(v^1,\ldots,v^\ell)$ is generally used).  Let $d_y \pi|_{(x,y)}(v_1,\ldots,v_k)$ be defined similarly as the determinant of the matrix whose $i,j$-entry is $\sum_{\ell=1}^{n'} v_j^{\ell} \frac{\partial \pi^i}{\partial y^\ell}$.
Finally, for any $n$-tuple $\omega := \{\omega_i\}_{i=1}^n$ of vectors in $\R^n$, let
\[ ||d_x \pi(x,y)||_\omega := \left[ \frac{1}{k!} \sum_{i_1=1}^n \cdots \sum_{i_k=1}^n \left| d_x \pi|_{(x,y)} ( \omega_{i_1},\ldots,\omega_{i_k}) \right|^2 \right]^\frac{1}{2} \]
and likewise set
\[ ||d_y \pi(x,y)||_{\omega'} := \left[ \frac{1}{k!} \sum_{i_1=1}^{n'} \cdots \sum_{i_k=1}^{n'} \left|  d_y \pi |_{(x,y)} (\omega'_{i_1},\ldots,\omega'_{i_k}) \right|^2 \right]^\frac{1}{2} \]
for any $n'$-tuple of vectors $\omega' := \{\omega'_i\}_{i=1}^{n'}$ in $\R^{n'}$. The notation $||d_x \pi(x,y)||$ and $||d_y \pi(x,y)||$ will be used when $\omega$ or $\omega'$ should be taken to be the tuple of standard basis vectors of $\R^n$ or $\R^{n'}$, respectively.

Any triple $(\Omega,\pi,\Sigma)$ will be called a smooth incidence relation on $\R^{n} \times \R^{n'}$ of codimension $k$ when $\Omega \subset \R^{n} \times \R^{n'}$ is open, $\pi : \Omega \rightarrow \R^{k}$ is smooth, and
\begin{equation} \Sigma := \set{ (x,y) \in \Omega}{ \pi(x,y) = 0, ||d_x \pi(x,y)||, ||d_y \pi(x,y)|| > 0}. \label{sigmadef} \end{equation}
As above, it will always be assumed that $k \leq \min\{n,n'\}$.
The notation $\li{x}$ and $\ri{y}$ will indicate slices of $\Sigma$ with fixed $x$ and $y$, respectively:
\begin{align*}
\li{x} & := \set{y \in \R^{n'}}{ (x,y) \in \Sigma} \text{ and }
\ri{y}  := \set{x \in \R^{n\vphantom{'}}}{(x,y) \in \Sigma}.
\end{align*}
On each slice $\li{x}$ and $\ri{y}$, $\sigma$ denotes what will be called the coarea measure (also known as the Leray or microcanonical measure elsewhere), given by 
\[ \int_{\li{x}} \! f d \sigma := \int_{\li{x}} \! \! f(y) \frac{\dH^{n'-k}(y)}{||d_y \pi (x,y)||} \text{ and } \int_{\ri{y}} \! f d \sigma := \int_{\ri{y}} \! \! f(x) \frac{\dH^{n-k}(x)}{||d_x \pi (x,y)||} \]
for any Borel-measurable function $f$ on the slices (Borel measurability is assumed for convenience throughout the paper to avoid technical difficulties associated with restricting Lebesgue-measurable functions to submanifolds), where $\dH^{s}$ is the usual $s$-dimensional Hausdorff measure.  

The main result of this paper is as follows. 
\begin{theorem} \label{mainthm} For any integer $m \geq 1$ and
each $j = 1,\ldots,m$, let $(\Omega_j,\pi_j,\Sigma_j)$ be a smooth incidence relation on $\R^{n} \times \R^{n_j}$ with codimension $k_j \leq \min \{n,n_j\}$ and let $\sigma_j$ denote the associated coarea measures on slices.
 Let $w_j : \Sigma_j \rightarrow [0,\infty)$ be continuous, and let $T_j$ be the generalized Radon-like transform given by
\begin{equation} T_j f(x) := \int_{\li{x}_j} f_j(y_j) w_j(x,y_j) d \sigma_j(y_j) \label{tjdef} \end{equation}
for all nonnegative Borel-measurable $f_j$ on $\R^{n_j}$.
Suppose $p_1,\ldots,p_m \in [1,\infty)$ and $q_1,\ldots,q_m \in (0,\infty)$ satisfy the scaling condition
\begin{equation} n = \sum_{j=1}^m \frac{k_jq_j}{p_j}. \label{scaling} \end{equation}
Let $||T||$ be the smallest positive constant (supposing one exists) such that for all nonnegative Borel measurable functions $f_j \in L^{p_j}(\R^{n_j})$,
\begin{equation} \int_{\R^n} \prod_{j=1}^m |T_j f_j(x)|^{q_j} dx  \leq ||T|| \prod_{j=1}^m ||f_j ||_{L^{p_j}(\R^{n_j})}^{q_j}. \label{bddness} \end{equation}
There exists a constant $C$ depending only on $n$ and $n_j,k_j,p_j,q_j$ for $j=1,\ldots,m$ such that for any $x \in \R^{n}$ and any vectors $\omega_1,\ldots,\omega_{n}$ with $|\det(\omega_1,\ldots,\omega_n)| = 1$ (where $\det (\omega_1,\ldots,\omega_n)$ is the determinant of the matrix whose columns are coordinates of $\omega_1,\ldots,\omega_n$ in the standard coordinate system),
\begin{equation}
\prod_{j \, : \, p_j = 1} \sup_{y_j \in \li{x}_j} \frac{|w_j(x,y_j)|^{q_j}}{||d_x \pi_j(x,y_j)||_\omega^{q_j}} \prod_{j \, : \, p_j > 1}  \! \left[ \int_{\li{x}_j} \!  \! \frac{|w_j(x,y_j)|^{p'_j} d \sigma_j(y_j) }{||d_x \pi_j(x,y_j)||_{\omega}^{p'_j-1}}  \right]^{\frac{q_j}{p'_j}} \! \leq C ||T|| \label{testing}
 \end{equation}
 where for each $j$, $p_j$ and $p_j'$ are H\"{o}lder dual exponents.
 Conversely, suppose $[[T]]$ is defined to be the supremum of
 \begin{equation}
\prod_{j \, : \, p_j = 1} \sup_{y_j \in \li{x}_j} \frac{|w_j(x,y_j)|^{q_j}}{||d_x \pi_j(x,y_j)||_\omega^{q_j}} \prod_{j \, : \, p_j > 1}  \left[ \int_{\li{x}_j} \! \! \frac{|w_j(x,y_j)|^{p'_j} d \sigma_j(y_j) }{||d_x \pi_j(x,y_j)||_{\omega}^{p'_j-1}}  \right]^{\frac{q_j}{p'_j}}  \label{testing2}
 \end{equation}
 over all $x \in \R^n$ and all $\{\omega_i\}_{i=1}^n$ with $|\det (\omega_1,\ldots,\omega_n)| = 1$. If $[[T]] < \infty$ 
 and each $\pi_j(x,y_j)$ is a polynomial function of $x$ with bounded degree as a function of $y_j$, then \eqref{bddness} holds for nonnegative $f_j$ with a finite value of $||T||$ satisfying $||T|| \leq  C' [[T]] \prod_{j=1}^m (\deg \pi_j)^{q_j/p_j} $ for some $C'$ depending only on $n, $ and $n_j,k_j,p_j,q_j$, where $\deg \pi_j := \sup_{y_j} \deg \pi^1_j(\cdot,y_j) \cdots \deg \pi^k_j(\cdot,y_j)$.
\end{theorem}
When the defining functions $\pi_j$ have the form $\pi_j(x,y_j) := y_j - L_j(x)$ for some linear map $L_j : \R^n \rightarrow \R^{k_j}$ of full rank, the inequality \eqref{bddness} reduces to the classical Brascamp-Lieb inequality. In this case, the testing condition \eqref{testing} simplifies significantly because the slices $\li{x}_j$ are simply points and the coarea measure is simply a delta measure at $y_j = L_j(x)$. The condition on the maps $L_j$ that results from \eqref{testing2} can be understood using ideas from Geometric Invariant Theory \cite{gressman2021}, and in particular,  the supremum of \eqref{testing2} over all $\{\omega_i\}_{i=1}^n$ with determinant $\pm 1$ exactly equals a constant multiple of the Brascamp-Lieb constant as a consequence of \cite{gressman2021}*{Lemma 1}.

Radon-like transforms can of course also be written in terms of defining functions $\pi_j$. The scaling condition \eqref{scaling} is in some cases not necessary for boundedness: several very general works on $L^p$-improving properties of Radon-like transforms, including the groundbreaking works of Tao and Wright \cite{tw2003}, Stovall \cites{stovall2010,stovall2011,stovall2014} and Christ, Dendrinos, Stovall and Street \cite{cdss2020}, include positive results beyond the scaling line \eqref{scaling}.  A number of other important results also fail to be captured by \eqref{scaling} and Theorem \ref{mainthm}, including results in mixed-norm Lebesgue spaces \cites{os1982,christ1984,oberlin2010,ce2008,eo2010,hickman2016} and results focusing specifically on minimal regularity assumptions for associated submanifolds (e.g., \cite{bbg2013}). These exceptional works notwithstanding, there is a truly vast body of literature which pertains specifically to the scaling line \eqref{scaling}.
The famous $L^{(n+1)/n} \rightarrow L^{n+1}$ inequality for averages over curved hypersurfaces, first proved by Littman \cite{littman1971}, has the scaling \eqref{scaling}, as do all results in intermediate dimensions for maximally-curved ``model surfaces'' \cites{ricci1997,oberlin2008}. Endpoint estimates for convolution with affine arclength on the moment curve, first proved in the restricted weak-type sense in all dimensions in the groundbreaking work of Christ \cite{christ1998}, also belong to the scaling line \eqref{scaling} (or more precisely, one of the two endpoint inequalities falls on that line, and the other follows by duality).
See \cites{dg1991,ps1994,oberlin1999,bak2000,bos2002} for just a few additional unweighted examples and 
\cites{secco1999,oberlin1999II,oberlin2000II,oberlin2002,choi2003,dlw2009} for various weighted cases (because the weights $w_j$ in \eqref{tjdef} are essentially arbitrary, Theorem \ref{mainthm} covers affine weights and, after a limiting argument, extends to fractional integration kernels as well).
This frequency is due to the fact that bounds for \eqref{bddness} on the given scaling are automatically sharp in the sense that no bounds can hold when the right-hand side of \eqref{scaling} is strictly larger than $n$ (see the end of Section \ref{necessarysec} for justification).  Thus, even in the linear case $m=1$, Theorem \ref{mainthm} represents a significant advance in the understanding of many of the most fundamentally-important $L^p$-improving inequalities for Radon-like transforms of any dimension and codimension.
 
Like the techniques of the recent paper \cite{gressman2021}, the method of proof used here is neither combinatorial in the typical way (involving the construction and analysis of inflation maps) nor Fourier analytic. This new approach circumvents a number of recurring limitations of these common strategies. For example, it is clear from the statement of Theorem \ref{mainthm} that there are no special constraints on the dimensions $n,n_j,$ and $k_j$ in which the theorem applies, while constraints of this sort frequently arise when working with inflation map technology. Moreover, even in comparison to \cite{gressman2021}, the current proof involves a number of critical improvements. One key change is that the approach to be used here does not require a direct analysis of any nonconcentration inequalities, which was a key component of \cite{gressman2021}. This is due to a significant shift in the way that the Kakeya-Brascamp-Lieb inequality is formulated (compare Theorem \ref{main1} in Section \ref{possec} to \cite{gressman2021}*{Theorem 1}). The shift illustrates that, while in some cases it is natural to use Brascamp-Lieb inequalities as a means of building sharp weights in Kakeya-type inequalities (e.g., Zhang's variety version of Brascamp-Lieb \cite{zhang2018}*{Theorem 8.1}), in this more general setting it turns out to be beneficial to be somewhat agnostic about the sort of weights that should be encountered and to proceed along slightly more abstract lines.

Theorem \ref{mainthm}'s criterion \eqref{testing} may be regarded as roughly analogous to a sort of uniform sublevel set inequality, which transforms the problem of proving \eqref{bddness} into a very different and more tractable form to which a host of powerful tools (e.g., \cites{ccw1999,pss2001,cw2002,greenblatt2005mrl,gressman2010II}) may be applied after suitable adaptation.  Even so, the estimation of \eqref{testing2} is not trivial, particularly for intermediate dimensions (averaging over submanifolds that are neither curves nor hypersurfaces). The problem of computing the supremum of \eqref{testing2} under relatively general conditions will be taken up in a follow-up series of papers. Section \ref{examplesec} does contain several simple examples of how the necessary computations can be accomplished; a primary application of Theorem \ref{mainthm} appearing in Section \ref{examplesec} is the following result.
\begin{theorem}
For any integers $2 \leq \ell \leq n$, any functions $f_1,\ldots,f_n$ on $\R^{\ell}$, and any exponent $p \in [1,\infty)$, \label{multiobj}
\begin{equation} \begin{split} \left[ \int_{\R^n} \left| \prod_{j=1}^n \int_{\R^{\ell-1}} f_j( x^{j+1} + t^1,\ldots, x^{j+\ell-1} + t^{\ell-1}, x^{j+\ell} + ||t||^2) dt \right|^{p} dx \right]^\frac{1}{p} \hspace{-14pt} & \\ \leq C_{p,\ell,n} \prod_{j=1}^n  ||f_j &  ||_{L^p(\R^{\ell})} \end{split} \label{examp1} \end{equation}
(where indices of $x$ are interpreted periodically with period $n$, e.g., $x^{n+1} = x^1$, etc., and $||t||^2 = (t^1)^2 + \cdots + (t^{\ell-1})^2$) for some finite constant $C_{p,n,\ell}$ depending only on $n, \ell$, and $p$, if and only if 
\[ \frac{n+1}{n} \leq p < \frac{\ell}{\ell-1}. \]
When $p = \ell / (\ell-1)$, the restricted strong-type analogue of \eqref{examp1} holds.
\end{theorem}
The inequality \eqref{examp1} may be viewed as a ``Radon-Brascamp-Lieb inequality,'' combining features of both Radon-like transforms and Brascamp-Lieb inequalities. A particularly interesting feature of \eqref{examp1} is that applying classical Brascamp-Lieb inequalities and known inequalities for convolution with affine hypersurface measure on the paraboloid in $\R^\ell$ establish \eqref{examp1} when $p = (\ell+1)/\ell$, but fail to explain why the inequality \eqref{examp1} must be true for the remaining ranges of $p$. In particular, even in the case $\ell = n$, the inequality \eqref{examp1} holds for a broader range of $p$ than the convolution inequality and H\"{o}lder's inequality combined would otherwise suggest.
Thus in some sense, the inequality \eqref{examp1} necessarily depends on some deeper interplay between the transversality and curvature properties of the relevant objects than can be understood through a naive approach.

\subsection{Notation} 
Although defining functions $\pi(x,y)$ will be essentially ubiquitous throughout this paper, there will be only a few specific points at which it is necessary to consider the simultaneous dependence of $\pi$ on both $x$ and $y$. As a consequence, it will be convenient in most cases to use notation which suppresses dependence on one or the other of the two ``sides'' of $\pi(x,y)$ and to think primarily in terms of one-sided computations and parametrized perturbations of one-sided objects. To be specific, supposing that $\pi$ is some smooth map from an open subset of $\R^n$ into $\R^k$, the notation $D \pi(x)$ will be used to indicate the Jacobian matrix of $\pi$ at a point $x$ in the standard coordinates, i.e., when $x := (x^1,\ldots,x^n)$ and $\pi(x) = (\pi^1(x),\ldots,\pi^k(x))$, then
\begin{equation} D \pi(x) := \begin{bmatrix} \frac{\partial \pi^1}{\partial x^1}(x) & \cdots & \frac{\partial \pi^1}{\partial x^n}(x) \\ \vdots & \ddots & \vdots \\ \frac{\partial \pi^k}{\partial x^1}(x) & \cdots & \frac{\partial \pi^k}{\partial x^n} (x) \end{bmatrix}. \label{jacobiandef} \end{equation}
When the point $x$ at which $D \pi(x)$ is to be calculated is clear from context, the shorthand notation $D \pi$ will be used.
If $\pi$ is defined on an open subset of a product space like $\R^{n} \times \R^{n'}$, the notation $D_x \pi(x,y)$ or $D_1 \pi(x,y)$ will indicate the $k \times n$ Jacobian matrix of $\pi$ with respect to the $x$ variables only (regarding $y$ as fixed); likewise $D_y \pi(x,y)$ and $D_2 \pi(x,y)$ both refer to the $k \times n'$ Jacobian matrix of $\pi$ with respect to the $y$ variables only. As before, the pair $(x,y)$ may in some cases be omitted when it is clear from context.

The notation $d \pi(x)$ indicates the $k$-fold wedge product $d \pi^1 \wedge \cdots \wedge d \pi^k$ at $x$:
\[ d \pi(x) := \bigwedge_{i=1}^k \left( \sum_{j=1}^n \frac{\partial \pi^i}{\partial x^j} (x) dx^j \right). \]
Just like above, when $\pi$ depends on multiple distinct groups of variables, notation like $d_x \pi(x,y)$ or $d_1 \pi |_{(x,y)}$ indicates that the $x$ or first collection of variables should be used for differentiation. This $d \pi(x)$ will also be regarded as a $k$-linear functional on vectors in $\R^n$ whenever it is convenient to do so: the symbol $d \pi|_x ( v_1,\ldots,v_k)$ indicates the result of evaluating $d \pi(x)$ on the $k$-tuple of vectors $v_1,\ldots,v_k$, i.e., if $v_i$ has coordinates $(v_{i}^1,\ldots, v_{i}^n)$ in the standard basis, then
\[ d \pi |_x(v_1,\ldots,v_k) := \det \left[ \sum_{\ell=1}^n \frac{\partial \pi^i}{\partial x^\ell} v_{j}^\ell \right]_{i,j = 1,\ldots,k}. \]
Note that one can also evaluate $d \pi |_x (v_1,\ldots,v_n)$ as $\det  (D \pi(x) V)$, where $V$ is the $n \times k$ matrix of coordinates of the vectors $v$, i.e., the row $\ell$, column $j$ entry of $V$ is simply $v_{j}^\ell$.

As already noted, it will be important to quantify the size of $d \pi(x)$ in essentially arbitrary local coordinate systems. In analogy with the definition already given, when $\omega_1^x,\ldots,\omega_n^x$ are pointwise linearly-independent vector fields on $\R^n$, define
\[ ||d \pi(x)||_{\omega^x} := \sqrt{ \frac{1}{k!} \sum_{i_1=1}^n \cdots \sum_{i_k=1}^n \left| d \pi |_x (\omega_{i_1}^x,\ldots, \omega_{i_k}^x) \right|^2}. \]
When no $\omega$ is specified, the notation $||d \pi(x)||$ indicates that the standard basis vectors on $\R^n$ should be used at every point.

\subsection{About the organization of this paper}

The proofs of Theorems \ref{mainthm} and \ref{multiobj} are divided into several stages. Section \ref{computevissec} provides some important identities regarding $d_x \pi$ (found in Section \ref{computesubsec}) and a proof of Lemma \ref{guthlemma}, which is the main technical lemma driving the sufficiency direction of Theorem \ref{mainthm}. This is the lemma which is based on Guth's visibility lemma; the lemma itself is formulated in such a way that  direct considerations of visibility can be confined exclusively to Section \ref{vissubsec}. The same is true of the algebraic constraint that $\pi_j(\cdot,y_j)$ be given by polynomial functions for each $j$---this assumption plays a role in the arguments of Section \ref{vissubsec} but is largely irrelevant elsewhere. 

Section \ref{possec} provides the proof of boundedness (i.e., the finiteness of $||T||$ in \eqref{bddness}) under the finiteness assumption on \eqref{testing2} and the algebraic assumption on the $\pi_j$. This is accomplished in two stages. The first stage is to establish Theorem \ref{main1} in Section \ref{proofsec1}, which is in some sense an analogue of Zhang's variety version of Brascamp-Lieb \cite{zhang2018}*{Theorem 8.1}. The main difference is that one \textit{does not} invoke the Brascamp-Lieb inequality, but rather uses an approach similar to Zhang's to estimate a more general object, which is called $Q(f_1,\ldots,f_m)$ in Theorem \ref{main1}. It turns out that this more abstract quantity $Q(f_1,\ldots,f_m)$ is often effectively larger than what would result from the Brascamp-Lieb power weight of Zhang's Theorem 8.1, particularly in the presence of curvature. The proof of \eqref{bddness} from \eqref{testing2} and the algebraic assumption on the $\pi_j$ is itself accomplished in Section \ref{mainpossec} as a consequence of Theorem \ref{restrictstrongtype}, which is a local version of the sufficiency portion of Theorem \ref{mainthm}, expanded to include the additional features of restricted strong-type inequalities and local estimates off the scaling line \eqref{scaling}.

Section \ref{examplesec} explores several applications: a corollary of Theorem \ref{mainthm} in the spirit of Stein's program to quantify the $L^p$-improving properties of convolution with singular measures \cite{stein1976}, a fractional integration-type result based on Theorem \ref{mainthm}, and Theorem \ref{multiobj} itself (proved under the assumption that Theorem \ref{mainthm} has been fully established). The nature of Theorem \ref{mainthm} means that the proof of Theorem \ref{multiobj} reduces to the analysis of the quantity \eqref{testing2}. The necessary inequalities for determinants are established in Section \ref{detsec} and the remainder of the proof of Theorem \ref{multiobj} appears in Section \ref{intsec}. Section \ref{mixnorm} repurposes some computations from Section \ref{multiobj} to establish an endpoint restricted strong-type mixed norm inequality for convolution with affine hypersurface measure on the paraboloid.

Section \ref{necessarysec} contains the proof of necessity of \eqref{testing} under the assumption that \eqref{bddness} holds. The proof is essentially a careful quantitative analysis of certain optimized Knapp-type examples.

Finally, Section \ref{appendix} is an appendix which establishes a number of background results \textit{a la} geometric measure theory concerning the behavior of smooth incidence relations and smooth perturbations of incidence relations. Readers will likely find the results of this section to be minor variations on existing results, but their proofs have been included for completeness, due to the fact that the somewhat qualitative nature of the definition \eqref{sigmadef} means that many of these elementary facts do not quite follow trivially from nice existing versions of the coarea formula, etc.

\section{Basic computations and visibility}
\label{computevissec}
\subsection{Initial computations regarding $d_x \pi$}
\label{computesubsec}
This section contains two very basic computations concerning $d \pi$ which will be used repeatedly throughout the remainder of this paper. Both deal with alternate ways of calculating the magnitude of or generally understanding the nature of $d_x \pi(x,y)$ for some smooth incidence relation $(\Omega,\pi,\Sigma)$.
\begin{proposition}
Suppose $(\Omega,\pi,\Sigma)$ is a smooth incidence relation on $\R^{n} \times \R^{n'}$ of codimension $k$.  At every point $(x,y) \in \Sigma$, \label{matprop}
\begin{equation} ||d_x \pi(x,y)|| = \sqrt{\det (D_x \pi(x,y)) (D_x \pi(x,y))^T} \label{equality1} \end{equation}
and
\begin{equation} ||d_y \pi(x,y)|| = \sqrt{\det (D_y \pi(x,y)) (D_y \pi(x,y))^T}. \label{equality2} \end{equation}
\end{proposition}
\begin{proof}
The proof of the proposition consists entirely of a string of observations about matrices and has nothing in particular to do with the geometric structure of $\Sigma$.  For any $k \times n$ matrix $M$, let $M_{i_1 \cdots i_k}$ be given by
\[ M_{i_1 \cdots i_k} := \det \begin{bmatrix} M_{1 i_1} & \cdots & M_{1 i_k} \\ \vdots & \ddots & \vdots \\ M_{k i_1} & \cdots & M_{k i_k} \end{bmatrix} \]
for any $i_1,\ldots,i_k \in \{1,\ldots,n\}$. It suffices to show that
\begin{equation} \det M M^T = \frac{1}{k!} \sum_{i_1=1}^{n} \cdots \sum_{i_k=1}^n \left| M_{i_1 \cdots i_k} \right|^2. \label{genmat} \end{equation}
Once this is established, the identities \eqref{equality1} and \eqref{equality2} follow by taking $M := D_x \pi$ and $M := D_y \pi$, respectively.

To begin, observe that both sides of \eqref{genmat} are unchanged when $M$ is replaced by $O_1 M$ for any $k \times k$ orthogonal matrix $O_1$: on the left-hand side this is because $\det (O_1 M)(O_1M)^T = (\det O_1) (\det M M^T) (\det O_1^T)= \det M M^T$, and on the right-hand side it is because $(O_1 M)_{i_1\cdots,i_k} = (\det O_1) M_{i_1 \cdots i_k} = \pm M_{i_1 \cdots i_k}$.  It is also the case that replacing $M$ by $M O_2$ for any $n \times n$ orthogonal matrix $O_2$ preserves both sides of \eqref{genmat}. This is more immediate to verify for the left-hand side because $(M O_2) (M O_2)^T = M (O_2 O^T_2) M^T = M M^T$. The computation for the right-hand side is a bit lengthier; to simplify, the subscript of $O_2$ will be temporarily suppressed. Substituting $MO$ in the place of $M$ on the right-hand side of \eqref{genmat} gives
\begin{align*}
\frac{1}{k!} & \sum_{i_1=1}^n \cdots \sum_{i_k=1}^n | (MO)_{i_1 \cdots i_k}|^2  = \frac{1}{k!} \sum_{i_1, \ldots, i_k} \left| \sum_{j_1,\ldots,j_k} M_{j_1 \cdots j_k} O_{j_1 i_1} \cdots O_{j_k i_k} \right|^2 \\
& = \frac{1}{k!} \sum_{i_1,\ldots,i_k} \sum_{j_1,\ldots,j_k} \sum_{j'_1,\ldots,j'_k} M_{j_1 \cdots j_k} M_{j_1' \cdots j_k'} O_{j_1 i_1} \cdots O_{j_k i_k} O_{j'_1 i_1} \cdots O_{j'_k i_k}.
\end{align*}
Summing over $i_1,\ldots,i_k$ first simplifies the expression significantly because 
\[ \sum_{i} O_{j i} O_{j' i} = \delta_{j j'} \]
for every pair $j,j' \in \{1,\ldots,n\}$, where $\delta$ is the Kronecker delta. Therefore
\[ \frac{1}{k!}  \sum_{i_1=1}^n \cdots \sum_{i_k=1}^n | (MO)_{i_1 \cdots i_k}|^2  = \frac{1}{k!} \sum_{j_1=1}^n \cdots \sum_{j_k=1}^n |M_{j_1 \cdots j_k} |^2 \]
as asserted. Now by the Singular Value Decomposition, there exist orthogonal matrices $O_1$ and $O_2$ such that $O_1 M O_2$ has its only nonzero entries on the diagonal. Let $\sigma_i$ denote the $i$-th diagonal entry of this matrix. Then clearly
\[ \det M M^T = \det (O_1 M O_2) (O_1 M O_2)^T = \prod_{i=1}^k \sigma_i^{2} \]
because $O_1 M O_2 (O_1 M O_2)^T$ is itself a diagonal matrix whose $i$-th diagonal entry is $\sigma_i^2$.
Similarly
\begin{equation} \frac{1}{k!} \sum_{j_1,\ldots,j_k} |M_{j_1 \cdots j_k}|^2 = \frac{1}{k!} \sum_{i_1 \cdots i_k} | (O_1 M O_2)_{i_1 \cdots i_k}|^2 = \prod_{i=1}^{k} \sigma_i^{2} \label{reusecalc} \end{equation}
because $(O_1 M O_2)_{i_1 \cdots i_k}$ vanishes unless $(i_1,\ldots,i_k)$ is a permutation of $(1,\ldots,k)$, in which case it equals $\prod_{i=1}^k \sigma_i$. Thus \eqref{genmat} is true and \eqref{equality1} and \eqref{equality2} follow.
\end{proof}

The second and final basic proposition to be proved at this point provides a dictionary of sorts to translate the notation of this paper into the form used by Zhang \cite{zhang2018}, which will become relevant shortly.
\begin{proposition}
Let $(\Omega,\pi,\Sigma)$ be a smooth incidence relation on $\R^{n} \times \R^{n'}$ of codimension $k$. Then at each point $x \in \ri{y}$, \label{formprop}
\begin{equation}
\frac{d_x \pi(x,y)}{||d_x \pi(x,y)||} \label{normalform}
\end{equation}
equals a $k$-fold wedge product $\omega_1^* \wedge \cdots \wedge \omega_k^*$, where the covectors $\omega_1^*,\ldots,\omega_k^*$ are orthonormal and annihilate the tangent space of $\ri{y}$ at $x$. In the notation of Zhang \cite{zhang2018}, this means that \eqref{normalform} equals $(T_x \ri{y})^{\perp}$ up to a factor of $\pm 1$. 
\end{proposition}
\begin{proof}
This result is closely related to Proposition \ref{matprop}.
Consider the Jacobian
\[ D_x \pi(x,y) = \begin{bmatrix} \frac{\partial \pi^1}{\partial x^1} (x,y) & \cdots & \frac{\partial \pi^1}{\partial x^n} (x,y) \\ \vdots & \ddots & \vdots \\  \frac{\partial \pi^k}{\partial x^1} (x,y) & \cdots & \frac{\partial \pi^k}{\partial x^n} (x,y) \end{bmatrix}. \]
By the Singular Value Decomposition, there is a $k \times k$ orthogonal matrix $O$ such that the row vectors
\[ R_i := \left( \sum_{j=1}^k O_{ij} \frac{\partial \pi^{j}}{\partial x^{\ell}} (x,y) , \ldots, \sum_{j=1}^{k} O_{ij} \frac{\partial \pi^j}{\partial x^n} (x,y) \right) \]
are pairwise orthogonal. Consequently 
\begin{equation}
\begin{split} \det \left[ (D_x \pi(x,y)) (D_x \pi (x,y))^T \right] & = \det \left[ O (D_x \pi(x,y)) (D_x \pi (x,y))^T O^T \right] \\ & = \prod_{i=1}^k ||R_i||^2 \end{split} \label{perpcomp} \end{equation}
because $(O D_x \pi) (O D_x \pi)^T$ is simply a diagonal matrix whose $i$-th diagonal entry is exactly $||R_i||^2$.
For each $i=1,\ldots,n$, let
\[ r_i^* := \sum_{j=1}^k \sum_{\ell=1}^{n} O_{ij} \frac{\partial \pi^j}{\partial x^\ell} d x^\ell. \]
The covectors $r_1^*,\ldots,r_k^*$ are pairwise orthogonal with $||r_i^*|| = ||R_i||$ for each $i$.
Observe also that $r_1^* \wedge \cdots \wedge r_k^* = (\det O) (d \pi^1 (x,y) \wedge \cdots \wedge d \pi^k(x,y))$. This is true because the map
\[ M \mapsto \left( \sum_{j=1}^k \sum_{\ell=1}^{n}  M_{1j} \frac{\partial \pi^j}{\partial x^\ell} d x^{\ell} \right) \wedge \cdots \wedge \left( \sum_{j=1}^k \sum_{\ell=1}^{n}  M_{kj} \frac{\partial \pi^j}{\partial x^\ell} d x^{\ell} \right) \]
is an alternating $k$-linear functional of the rows of $M$. Since scalar-valued alternating $k$-linear functionals on $\R^k$ are unique up to a scalar multiple, it follows by writing the coordinates of the above $k$-form in the standard basis that
\begin{align*}
 \left( \sum_{j=1}^k \sum_{\ell=1}^{n}  M_{1j} \frac{\partial \pi^j}{\partial x^\ell} d x^{\ell} \right) \wedge \cdots \wedge \left( \sum_{j=1}^k \sum_{\ell=1}^{n}  M_{kj} \frac{\partial \pi^j}{\partial x^\ell} d x^{\ell} \right) &  \\ = (\det M) (d_x \pi^1(x,y) \wedge  \cdots \wedge d_x \pi^k(x,y)) & = (\det M) d_x \pi(x,y)
 \end{align*}
 (because each coefficient of the $k$-fold wedge product in the standard basis is a scalar alternating $k$-linear functional of the rows of $M$ and so equals $\det M$ times its value when computed on the identity matrix).
 To finish, if one fixes $\omega_i^* := r_i^* / ||R_i||$, then it follows that the $\omega_i^*$ are now orthonormal covectors, each of which annihilates all vectors tangent to $\ri{y}$ at the point $x$, and
 \[ \frac{d_x \pi^1 (x,y) \wedge \cdots \wedge d_x \pi^k(x,y)}{||d_x \pi(x,y)||} = \pm \frac{r_1^* \wedge \cdots \wedge r_k^*}{||R_1|| \cdots ||R_k||} = \pm \omega_1^* \wedge \cdots \wedge \omega_k^*, \]
where \eqref{perpcomp} and \eqref{equality1} are used to establish that $||d_x \pi(x,y)|| = ||R_1|| \cdots ||R_k||$.
 In other words, one has
\[ \frac{d_x \pi(x,y)}{||d_x \pi(x,y)||} = \pm (T_x \ri{y} )^{\perp} \]
in the notation of Zhang \cite{zhang2018}*{p. 557}.
\end{proof}

\subsection{Visibility considerations and Lemma \ref{guthlemma}} 
\label{vissubsec} 
The stage is now set to prove the main lemma behind the sufficiency of \eqref{testing2} with regard to Theorem \ref{mainthm}. As mentioned earlier, this key lemma is primarily a reformulation of Guth's visibility lemma \cite{guth2010}, originally developed to prove the endpoint case of the Multilinear Kakeya conjecture  (previously formulated and proved up to the endpoint by Bennett, Carbery, and Tao \cite{bct2006}). Carbery and Valdimarsson \cite{cv2013} provide a very nice alternate proof of the lemma which is based on a variation of the Borsuk-Ulam Theorem and avoids a number of advanced tools from algebraic topology which featured in Guth's original approach. While these lemmas certainly involve visibility in a rather direct way, in the context of the present paper, it is perhaps misleading to think of them as being ``about'' visibility because the version of the lemma recorded below is, in fact, easiest to state without any reference to the notion of visibility at all. At its heart, the lemma below is a geometric lemma about the behavior of integrals of $||d \pi||_{\omega^x}$ along slices $\Sigma_\pi$.

\begin{lemma}
For any dimension $n \geq 1$, there exists a constant $C_n$ such that the following holds: for any positive integer $R$ and any Borel measurable, nonnegative integrable function $\psi$ on the box $B_R := [-R,R)^n$, there exist Borel measurable $\R^n$-valued functions $\omega_1^x,\ldots,\omega_n^x$ on $B_R$ (i.e., measurable vector fields) such that \label{guthlemma}
\begin{equation}
|\det \{\omega_i^x\}_{i=1}^n | = 1 \text{ for every } x \in B_R,
\end{equation}
($\det \{\omega_i^x\}_{i=1}^n$ is the determinant of the $n \times n$ matrix whose columns are $\omega_1^x,\ldots,\omega_n^x$ expressed in standard coordinates) and a nonnegative Borel-measurable function $\widetilde \psi$ on $B_R$ equal to $\psi$ almost everywhere such that every polynomial map $\pi : \R^n \rightarrow \R^{k}$ with $1 \leq k \leq n$ satisfies
\begin{equation}
\begin{split}
\int_{ \Sigma_\pi \cap B_R} \left[ \widetilde \psi(x) \right]^{\frac{n-k}{n}} &  || d \pi (x)||_{\omega^x} d \sigma(x) 
\leq C_n  (\deg \pi)   \left[ \int_{B_R} \psi(x) dx \right]^{\frac{n-k}{n}},
\end{split} \label{variety}
\end{equation}
where $\Sigma_\pi := \set{x \in \R^{n}}{ \pi(x) = 0, ||d_x \pi(x)|| > 0}$ and $d \sigma := \dH^{n-k} / ||d_x \pi||$.
 Here  $\deg \pi$ is the product of degrees of the coordinate functions of $\pi$, i.e., $\deg \pi := (\deg \pi^1) \cdots (\deg \pi^k)$.
\end{lemma}

As a brief aside, observe that the passage from $\psi$ to $\widetilde \psi$ in \eqref{variety} is necessary in general because the varieties $\Sigma_\pi$ are themselves sets of measure zero. If the function $\widetilde \psi$ were forced to equal $\psi$ everywhere, it would be possible, by adding a large multiple of $\chi_{\Sigma_\pi}$ to $\psi$ for some fixed $\pi$, to make the left-hand side of \eqref{variety} as large as desired without changing the right-hand side.

Before developing a full proof of the lemma, it is convenient to first handle the case $k=n$, as the only dependence of either side of \eqref{variety} on $\psi$ is through the vector fields $\{\omega_i^x\}_{i=1}^n$. It turns out that $||d \pi(x)||_{\omega^x}$ happens in this case to be independent of $\{\omega^x_i\}_{i=1}^n$. To see this, observe that
\[ || d \pi(x)||_{\omega^x}^2 = \frac{1}{n!} \sum_{i_1=1}^n \cdots \sum_{i_n=1}^n | d \pi|_x (\omega_{i_1}^x,\ldots,\omega_{i_n}^x)|^2 = |d \pi|_x(\omega^x_{1},\ldots,\omega^x_n)|^2. \]
As $d \pi$ is an alternating $n$-linear functional on $\R^n$, $|d \pi|_x(\omega^x_{1},\ldots,\omega^x_n)|$ must simply equal $|\det \{\omega_i^x\}_{i=1}^n|$ times some $x$-dependent function which is otherwise independent of $\{\omega^x_i\}_{i=1}^n$.  Since $|\det \{\omega_i^x\}_{i=1}^n|$ is constrained to equal $1$ everywhere, it follows that $||d \pi(x)||_{\omega^x} = ||d \pi(x)||$ at every $x$ (i.e., the value does not change when $\{\omega^x_i\}_{i=1}^n$ is replaced by the standard basis). This means that $||d \pi(x)||_{\omega^x} d \sigma$ is just $\dH^0$, i.e., counting measure. Therefore the left-hand side of \eqref{variety} simply counts nondegenerate solutions of the equation $\pi(x) = 0$ inside $B_R$ (it counts only nondegenerate solutions because $\Sigma_\pi$ contains only points where the Jacobian $\partial \pi / \partial x$ is nonsingular). B\'{e}zout's Theorem \cite{fulton1984}*{Chapter 8, Section 4} gives an upper bound of $\deg \pi$ for the number of such points (note that the complex version of the B\'{e}zout's Theorem which counts irreducible components of the solution variety is sufficient in this real setting because real nondegenerate solutions are also irreducible components of the solution variety over $\mathbb C$).

The rest of the proof of Lemma \ref{guthlemma} requires additional terminology and a few auxiliary propositions.
Let $\mu$ be any finite positive measure on $\R^n$. The fading zone $F(\mu)$ of this measure will be defined to equal the symmetric convex set
\begin{equation} F(\mu) := \set{ u \in \R^n}{ \int |u \cdot y^*| d \mu(y^*) \leq 1} \label{fadedef} \end{equation}
and the visibility $\vis(\mu)$ defined to equal
\begin{equation} \vis(\mu) = \frac{1}{|F (\mu)|} \label{visdef} \end{equation}
where $|\cdot|$ indicates Lebesgue measure. This definition deviates in a small but crucial way from that of other authors in that the fading zone is not assumed to be a subset of the unit ball. This makes the fading zone formally larger than the object of the same name considered elsewhere and consequently makes the visibility, as defined here, formally smaller than its standard counterpart.

Two propositions upon which Lemma \ref{guthlemma} rests are given below. The proof of Lemma \ref{guthlemma} will proceed immediately after the statement of both propositions. Once completed, the propositions themselves will be proved.
\begin{proposition}
Let $R$ and $B_R$ be as in Lemma \ref{guthlemma}.
Suppose $\delta := 2^{-j}$ for some nonnegative integer $j$, and let $\Lambda_\delta$ be the collection of boxes $[j_1 \delta,(j_1 + 1)\delta) \times \cdots \times [j_n \delta, (j_n+1) \delta)$ contained in the large box $B_R$, where $j_1,\ldots,j_n \in \Z$. 
\label{visibility}
Suppose $\psi$ is a function on the box $B_R$ which is nonnegative, constant on every box $Q \in \Lambda_\delta$, and not identically zero. For each $x \in B_R$, there is a finite positive measure $\mu^x$ on the unit sphere ${\mathbb S}^{n-1}$ such that
\begin{equation}
 \vis(\mu^x) \geq \psi(x) \text{ and } \label{visbound}
 \end{equation}
\begin{equation}
\int |u \cdot y^*| d \mu^x(y^*) \geq  \frac{||u||}{2R} \left(\int_{B_R} \psi(x) dx\right)^{\frac{1}{n}} \text{ for all } u \in \R^n. \label{bignorm0}
\end{equation}
The measures $\mu^x$ are constant as a function of $x$ on every box $Q \in \Lambda_\delta$, and for all polynomial maps $\pi : \R^n \rightarrow \R^k$ with $1 \leq k < n$,
\begin{equation}
\begin{split}
\int_{\Sigma_\pi} \left[ \int | d \pi (x) \wedge y_1^* \wedge \cdots \wedge y_{n-k}^*| d \mu^x(y_1^*) \cdots d \mu^x(y_{n-k}^*) \right] & d \sigma(x) \\
\leq C_n  \deg \pi
 \left[ \int_{B_R} \psi(x) dx \right]^{\frac{n-k}{n}} &\end{split} \label{whatisit}
\end{equation}
for some constant $C_n$ depending only on $n$, with $\Sigma_\pi$ and $d \sigma$ as in Lemma \ref{guthlemma}.
\end{proposition}

\begin{proposition}
Let $\mu$ be a finite positive measure on $\R^n$ such that \label{basis}
\[  \int |u \cdot y^*| d \mu(y^*) \geq \psi_0 ||u|| \text{ for all } u \in \R^n, \]
where $|| \cdot ||$ is the standard Euclidean norm and $\psi_0 > 0$. There exists a basis $\{\omega_i\}_{i=1}^n$ of $\R^n$ for which $|\det \{\omega_i\}_{i=1}^n| = 1$ such that for any $k \in \{1,\ldots,n-1\}$ and any $k$-form $A^*$ on $\R^n$, regarded as an alternating $k$-linear functional,
\begin{align}
\left( \frac{2}{n} \right)^{n-k} \left( \vis \mu \right)^{\frac{n-k}{n}} & \max_{i_1,\ldots,i_k } | A^*(\omega_{i_1},\ldots,\omega_{i_k})| \nonumber \\
& \leq \int |A^* \wedge y_1^* \wedge \cdots \wedge y_{n-k}^*| d \mu(y_1^*) \cdots d \mu(y_{n-k}^*) \label{midineq}.
\end{align}
\end{proposition}

\begin{proof}[Proof of Lemma \ref{guthlemma}] 
Using these two propositions, the proof of Lemma \ref{guthlemma} is rather routine. 
The first step is to handle the case when $\psi$ is a locally constant function of the sort described in Proposition \ref{visibility}.  If $\psi$ is identically zero, \eqref{variety} is trivially satisfied with $C_n = 1$, $\widetilde \psi \equiv 0$ and $\{\omega_i^x\}_{i=1}^n$ is the standard basis at every point because both sides of \eqref{variety} will be zero when $k < n$ (and the case $k = n$ holds for reasons already identified).
Fixing $\psi$, the measures $\mu^x$ from Proposition \ref{visibility} are constant on cubes $Q \in \Lambda_\delta$. Since $\int_{B_R} \psi > 0$, \eqref{bignorm0} implies that Proposition \ref{basis} may be applied pointwise to these measures $\mu^x$ to give $x$-dependent vectors $\{\omega^x_i\}_{i=1}^n$ (which will also be constant on cubes in $\Lambda_\delta$). For any $x$, \eqref{midineq} implies by taking $A^* = d \pi(x)$ that
\begin{align*}
\int & |d \pi(x) \wedge y_1^* \wedge \cdots \wedge y_{n-k}^*| d \mu^x(y_1^*) \cdots d \mu^x(y_{n-k}^*) \\
& \geq \left( \frac{2}{n} \right)^{n-k} \left( \vis{\mu^x} \right)^{\frac{n-k}{n}} \max_{i_1 \cdots i_k} | d \pi(\omega_{i_1}^x,\ldots, \omega_{i_k}^x)| \\
& \geq \left( \frac{2}{n} \right)^{n-k} \left( \vis{\mu^x} \right)^{\frac{n-k}{n}} \sqrt{\frac{k!}{n^k}} \left( \frac{1}{k!} \sum_{i_1 \cdots i_k=1}^n | d \pi(\omega_{i_1}^x,\ldots, \omega_{i_k}^x)|^2 \right)^{\frac{1}{2}}.
\end{align*}
Applying \eqref{visbound}, integrating over $\Sigma_\pi$, and employing \eqref{whatisit} gives \eqref{variety} with $\widetilde \psi := \psi$
(the constant can be chosen independently of $k$ because the number of possible values of $k$ is finite and depends only on $n$).

Now that \eqref{variety} has been established for all $\psi$ which are constant on dyadic boxes (and noting that one may assume without loss of generality that $\widetilde \psi = \psi$ in all such cases), the final step is to extend it to all nonnegative integrable Borel functions on $B_R$.  Let $\psi$ be any such function; it may be assumed that $\psi$ is positive on a set of positive measure in $B_R$, as otherwise \eqref{variety} will once again hold by taking $\{\omega_i^x\}_{i=1}^n$ to be the standard basis at every point $x \in B_R$ and taking $\widetilde \psi$ to be identically zero.
The vector fields $\{\omega_i^x\}_{i=1}^n$ are constructed via approximation. To that end, let  $\varphi_0 := \psi$. For all integers $j \geq 0$, let $\psi_j$ be any nonnegative function on $B_R$, constant on some dyadic scale $\delta_j$, such that 
\[ \int_{B_R} | \psi_{j}(x) - \varphi_{j}(x)| dx \leq \frac{1}{N^2} \int_{B_R} \varphi_{j}(x) dx \]
for some large $N$ to be specified. Then let $\varphi_{j+1}(x) := \varphi_j(x) \chi_{\psi_j(x) \leq \frac{N-1}{N} \varphi_j(x)}$. 

The inequality $\psi_j(x) \leq (N-1) \varphi_j(x)/N$ holding for a particular $x$ implies that $\varphi_j(x) \leq N (\varphi_j(x) - \psi_j(x)) \leq N |\varphi_j(x) - \psi_j(x)|$. It follows that when $\varphi_{j+1}(x)$ is not simply zero, $\varphi_{j+1}(x) \leq N |\psi_j(x) - \varphi_j(x)|$. Consequently
\[ \int_{B_R} \varphi_{j+1} (x) dx \leq N \int_{B_R} |\psi_{j}(x) - \varphi_j(x)| dx \leq \frac{1}{N} \int_{B_R} \varphi_j(x) dx,   \]
so by induction on $j$ it follows that
\begin{equation}  \int_{B_R} \varphi_{j} (x) dx \leq N^{-j} \int_{B_R} \psi(x) dx \ \text{ for all } \ j \geq 0. \label{geodecay} \end{equation}
Similarly, the triangle inequality dictates that
\begin{equation} \int_{B_R} \psi_j \leq \int_{B_R} \varphi_j + \int_{B_R} |\varphi_j - \psi_j| \leq \frac{N^2+1}{N^2} \int_{B_R} \varphi_j \leq \frac{N^2+1}{N^{j+2}} \int_{B_R} \psi \label{psiub} \end{equation}
for every $j \geq 0$.

 The functions $\varphi_j$ are pointwise nonincreasing as functions of $j$ and in fact for each $x$, the sequence $\{\varphi_j(x)\}_{j=0}^{\infty}$ must either be constant or must be zero beyond some finite value of $j$. The Lebesgue Dominated Convergence Theorem (which applies because $\varphi_j(x) \leq \psi(x)$ for all $x \in B_R$ and because $\lim_{j \rightarrow \infty} \varphi_j(x)$ exists for every $x \in B_R$) implies that $\varphi_j(x) \rightarrow 0$ for almost every $x \in B_R$, so
  almost every $x \in B_R$ admits some finite minimal index $j_0 \geq 0$ for which $\varphi_{j_0} (x) = 0$. For each index $j \geq 0$, let $E_j$ be the set of those points $x \in B_R$ such that $j$ is the minimal index for which $\varphi_j(x) = 0$. Let $E_{\infty}$ be the collection of those points $x \in B_R$ not belonging to any $E_j$ for a finite value of $j$. As defined, the sets $E_\infty,E_0,E_1,\ldots$ are pairwise disjoint and their union is $B_R$.
  
  Now for each finite $j$, let $(\omega_{1}^x)_j,\ldots,(\omega_n^x)_j$ be the piecewise-constant vector fields obtained by applying Lemma \ref{guthlemma} to locally constant function $\psi_j$ (which is possible because the Lemma has already been established as true for such functions). Let $\{(\omega_1^x)_{\infty},\ldots,(\omega_n^x)_{\infty}\}$ be the standard basis (i.e., the vector fields are constant as functions of $x$ for each $i=1,\ldots,n$). The lemma will be shown to hold for the vector fields
\[ \omega_{i}^x := \chi_{E_{\infty}}(x) (\omega_i^x)_{\infty} + \chi_{E_0} (x) (\omega_i^x)_0 + \sum_{j=1}^\infty \chi_{E_j}(x) (\omega_i^x)_{j-1}, \qquad i = 1,\ldots,n \]
when $\widetilde \psi(x) := \psi(x) \chi_{B_R \setminus E_\infty}(x)$.
Every $x \in B_R$ belongs to exactly one of the sets $E_\infty, E_0, E_1,\ldots$, so the condition $|\det \{ \omega_i^x\}_{i=1}^n | = 1$ is satisfied at every $x \in B_R$ because $|\det \{ (\omega_i^x)_j\}_{i=1}^n | = 1$ for every $x \in B_R$ and every $j$.
Now substitute this definition of $\omega_i^x$ into \eqref{variety} and expand the sum. Because $\widetilde \psi$ vanishes on the null set $E_\infty$ by definition and because because $\psi$ vanishes at every point of $E_0$, it follows that
\begin{align*}
\int_{\Sigma_\pi}  & [ \widetilde \psi(x)]^{\frac{n-k}{n}}   || d \pi (x)||_{\omega^x} d \sigma(x) 
= \sum_{j=1}^\infty \int_{\Sigma_{\pi} \cap E_j} [ \psi(x)]^{\frac{n-k}{n}} || d \pi (x)||_{(\omega^x)_{j-1}} d \sigma(x)
\end{align*}
for every $k$ with $1 \leq k < n$ and every polynomial map $\pi : \R^n \rightarrow \R^k$.
For any $x \in E_j$ with $j \geq 1$, $\varphi_{j-1}(x) = \psi(x) \neq 0$; consequently the definition of $\varphi_j$ combined with the knowledge that $\varphi_j(x) = 0$ implies that
$\psi_{j-1}(x) > (N-1) \varphi_{j-1}(x) / N = (N-1) \psi(x) / N$.
In other words $\psi(x) \leq N \psi_{j-1}(x) / (N-1)$ for every $x \in E_j$, so
\begin{align*}
\int_{\Sigma_\pi  \cap E_j}[ \psi(x)]^{\frac{n-k}{n}}  &  || d \pi (x)||_{(\omega^x)_{j-1}} d \sigma(x) \\ & \leq \int_{\Sigma_\pi  \cap E_j} \left[  \frac{N}{N-1} \psi_{j-1}(x) \right]^{\frac{n-k}{n}}   || d \pi (x)||_{(\omega^x)_{j-1}} d \sigma(x) \\
& \leq C_n (\deg \pi) \left[ \frac{N}{N-1} \right]^\frac{n-k}{n} \left[ \int_{B_R} \psi_{j-1} \right]^{\frac{n-k}{n}}
\end{align*}
for all $j \geq 1$ by virtue of the fact that $\psi_{j-1}$ is constant on some dyadic grid and therefore \eqref{variety} is known to hold with $\widetilde \psi_{j-1} = \psi_{j-1}$. Summing over $j$ and using the upper bound \eqref{psiub} gives
\begin{align*} \int_{\Sigma_\pi}  & [ \psi(x)]^{\frac{n-k}{n}}   || d \pi (x)||_{\omega^x} d \sigma(x)  \\ & \leq C_n (\deg \pi) \left[ \int_{B_R} \psi \right]^{\frac{n-k}{n}} \sum_{j=0}^\infty \left[ \frac{N}{N-1} \right]^{\frac{n-k}{n}} \left[ \frac{N^2+1}{N^{j+2}} \right]^{\frac{n-k}{n}}. 
\end{align*}
Since $k < n$, choosing $N$ suitably large depending on $n$ and $k$, one may assume
\[ \sum_{j=0}^\infty  \left[ \frac{N^2+1}{N^{j+1}(N-1)} \right]^\frac{n-k}{n} \leq 2, \]
which means that \eqref{variety} holds for any integrable $\psi$ with a $C_n$ no more than double the constant which holds for functions constant on dyadic scales. 
\end{proof}

\begin{proof}[Proof of Proposition \ref{visibility}]
This lemma is a consequence of Guth's visibility lemma \cite{guth2010}*{Lemma 6.6} (or, alternatively, Theorem 3 of Carbery and Valdimarsson \cite{cv2013}) and Zhang's intersection estimate \cite{zhang2018}*{Theorem 5.2}.
It is assumed by the proposition that $\psi$ is nonnegative and not identically zero, so it suffices to assume that the integral of $\psi$ on $B_R$ is exactly $(2R)^n$, since if not, one may apply the lemma to the function $(2R)^n (\int_{B_R} \psi)^{-1} \psi(x)$ and then multiply the resulting measures $\mu^x$ by $(2R)^{-1} (\int_{B_R} \psi)^{1/n}$ to recover the full proposition. 

The proof proceeds by first taking $\delta = 1$ and then establishing all other cases by rescaling.
By Guth's Lemma 6.6, given any finitely-supported, nonnegative, integer-valued function $M(Q)$ defined on the lattice of cubes $\Lambda_1$, there exists an algebraic hypersurface $Z$ of degree at most $C_n (\sum_{Q} M(Q))^{1/n}$ for some $C_n$ (not the same as in \eqref{whatisit}) such that
\begin{equation} \overline{\vis} [ Z \cap Q ] \geq M(Q) \text{ for all } Q \in \Lambda_1,\label{guth} \end{equation}
where $\overline{\vis}[Z \cap Q]$ is the quantity called mollified visibility, defined to be the reciprocal of the Euclidean volume of the convex set of vectors $u$ for which $||u|| \leq 1$ and
\begin{equation} \frac{1}{|B(Z,\epsilon)|} \int_{B(Z,\epsilon)} \int_{Z' \cap Q} |u \cdot \widehat n(z') | \dH^{n-1} (z') d Z' \leq 1.\label{guth0} \end{equation}
Here $\widehat{n}(z)$ is the unit normal to $Z'$ at the point $z'$ and the metric structure and measure on the space of algebraic hypersurfaces is the one inherited by identifying each hypersurface $Z'$ of the given degree with the polynomial defining it modulo nonzero scalar multiples.
The technical constraint $||u|| \leq 1$ in Guth's definition of mollified visibility is one which must be properly handled, as on its surface it makes mollified visibility larger than it might otherwise be. One of the principal points of the current proposition is to remove this requirement so that one may work with the unrestricted definition of visibility given by \eqref{visdef}. 

Zhang's approach of adding hyperplanes works here as well. Let $P$ be the union of all hyperplanes having the form
\[\set{(x_1,\ldots,x_n)}{x_i = \frac{1}{2} + j } \] 
for some $i \in \{1,\ldots,n\}$ and $j \in \{-R,\ldots,R-1\}$. This collection $P$ is itself the zero set of a polynomial of degree $2Rn$ and each cube $Q \in \Lambda_1$ intersects $P$ in a union of $n$ orthogonal faces, each with $(n-1)$-dimensional Hausdorff measure equal to $1$. Therefore
\[ \int_{(Z' \cup P) \cap Q} |u \cdot \widehat n(z)| \dH^{n-1}(z') = \int_{Z' \cap Q} |u \cdot \widehat n(z)| \dH^{n-1}(z') + \sum_{i=1}^n |u \cdot e_i| \]
where $e_1,\ldots,e_n$ are standard Euclidean unit vectors. 
Averaging over $Z'$ implies that any $u \in \R^n$ for which
\[  \frac{1}{|B(Z,\epsilon)|} \int_{B(Z,\epsilon)} \int_{(Z' \cup P) \cap Q} |u \cdot \widehat n(z') | \dH^{n-1} (z') d Z' \leq 1 \]
(note that the difference from \eqref{guth0} is that the inner integral is now over $(Z' \cup P) \cap Q$), must satisfy both $||u|| \leq 1$ and \eqref{guth0}. Now for each $Q \in \Lambda_1$, define a measure $\mu_Q$ on the unit sphere $\mathbb S^{n-1}$ by means of the pushforward formula
\begin{equation} \int f d \mu_Q := \frac{1}{|B(Z,\epsilon)|} \int_{B(Z,\epsilon)} \int_{(Z' \cup P) \cap Q} f( \widehat n(z')) \dH^{n-1} (z') d Z'.  \label{push} \end{equation}
(Note that $\mu_Q$ is finite because the $(n-1)$-dimensional Hausdorff measure of an algebraic hypersurface must be finite on $Q$ with a bound depending only on degree. One of many possible proofs of this fact is to use Zhang's inequality (5.5) with $U := [-1,1]^{n(n-1)}$, $Z_1 := Z'$ and $Z_2,\ldots,Z_n$ ranging over all hyperplanes which pass through the center of $Q$ and have normals pointing in standard coordinate directions.) 
The (untruncated) fading zone $F(\mu_Q)$ is automatically contained in the intersection of the unit ball and also in the set of those $u \in \R^n$ satisfying \eqref{guth0}. Therefore 
\[ \vis(\mu_Q) \geq \overline{\vis}[Z \cap Q] \geq M(Q) \text{ and } \int |u \cdot y^*| d \mu_Q(y^*) \geq ||u|| \text{ for all } u \in \R^n. \]
Each variety $Z' \cup P$ appearing on the right-hand side of \eqref{push} has degree at most $C_n (\sum_Q M(Q))^{1/n} + 2Rn$. Now for each $x \in [-R,R)^n$, let $\mu^x$ be the measure $\mu_{Q}$ for the unique $Q \in \Lambda_1$ containing $x$. This gives that $\vis(\mu^x) = \vis(\mu_Q) \geq M(Q)$ for all $x \in Q$ and $\int |u \cdot y^*| d \mu^x(y^*) \geq ||u||$ for all $u \in \R^n$.

The next step is to use Zhang's Theorem 5.2 to establish that when $\Sigma_\pi$ is the smooth zero set of $\pi$,
\begin{align}
 \int_{\Sigma_\pi} & \left[ \int \cdots \int \left|\frac{d\pi(x)}{||d\pi(x)||} \wedge y_1^* \wedge \cdots \wedge y_{n-k}^* \right| d \mu^x(y_1^*) \cdots d \mu^x(y_{n-k}^*) \right] d \mathcal H^{n-1}(x) \nonumber
 \\ & \leq 2^{n(n-1)} \deg \pi \left[ C_n \left(\sum_Q M(Q) \right)^{1/n} + 2Rn \right]^{n-k}. \label{interwhat}
\end{align}
On the left-hand side of \eqref{interwhat}, write each $d \mu^x(y_i^*)$ in terms of \eqref{push} as an integral over varieties. For a given $x \in B_R$, the measure $\mu_Q$ for which $\mu^x = \mu_Q$ has the property that on the right-hand side of \eqref{push}, the integral over $z'$ is restricted in such a way that $x$ and $z'$ both belong to the same unique $Q \in \Lambda_1$. This means that $z' - x \in [-1,1]^n$ for any $x \in B_R$ and any $z'$ in the support of the integral defining $\mu^x$. Thus
 \begin{align*}
 \int_{\Sigma_\pi} & \left[ \int \cdots \int \left|\frac{d\pi(x)}{||d\pi(x)||} \wedge y_1^* \wedge \cdots \wedge y_{n-k}^* \right| d \mu^x(y_1^*) \cdots d \mu^x(y_{n-k}^*) \right] d \mathcal H^{n-k}(x)  \\
 & \leq \frac{1}{|B(Z,\epsilon)|^{n-k}} \int_{B(Z,\epsilon)} \cdots \int_{B(Z,\epsilon)} I(Z'_1,\ldots,Z'_{n-k}) d Z'_1 \cdots d Z'_{n-k} 
 \end{align*}
 with $I(Z'_1,\ldots,Z'_{n-k})$ equal to
 \begin{align*}
 \int_{\Sigma_\pi} &  \int_{Z'_1} \cdots \int_{Z'_{n-k}} \chi_U(z'_1-x,\ldots,z'_{n-k}-x) \\ & \cdot \left| \frac{d \pi(x)}{|| d \pi(x)||} \wedge \widehat{n}(z'_1) \wedge \cdots \wedge \widehat{n}(z'_{n-k}) \right| \dH^{n-1}(z'_{n-k}) \cdots \dH^{n-1}(z'_1) \dH^{n-k}(x) 
  \end{align*}
 and $U := [-1,1]^{n(n-k)}$. Now Proposition \ref{formprop} combined with the observation that $\widehat{n}(z'_i) = (T_{z'_i} Z')^\perp$ in Zhang's notation
 allows one to apply his Theorem 5.2 to conclude that
 \begin{equation} I(Z'_1,\ldots,Z'_{n-k}) \leq 2^{n(n-k)} \deg \pi \left[ C_n \left(\sum_Q M(Q) \right)^{1/n} + 2Rn \right]^{n-k}. \label{tonash} \end{equation}
 Since the left-hand side of \eqref{interwhat} is simply an average over $Z'_1,\ldots,Z'_{n-k}$, the full inequality \eqref{interwhat} follows immediately (with constant $2^{n(n-1)}$ because $k \geq 1$). It is also worth noting that it is possible to slightly relax the condition that $\pi$ is polynomial: if the functions $\pi^1,\ldots,\pi^k$ are Nash functions (see \cite{kollar2017} for an accessible introduction), then \eqref{tonash} holds with $\deg \pi$ replaced by the product of the complexities of $\pi^1,\ldots,\pi^k$ thanks to the B\'{e}zout Theorem for Nash functions \cite{ramanakoraisina89}. As this is the only place in the proof where B\'{e}zout's Theorem is needed (aside from the earlier proof of the case $k=n$, which can also be replaced by the Nash analogue), a Nash version of Theorem \ref{mainthm} must also hold once the remaining portions of the main proof are complete: finiteness of the supremum \eqref{testing2} implies boundedness of \eqref{bddness} with $||T|| \leq  C' [[T]] \prod_{j=1}^m (\prod_{i=1}^{k_j} c(\pi_j^i))^{q_j/p_j}$, where $c(\pi_j^i)$ is complexity. The analogous version of Theorem \ref{restrictstrongtype} (which is stated in Section \ref{mainpossec}) holds as well.
 
To conclude the case $\delta = 1$, observe that \eqref{interwhat} directly implies
\eqref{whatisit} when one takes $M(Q)$ to be the smallest integer greater than or equal to $1 + \psi$ on $Q$. Because $1 + \psi|_Q \leq M(Q) \leq 2 + \psi|_Q$ for each $Q$,
 \[ \sum_Q M(Q) \leq \sum_Q (2 + \psi) = 2 (2R)^n + \sum_Q \psi|_Q = 3 (2R)^n = 3 \int_{B_R} \psi, \]
 so
 \[ C_n \left( \sum_{Q} M(Q) \right)^{1/n} + 2Rn \leq (3^{1/n} C_n + n) \left( \int_{B_R} \psi \right)^{\frac{1}{n}}. \]
 The lower bound \eqref{bignorm0} follows simply because it has already been shown that left-hand side is greater than $||u||$, and $(\int_{B_R} \psi)^{1/n} / (2R) = 1$ with the current normalization of $\psi$.

The proposition is now fully proved when $\delta = 1$. At finer scales $\delta$, apply the scale $1$ version of the proposition to the function $\psi(\delta x)$ on the box $[-R \delta^{-1}, R \delta^{-1})^n$. This yields measures $ \tilde \mu^x$ for $x \in [-R \delta^{-1}, R \delta^{-1})^n$ such that
\[ \vis(\tilde \mu^{\delta^{-1} x}) \geq \psi(x) \text{ and } \int |u \cdot y^*| d \tilde \mu^{\delta^{-1} x} (y^*) \geq \frac{\left(\int_{[-R \delta^{-1} , R \delta^{-1})^n} \psi(\delta x) dx\right)^{\frac{1}{n}}}{2 R \delta^{-1}} ||u|| \]
for all $x \in B_R$ (and a change of variables shows that the coefficient of $||u||$ above is exactly $(\int_{B_R} \psi)^{1/n}/(2R)$). Now consider the quantity
\begin{align*} & \int_{\Sigma_\pi}  \left[ \int \left| \frac{d \pi(x)}{||d \pi(x)||} \wedge y_1^* \wedge \cdots \wedge y_{n-k}^* \right| d \tilde \mu^{\delta^{-1} x}(y_1^*) \cdots d \tilde \mu^{\delta^{-1} x}(y_{n-k}^*) \right] \dH^{n-k}(x),
\end{align*}
where the integral sign inside the brackets is shorthand for the $(n-k)$-fold iterated integral over $y_1^*,\ldots,y_{n-k}^*$.
After rescaling $x \mapsto \delta x$, this must equal
\begin{align*}
& \delta^{n-k} \int_{\Sigma_{\pi_\delta}} \left[ \int \left|\frac{d \pi_\delta(x)}{||d \pi_\delta(x)||} \wedge y_1^* \wedge \cdots \wedge y_{n-k}^* \right| d \tilde \mu^{ x}(y_1^*) \cdots d \tilde \mu^{x}(y_{n-k}^*) \right] \dH^{n-k}(x)
\end{align*}
where $\pi_\delta(x)$ is the polynomial $\pi(\delta x)$ (note, e.g., $\delta^{-1} \Sigma_{\pi} = \Sigma_{\pi_\delta}$). By \eqref{whatisit}, this does not exceed
\begin{align*}
 \delta^{n-k} C_{n} \deg \pi_\delta & \left[ \int_{[-R \delta^{-1}, R \delta^{-1})^n}   \psi(\delta x) dx \right]^\frac{n-k}{n} = C_n \deg \pi \left[ \int_{B_R} \psi(x) dx \right]^\frac{n-k}{n}. 
 \end{align*}
Therefore the proposition must hold at scale $\delta$ by choosing $\mu^x := \tilde \mu^{\delta^{-1}x}$ from the unit scale construction.
\end{proof}

\begin{proof}[Proof of Proposition \ref{basis}]
Consider the quantity
\[ ||u||_\mu := \int |u \cdot y^*| d \mu(y^*). \]
The fading zone $F(\mu)$ is precisely the set of those vectors $u$ for which $||u||_\mu \leq 1$. Moreover, the assumption that $||u||_\mu \geq \psi_0 ||u||$ and the finiteness of the measure $\mu$ guarantee that $F(\mu)$ is compact and contains an open ball centered at the origin. Let $(u_1,\ldots,u_n)$ be any tuple in $(F(\mu))^n$ which maximizes
\[ |\det (u_1,\ldots,u_n)|. \]
Because $F(\mu)$ contains a neighborhood of the origin, the maximum value attained is not zero. Let $u_1^*, \ldots, u_n^*$ be such that $u_i^* \cdot u_j = \delta_{ij}$. It follows that
\begin{equation} u_i^* \cdot v = (-1)^{i-1} \frac{\det (v,u_1,\ldots,\widehat{u_i},\cdots,u_n)}{\det (u_1,\ldots,u_n)} \label{cramer} \end{equation}
(here $\widehat{\cdot}$ denotes omission) for all $v$
because both sides of \eqref{cramer} are linear functions of $v$ which equal one when $v = u_i$ and vanish when $v = u_{j}$ for $j \neq i$. In particular, if $v$ belongs to the unit ball of $||\cdot||_\mu$, then the ratio of determinants on the right-hand side of \eqref{cramer} has magnitude at most $1$. 
After scaling, this implies that
\begin{equation} \max_{i=1,\ldots,n} | u_i^*\cdot v| \leq ||v||_\mu \label{flower} \end{equation}
for all $v \in \R^n$.
Since
\begin{equation} v = \sum_{i=1}^n (u_i^* \cdot v) u_i \label{fcoeff} \end{equation}
(again, because both sides are equal when $v = u_j$ for any $j=1,\ldots,n$), by the triangle inequality, 
\begin{equation} ||v||_\mu \leq \sum_{i=1}^n |u_i^* \cdot v|. \label{fupper} \end{equation}

Now consider the set of vectors $v$ such that $v = \sum_{i=1}^n \theta_i u_i$ for $\sum_{i=1}^n |\theta_i| \leq 1$. By \eqref{fcoeff} and \eqref{fupper}, $||v||_\mu \leq 1$, meaning every such $v$ belongs to the fading zone $F(\mu)$. This set of $v$ is a polytope in $\R^n$ of volume $\frac{2^n}{n!} |\det (u_1,\ldots,u_n)|$, so 
\[ \frac{2^n}{n!} |\det (u_1,\ldots,u_n)| \leq |F(\mu)|. \]
Likewise, \eqref{fcoeff} and \eqref{flower} imply that the fading zone is contained in the set of $v$'s expressible as $\sum_{i=1}^n \theta_i v_i$ with $\max_i |\theta_i| \leq 1$, which is a polytope of volume $2^n |\det (u_1,\ldots,u_n)|$. Therefore
\[ \frac{2^n}{n!} |\det (u_1,\ldots,u_n)| \leq |F(\mu)| \leq 2^n |\det (u_1,\ldots,u_n)|. \]
 Taking reciprocals gives
\begin{equation} 2^{-n} |\det (u_1^*,\ldots,u_n^*)| \leq \vis(\mu) \leq n! 2^{-n} |\det (u_1^*,\ldots,u_n^*)|. \label{itovis} \end{equation}

Let $\omega_i := u_i / |\det \{u_{i'}\}_{i'=1}^n |^{1/n}$ and $\omega_i^* := u_i^* / |\det \{u_{i'}^*\}_{i'=1}^n |^{1/n}$ for each $i=1,\ldots,n$.
Now $|\det \{\omega_i^*\}_{i=1}^n | = |\det \{\omega_i\}_{i=1}^n | = 1$; \eqref{flower} and \eqref{fupper} imply
\[ |\det (u_1^*,\ldots,u_n^*)|^{1/n} \max_{i=1,\ldots,n} |\omega_i^* \cdot v| \leq ||v||_\mu \leq |\det(u_1^*,\ldots,u_n^*)|^{1/n} \sum_{i=1}^n |\omega_i^* \cdot v|. \]
Using \eqref{itovis} to estimate $|\det (u_1^*,\ldots,u_n^*)|^{1/n}$ gives
\begin{equation} \frac{2}{(n!)^{1/n}} (\vis{\mu})^\frac{1}{n} \max_{i=1,\ldots,n} |\omega_i^* \cdot v| \leq \int |v \cdot y^*| d \mu(y^*) \leq 2 (\vis{\mu})^\frac{1}{n} \sum_{n=1}^n |\omega_i^* \cdot v| \label{mainconvex} \end{equation}
for every $v \in \R^n$.

The proof of \eqref{midineq} is by induction on $n - k$. Regarding $A^* \wedge y_1^* \wedge \cdots \wedge y_{n-k}^*$ as a linear functional acting on $y_{n-k}^*$ gives the existence of $v \in \R^n$ depending on $A^*$ and $y_1^*,\ldots,y_{n-k-1}^*$ such that $A^* \wedge y_1^* \wedge \cdots \wedge y_{n-k}^* = v \cdot y_{n-k}^*$; applying \eqref{mainconvex} to this particular $v$ gives
\begin{align*}
\frac{2}{n} (\vis \mu)^{\frac{1}{n}} \max_{i=1,\ldots,n} &  |A^* \wedge y_1^* \wedge \cdots \wedge y_{n-k-1}^* \wedge \omega_i^*| \\ & \leq \int |A^* \wedge y_1^* \wedge \cdots \wedge y_{n-k}^*| d \mu(y_{n-k}^*).  
\end{align*}
Integrating over the remaining $y_1^*,\ldots,y_{n-k-1}^*$ gives
\begin{align*}
\frac{2}{n} & (\vis \mu)^{\frac{1}{n}} \max_{i=1,\ldots,n}   \int \cdots \int |A^* \wedge y_1^* \wedge \cdots \wedge y_{n-k-1}^* \wedge \omega_i^*| d \mu(y_1^*) \cdots d \mu(y_{n-k-1}^*) \\ & \leq \int \cdots \int |A^* \wedge y_1^* \wedge \cdots \wedge y_{n-k}^*| d \mu(y_1^*) \cdots d \mu(y_{n-k}^*). 
\end{align*}
Now the induction hypothesis (applied to the $(k+1)$-linear functional $A^* \wedge \omega_i^*$) gives the inequalities
\begin{align*}
\int \cdots \int & |A^* \wedge y_1^* \wedge \cdots \wedge y_{n-k-1}^* \wedge \omega_i^*| d \mu(y_1^*) \cdots d \mu(y_{n-k-1}^*) \\
& \geq \left( \frac{2}{n} (\vis \mu)^{\frac{1}{n}} \right)^{n-k-1} \max_{i_1,\ldots,i_{n-k-1}} |A^* \wedge \omega_{i_1}^* \wedge \cdots \wedge \omega_{i_{n-k-1}}^* \wedge \omega_i^*|,
\end{align*}
which ultimately implies that
\begin{align*}
\left( \frac{2}{n} \left( \vis \mu \right)^{\frac{1}{n}} \right)^{n-k} & \max_{i_1,\ldots,i_{n-k} } | A^* \wedge \omega_{i_1}^* \wedge \cdots \wedge \omega_{i_{n-k}}^*| \\
& \leq \int |A^* \wedge y_1^* \wedge \cdots \wedge y_{n-k}^*| d \mu(y_1^*) \cdots d \mu(y_{n-k}^*).
\end{align*}
The conclusion of the proposition rests on the observation that $|A^* \wedge \omega_{i_{1}}^* \wedge \cdots \wedge \omega_{i_{n-k}}^*| = |A^*(\omega_{j_{1}},\ldots,\omega_{j_k})|$ where $\{i_1,\ldots,i_{n-k}\} \cup \{j_1,\ldots,j_k\} = \{1,\ldots,n\}$, which one can easily see using the fact that $\omega_i^* \cdot \omega_j = \delta_{ij}$ and writing $A^*$ in terms of the basis $\omega_{i_1}^* \wedge \cdots \wedge \omega_{i_k}^*$ for all possible $i_1 < i_2 < \cdots < i_k$.
\end{proof}

\section{Generalized Brascamp-Lieb on varieties}
\label{possec}
\subsection{Statement and proof of Theorem \ref{main1}}
\label{proofsec1}
The portion of Theorem \ref{mainthm} dealing with the sufficiency of the testing condition \eqref{testing} follows from Theorem \ref{main1} below, which is itself a rather direct consequence of Lemma \ref{guthlemma}.
\begin{theorem}
For any positive integer $m$, suppose that for each $j = 1,\ldots,m$, $(\Omega_j,\pi_j,\Sigma_j)$ is a smooth incidence relation on $\R^{n} \times \R^{n_j}$ with codimension $k_j$ such that $\pi_j(x,y_j)$ is a polynomial map as a function of $x$ with bounded degree as $y_j$ varies.  Fix exponents $r_1,\ldots,r_m \geq 0$ satisfying $k_1 r_1 + \cdots + k_m r_m = n$, and for each nonnegative integrable Borel function $f_j$ on $\R^{n_j}$, $j=1,\ldots,m$, let \label{main1}
\begin{equation}
\begin{split}  Q  (f_1,& \ldots,f_m)(x) := \\ & \mathop{\inf_{\{\omega_i\}_{i=1}^n}}_{|\det \{\omega_i\}_{i=1}| = 1} \prod_{j=1}^m \left( \int_{\li{x}_j} f_j(y_j) ||d_x \pi_j (x,y_j)||_\omega d \sigma_j (y_j) \right)^{r_j}. \end{split}
\label{multilinear} \end{equation}
There exists a constant $C_n$ depending only on $n$ such that 
\begin{equation}
\int_{[-R,R)^n} Q(f_1,\ldots,f_m)(x) dx \leq  \prod_{j=1}^m \left[ C_n (\deg \pi_j) \int_{\R^{n_j}} f_j \right]^{r_j}
\end{equation}
for any positive integer $R$, where $\deg \pi_j := \max_{y_j} \deg \pi_j(\cdot,y_j)$.
\end{theorem}

This result should be compared to Zhang's variety version of Brascamp-Lieb \cite{zhang2018}*{Theorem 8.1}. An interesting feature here is that there is in some sense no need to introduce the Brascamp-Lieb machinery at this stage because Lemma \ref{guthlemma} is already powerful enough not only to reproduce the Brascamp-Lieb inequalities, but to yield a strictly richer family of inequalities (or, to view it another way, to yield multilinear inequalities where the weight factor involved is, in some favorable situations, strictly larger than the power of the Brascamp-Lieb constant that would otherwise be found there). For the moment, all the associated subtlety of this problem is encapsulated in the quantity $Q$, and the interesting geometric question which follows after the proof of Theorem \ref{main1} is one of establishing various lower bounds for $Q$.

\begin{proof}[Proof of Theorem \ref{main1}.]
Let $\psi$ be any nonnegative Borel function on $B_R$. Let $\{\omega^x_i\}_{i=1}^n$ and $\widetilde \psi$ be the promised vector fields and function on $B_R$, respectively, from Lemma \ref{guthlemma}. Because $|\det (\omega_1^x,\ldots,\omega_n^x)| = 1$ for every $x \in B_R$, $Q$ is bounded above by the quantity obtained by striking the infimum in \eqref{multilinear} and replacing $\omega$ by $\omega^x$. As a consequence, fixing $r := r_1 + \cdots + r_m$ gives
\begin{align}
\int_{B_R} & \left[ Q (f_1,\ldots,f_m) (x) \right]^{\frac{1}{r}} \left[ \psi(x) \right]^{1 - \frac{1}{r}} dx \nonumber \\
& \leq \int_{B_R}  \left[ \widetilde \psi(x) \right]^{\frac{r-1}{r}} \prod_{j=1}^m \left( \int_{\li{x}_j} f_j(y_j) ||d_x \pi_j(x,y_j) ||_{\omega^x} d \sigma_j(y_j) \right)^{\frac{r_j}{r}} dx. \label{mainstep0}
\end{align}
The appearance of $\widetilde \psi$ on the right-hand side of \eqref{mainstep0} follows simply because $\psi$ and $\widetilde \psi$ are equal almost everywhere with respect to $n$-dimensional Lebesgue measure.  By virtue of the identity \[\sum_{i=1}^m \frac{n-k_j}{n} \frac{r_j}{r} = 1 - \frac{1}{r}, \] one may pull the factor $(\widetilde \psi)^{(r-1)/r}$ into the product over $j$ and apply H\"{o}lder's inequality to conclude
\begin{align}
\int_{B_R}  & \left[ \widetilde \psi(x) \right]^{\frac{r-1}{r}}  \prod_{j=1}^m \left( \int_{\li{x}_j} f_j (y_j)||d_x \pi_j(x,y_j) ||_{\omega^x} d \sigma_j(y_j) \right)^{\frac{r_j}{r}} dx \nonumber \\
& = \int_{B_R} \prod_{j=1}^m \left( [ \widetilde \psi(x)]^{\frac{n-k_j}{n}} \int_{\li{x}_j} f_j(y_j) ||d_x \pi_j(x,y_j) ||_{\omega^x} d \sigma_j(y_j) \right)^{\frac{r_j}{r}} dx \nonumber \\
& \leq \prod_{j=1}^m \left( \int_{B_R}  [ \widetilde \psi(x)]^{\frac{n-k_j}{n}} \int_{\li{x}_j} f_j(y_j) ||d_x \pi_j(x,y_j) ||_{\omega^x} d \sigma_j(y_j) \, dx \right)^{\frac{r_j}{r}}. \label{afterholder} 
\end{align}
To estimate \eqref{afterholder} using the inequality \eqref{variety} given by Lemma \ref{guthlemma}, one needs an auxiliary Fubini-type result which guarantees that the measure $d \sigma_j(y_j) \, dx$ on $\Sigma_j$ is equal to $d \sigma_j(x) dy_j$ (where in the first case coarea measure is on slices $\li{x}_j$ and in the second case is on slices $\ri{y_j}_j$). This a consequence of the identity \eqref{mainfubini} proved in Section \ref{appendix}. Using this fact gives
\begin{align}
 \int_{B_R}  [ \widetilde \psi(x)]^{\frac{n-k_j}{n}} &  \int_{\li{x}_j} f_j (y_j) ||d_x \pi_j(x,y_j) ||_{\omega^x} d \sigma_j(y_j) \, dx = \nonumber \\
 &  \int_{\R^{n_j}}  f_j(y_j) \int_{\ri{y_j}_j \cap B_R}  [ \widetilde \psi(x)]^{\frac{n-k_j}{n}} ||d_x \pi_j(x,y_j) ||_{\omega^x} d \sigma_j(x) \, dy_j. \label{mainstep1}
\end{align}
By Lemma \ref{guthlemma} and Fubini's Theorem, the quantity \eqref{mainstep1} is no greater than
\[ C_n (\deg \pi_j)  \left[ \int_{B_R} \psi \right]^{\frac{n-k_j}{n}} ||f_j||_{L^1(\R^{n_j})}. \]
Thus
\begin{align*}
 \int_{B_R}   \left[ Q (f_1,\ldots,f_m) (x) \right]^{\frac{1}{r}} & \left[ \psi(x) \right]^{1 - \frac{1}{r}} dx \\
& \leq \prod_{j=1}^m \left[ C_n (\deg \pi_j) \left[ \int_{B_R} \psi \right]^{\frac{n-k_j}{n}} ||f_j||_{L^1(\R^{n_j})} \right]^{\frac{r_j}{r}}  \\
& = \left[ \int_{B_R} \psi \right]^{1 - \frac{1}{r}} \prod_{j=1}^m ( C_n \deg \pi_j ||f_j||_{L^1(\R^{n_j})})^{\frac{r_j}{r}}.
\end{align*}
When $r = 1$, the inequality just proved has no dependence on $\psi$, so \eqref{multilinear} is immediate in this case. Otherwise $r$ must be strictly greater than $1$, because $r_1 + \cdots + r_k \geq r_1(k_1/n) + \cdots + r_m (k_m/n) = 1$ by virtue of the fact that $k_i/n \leq 1$ for each $i$.  In this case, let $\psi := |g|^{r'}$ for any $g \in L^{r'}(B_R)$, where $r'$ is exponent dual to $r$ in the H\"{o}lder sense. Then 
\begin{align*}
 \int_{B_R}  & \left[ Q (f_1,\ldots,f_m) (x) \right]^{\frac{1}{r}} |g(x)| dx \\
& \leq \prod_{j=1}^m \left[ C_n (\deg \pi_j) \left[ \int_{B_R} \psi \right]^{\frac{n-k_j}{n}} ||f_j||_{L^1(\R^{n_j})} \right]^{\frac{r_j}{r}}  \\
& = ||g||_{L^{r'}(B_R)} \prod_{j=1}^m ( C_n \deg \pi_j ||f_j||_{L^1(\R^{n_j})})^{\frac{r_j}{r}}.
\end{align*}
By duality, then, it must be the case that $(Q(f_1,\ldots,f_m))^{1/r}$ belongs to $L^r$ and
\[ \int_{B_R} Q(f_1,\ldots,f_m)(x) dx \leq  \prod_{j=1}^m \left( C_n \deg \pi_j ||f_j||_{L^1(\R^{n_j})} \right)^{r_j}\]
as desired.
\end{proof}

\subsection{Proof of Theorem \ref{mainthm}: Sufficiency of the testing condition}
\label{mainpossec}

This section contains the proof of Theorem \ref{mainthm}, which is a slight generalization of Theorem \ref{mainthm}. The inclusion of the parameter $s$ allows one to deduce local inequalities for products of Radon-Brascamp-Lieb transforms; unlike their $s = 0$ counterparts, the $s > 0$ cases are not expected to be sharp, but are included for their natural utility nevertheless. Theorem \ref{restrictstrongtype} also includes provisions for establishing restricted strong-type inequalities; this will be useful because several examples in Section \ref{examplesec} will showcase settings in which restricted strong-type inequalities are the best possible endpoint inequalities.
\begin{theorem}
For each $j=1,\ldots,m$, let $(\Omega_j,\pi_j,\Sigma_j)$ be a smooth incidence relation on $\R^n \times \R^{n_j}$ with codimension $k_j$ and let $w_j$ be a continuous, nonnegative function on $\Sigma_j$.  Suppose $p_1,\ldots,p_m \in [1,\infty)$ and $q_1,\ldots,q_m,s \in [0,\infty)$ satisfy the scaling condition \label{restrictstrongtype}
\begin{equation} n-s =  \sum_{i=1}^m \frac{k_j q_j}{p_j}. \label{genscaling} \end{equation} 
 Let $J_0$, $J_1$, and $J_2$ be pairwise disjoint subsets of $\{1,\ldots,m\}$ whose union is $\{1,\ldots,m\}$ and let $J_1 \cup J_2$ be exactly the set of those indices $j$ for which $p_j > 1$. Let $[[T]]$ be the supremum of
 \begin{equation}
 \begin{split}
 \left[ \sum_{i=1}^n ||\omega_i|| \right]^{-s} & \!  \! \! \! \prod_{j \in J_0}  \sup_{y_j \in \li{x}_j} \frac{|w_j(x,y_j)|}{||d_x \pi_j(x,y_j)||_\omega} \prod_{j \in J_1}   \! \left[ \int_{\li{x}_j} \! \! \frac{|w_j(x,y_j)|^{p'_j} d \sigma_j(y_j) }{||d_x \pi_j(x,y_j)||_{\omega}^{p'_j-1}}  \right]^{\frac{q_j}{p'_j}}  \\
\cdot \prod_{j \in J_2}   \sup_{\epsilon > 0}  & \left[ \epsilon^{1-p_j'} \int_{\li{x}_j} \chi_{||d_x \pi_j(x,y_j)||_\omega < \epsilon |w_j(x,y_j)|} |w_j(x,y_j)| d \sigma_j(y_j) \right]^{\frac{q_j}{p_j'}}
\end{split}\label{testing3}
 \end{equation}
 over all $x \in \R^n$ and all $\{\omega_i\}_{i=1}^n$ with $|\det \{\omega_i\}_{i=1}^n| = 1$. If $[[T]] < \infty$ 
 and each $\pi_j(x,y_j)$ is a polynomial in $x$ with degree bounded as a function of $y_j$, then for each $R > 0$,
 \begin{equation}   \int_{B_R} \prod_{j=1}^m \left|T_j f_j(x) \right|^{q_j} dx   \leq C [[T]] R^s \prod_{j=1}^m  (\deg \pi_j)^{\frac{q_j}{p_j}} ||f_j ||_{L^{p_j}(\R^{n_j})}^{q_j}  \label{restrictresult} \end{equation}
  holds for Borel measurable $f_1,\ldots,f_m$ under the assumption that $f_j$ is a characteristic function for each $j \in J_2$, with the constant $C$ depending only on $n, s,$ and $n_j,k_j,p_j,q_j$. As before, $\deg \pi_j := \sup_{y_j} \deg \pi^1_j(\cdot,y_j) \cdots \deg \pi^k_j(\cdot,y_j)$ and $B_R := [-R,R)^n$.
\end{theorem}

The proof of Theorem \ref{restrictstrongtype} is accomplished by establishing a number of different lower bounds for quantities like those appearing in the definition \eqref{multilinear} of $Q(f_1,\ldots,f_m)$. These inequalities are recorded in the proposition below. 
\begin{proposition}
Suppose $f$ is a nonnegative Borel function on $\R^{n'}$ and that $(\Omega,\pi,\Sigma)$ is a smooth incidence relation on $\R^n \times \R^{n'}$ of codimension $k$. Suppose also that $w$ is any continuous, nonnegative function on $\Sigma$. Let $\{\omega_i\}_{i=1}^n$ be any tuple of vectors with $|\det \{\omega_i\}_{i=1}^n| = 1$.
\begin{itemize}
\item If $p \in (1,\infty)$, then
\begin{equation}
\begin{split}
 \int_{\li{x}}   f(y)  w (x,y)  d \sigma(y)  \leq & \left[  \int_{\li{x}} \! \! |f(y)|^{p} ||d_x \pi(x,y)||_\omega d \sigma(y)  \right]^{\frac{1}{p}} \\
  &  \cdot \left[ \int_{\li{x}} \frac{|w(x,y)|^{p'} d \sigma (y)  }{||d_x \pi(x,y)||_{\omega}^{p'-1}} \right]^{\frac{1}{p'}}.
 \end{split} \label{strongineq}
 \end{equation}
 If, in addition, $f$ is a characteristic function, then
 \begin{equation}
 \begin{split}
 \int_{\li{x}}   f(y)   & w (x,y)  d \sigma(y)  \leq 2 \left[  \int_{\li{x}} \! \! f(y) ||d_x \pi(x,y)||_\omega d \sigma(y)  \right]^{\frac{1}{p}} \\
  \cdot \sup_{\epsilon > 0} & \left[ \epsilon^{1-p'} \int_{\li{x}} \chi_{||d_x \pi(x,y)||_\omega < \epsilon w(x,y)} w(x,y) d \sigma(y) \right]^{\frac{1}{p'}}.
\end{split} \label{restrictineq}
 \end{equation}
 \item If $p = 1$, then
 \begin{equation}
\begin{split}
  \int_{\li{x}}   f(y) & w (x,y)  d \sigma(y)  \leq \left[  \int_{\li{x}} \! \! f(y) ||d_x \pi(x,y)||_\omega d \sigma(y)  \right] \\
  &  \cdot \sup_{y \in \li{x}}  \frac{w(x,y)}{||d_x \pi(x,y)||_{\omega}}.
 \end{split} \label{l1ineq}
 \end{equation}
 \item If $\pi (x,y) := \frac{1}{2} \left( |x-y|^2 - R^2 \right) $ on $\R^n \times \R^n$ and $f$ is $(4R)^{-n}$ times the characteristic function of $[-2R,2R]^n$, then there is a constant $C_n$ depending only on $n$ such that
 \begin{equation}
 \chi_{B_R}(x) \leq C_n R \left[ \sum_{i=1}^n ||\omega_i|| \right]^{-1} \left[ \int_{\li{x}} f(y)  || d_x \pi(x,y)||_\omega d \sigma(y)  \right]. \label{sphereavg}
 \end{equation}
\end{itemize}
\end{proposition}
\begin{proof}
To establish \eqref{strongineq}, observe by H\"{o}lder's inequality that
\begin{align}
\int_{\li{x}}  &  f(y) w(x,y) d \sigma(y) =  \int_{\li{x}} f(y) ||d_x \pi(x,y)||_\omega^\frac{1}{p}  \frac{w(x,y) d \sigma(y)}{||d_x \pi(x,y)||_{\omega}^{1/p}} \nonumber \\ 
 \leq &
\left[  \int_{\li{x}}\! \! |f(y)|^{p} ||d_x \pi(x,y)||_\omega d \sigma(y)  \right]^{\frac{1}{p}} \left[ \int_{\li{x}} \!   \frac{ (w(x,y))^{p'} d \sigma(y) }{||d_x \pi(x,y)||_{\omega}^{p'/p}} \right]^{\frac{1}{p'}} \nonumber \\
 = & \left[  \int_{\li{x}} \! \! |f(y)|^{p} ||d_x \pi(x,y)||_\omega d \sigma(y)  \right]^{\frac{1}{p}} \left[ \int_{\li{x}} \!   \frac{|w(x,y)|^{p'} d \sigma (y)  }{||d_x \pi(x,y)||_{\omega}^{p'-1}} \right]^{\frac{1}{p'}}. \nonumber
\end{align}
The inequality \eqref{l1ineq} is even simpler:
\begin{align*}
\int_{\li{x}} & f(y) w (x,y) d \sigma(y) \leq \int_{\li{x}} f(y) ||d_x \pi(x,y)||_\omega \frac{w (x,y)}{||d_x \pi(x,y)||_\omega} d \sigma(y) \\
& \leq    \int_{\li{x}} f(y) ||d_x \pi(x,y)||_\omega d \sigma(y)   \sup_{y \in \li{x}}  \frac{w(x,y)  }{||d_x \pi(x,y)||_{\omega}}.
\end{align*}

It suffices to prove \eqref{restrictineq} under the assumption that the right-hand side is finite and the left-hand side is nonzero. This implies that both sides are finite and nonzero, since if 
\[ \int_{\li{x}} f(y) ||d_x \pi(x,y)||_\omega d \sigma(y) = 0, \]
the fact that $||d_x \pi(x,y)||_\omega > 0$ almost everywhere with respect to $\sigma$ (by definition of $\Sigma$) means that $f = 0$ almost everywhere and thus implies that
\[\int_{\li{x}} f(y) w(x,y) d \sigma(y) = 0 \]
as well, which has already been assumed otherwise. Similarly, if
\[ \int_{\li{x}} f(y) w(x,y) d \sigma(y) = \infty, \]
 let $F$ be the set on which $||d_x \pi(x,y)||_\omega < w(x,y)$. Then
\[ \int_{\li{x} \cap F} w(x,y) d \sigma(y) < \infty \]
by virtue of the fact that the supremum on the right-hand side of \eqref{restrictineq} is, without loss of generality, finite when $\epsilon = 1$.
Thus if $E \subset \li{x}$ is the set of those $y$ for which $f(y) = 1$, then $E$ has infinite measure with respect to $w d \sigma$ and $||d_x\pi(x,y)||_\omega \geq w(x,y)$ on all but a set of finite measure with respect to $w d \sigma$. Consequently
\[ \infty = \int_{\li{x} \cap (E \setminus F)} w(x,y) d\sigma(y) \leq \int_{\li{x} \cap (E \setminus F)} ||d_x\pi(x,y_j)||_\omega d\sigma(y), \]
which forces the right-hand side of \eqref{restrictineq} to be infinite.
Thus without loss of generality, one may also assume that the ratio
\[ A := \frac{1}{\int_{\li{x} \cap E} w(x,y) d \sigma(y)} \int_{\li{x} \cap E}  ||d_x \pi (x,y)||_\omega d \sigma(y) \]
 is finite and nonzero. By
 Chebyshev's inequality, if $E'$ consists of those points $y \in E$ at which $||d_x \pi(x,y)||_\omega \geq 2 A w(x,y)$, then
\begin{align*} \int_{\li{x} \cap E'} w(x,y) d \sigma(y) & \leq \int_{\li{x} \cap E'} \frac{||d_x \pi(x,y)||_\omega}{2A} d \sigma(y) \\ & \leq \frac{1}{2} \int_{\li{x} \cap E} w(x,y) d \sigma(y). \end{align*}
It follows that the measure of $E \setminus E'$ with respect to $w d \sigma$ is at least half of the measure of $E$ itself. As $E \setminus E'$ consists only of points $y$ at which $||d_x \pi(x,y)||_\omega < 2 A w(x,y)$, the measure of $E$ with respect to $w d \sigma$ being bounded above by twice the measure of $E \setminus E'$ gives the integral inequality
\begin{align*}
\int_{\li{x} \cap E} & w(x,y) d \sigma(y)  \leq 2 \int_{\li{x}} \chi_{||d_x \pi(x,y)||_\omega < 2 A w(x,y)} w(x,y) d \sigma(y) \\
& \leq 2 (2 A)^{p'-1} \sup_{\epsilon > 0}  \epsilon^{1-p'} \int_{\li{x}} \chi_{||d_x \pi(x,y)||_\omega < \epsilon w(x,y)} w(x,y) d \sigma(y).
\end{align*}
Substituting the definition of $A$ back into this last line and simplifying gives
\begin{align*}
 \left( \int_{\li{x} \cap E}  w(x,y) d \sigma(y) \right)^{p'}   &  \leq 2^{p'}   \left(\int_{\li{x} \cap E} ||d_x \pi(x,y)|| d \sigma(y) \right)^{p'-1} \\ &  \cdot \sup_{\epsilon > 0}  \epsilon^{1-p'} \int_{\li{x}} \chi_{||d_x \pi(x,y)||_\omega < \epsilon w(x,y)} w(x,y) d \sigma(y).
 \end{align*}
 Raising each side to the power $1/p'$ and recalling the definition of $E$ gives \eqref{restrictineq}.

Finally, consider \eqref{sphereavg}. Suppose that $\pi (x,y) := \frac{1}{2} \left( |x-y|^2 - R^2 \right) $ on $\R^n \times \R^n$ and $f$ is $(4R)^{-n}$ times the characteristic function of $[-2R,2R]^n$. Computation gives that $D_y \pi|_{(x,y)} v = v \cdot (y-x)$ and $D_x \pi |_{(x,y)} v = v \cdot (x-y)$, meaning that for $(x,y) \in \Sigma$, $||d_y \pi  || = ||x-y|| = R$ and $||d_x \pi||_\omega = \sqrt{\sum_{i=1}^n |(x-y)\cdot \omega_i|^2}$. Thus
\begin{align*}
\int_{\li{x}} f(y) & || d_x \pi(x,y)||_\omega d \sigma(y) = \int_{|x-y|=R} f(y) \left( \sum_{i=1}^n |(x-y)\cdot \omega_i|^2 \right)^{\frac{1}{2}} \frac{\dH^{n-1}(y)}{R} \\
& \geq \chi_{B_R}(x)  \int_{{\mathbb S}^{n-1}} \frac{1}{4^n R^n}  \left( \sum_{i=1}^n |R z \cdot \omega_i|^2 \right)^{\frac{1}{2}} \frac{R^{n-1}}{R} \dH^{n-1}(z)
\end{align*}
because $f(y)$ is identically equal to $(4R)^{-n}$ on the sphere $|x-y| = R$ when $x \in B_R$ (and note that the integral over $z$ is obtained simply by the change of variables $z = R (x-y)$, which gives $\dH^{n-1}(y) = R^{n-1} \dH^{n-1}(z)$). Continuing this chain of inequalities leads to the conclusion
\begin{align*}
\int_{\li{x}} f(y) & || d_x \pi(x,y)||_\omega d \sigma(y)  \geq  \frac{\chi_{B_R}(x)}{4^n R}  \int_{{\mathbb S}^{n-1}} \left( \sum_{i=1}^n |z \cdot \omega_i|^2 \right)^{\frac{1}{2}}\dH^{n-1}(z) \\
& \geq \frac{\chi_{B_R}(x)}{4^n R}  \int_{{\mathbb S}^{n-1}} \frac{1}{\sqrt{n}} \sum_{i=1}^n |z \cdot \omega_i| \dH^{n-1}(z) \\
& \geq \frac{\chi_{B_R}(x)}{C_n R} \sum_{i=1}^n ||\omega_i||.
\end{align*}
The last line of this derivation follows by the fact that symmetry and scaling imply
\[ \int_{\mathbb S^{n-1}} |z \cdot \omega_i| \dH^{n-1}(z) = ||\omega_i|| \int_{\mathbb S^{n-1}} |z \cdot e_1| \dH^{n-1}(z), \]
where $e_1$ is the first vector in the standard basis. The constant $C_n$, as the notation suggests, depends only on $n$. This completes the proof of \eqref{sphereavg}.
\end{proof}

\begin{proof}[Proof of Theorem \ref{mainthm}: finiteness of \eqref{testing2} implies \eqref{bddness}.] Assuming that Theorem \ref{restrictstrongtype} is established, the proof of \eqref{bddness} from finiteness of the supremum \eqref{testing2} is almost immediate. Take $J_2$ to be empty and set $s = 0$. The quantity \eqref{testing3} then reduces exactly to \eqref{testing2}. For any fixed $f_1,\ldots,f_m$, the Monotone Convergence Theorem applied in the limit $R \rightarrow \infty$ implies that
 \begin{equation}   \int_{\R^n} \prod_{j=1}^m \left|T_j f_j(x) \right|^{q_j} dx   \leq C [[T]] \prod_{j=1}^m  (\deg \pi_j)^{\frac{q_j}{p_j}} ||f_j ||_{L^{p_j}(\R^{n_j})}^{q_j},  \label{restrictresult2} \end{equation}
 which then implies exactly the promised bound on $||T||$ in \eqref{bddness}.
\end{proof}

\begin{proof}[Proof of Theorem \ref{restrictstrongtype}]
Let $[[T]]_{x,\omega}$ equal \eqref{testing3} (i.e., before the supremum over $x$ or $\omega$ is taken). Let $J_0,J_1,J_2$ be as described in the statement of the theorem, and if $s > 0$, let $f_{m+1}$ be the function $(4R)^{-n} \chi_{[-2R,2R]^n}$ and set $\pi_{m+1}(x,y) := \frac{1}{2} (|x-y|^2 - R^2)$ for the current fixed value of $R$. By taking pointwise products of \eqref{strongineq}--\eqref{sphereavg}, each raised to the appropriate power,
\begin{align}
\chi_{B_R}(x) &  \prod_{j=1}^m \left| T_j f_j(x) \right|^{q_j}  \leq [[T]]_{x,\omega} \prod_{j \in J_0} \left( \int_{\li{x}_j} f_j(y_j) ||d_x \pi_j(x,y_j) ||_\omega d \sigma_j(y_j) 
\right)^{q_j} \nonumber \\
& \cdot \prod_{j \in J_1} \left( \int_{\li{x}_j} |f_j(y_j)|^{p_j} ||d_x \pi_j(x,y_j) ||_\omega d \sigma_j(y_j) \right)^{\frac{q_j}{p_j}} \nonumber \\
& \cdot \prod_{j \in J_2} 2^{q_j} \left( \int_{\li{x}_j} f_j(y_j) ||d_x \pi_j(x,y_j) ||_\omega d \sigma_j(y_j) \right)^{\frac{q_j}{p_j}} \label{justj2} \\
& \cdot \left( C_n R \int_{\li{x}_{m+1}} f_{m+1}(y_{m+1}) ||d_x \pi_{m+1}(x,y_{m+1}) ||_\omega d \sigma_{m+1}(y_{m+1}) \right)^{s} \nonumber
\end{align}
under the assumption that $f_j$ is a characteristic function for each $j \in J_2$.
(If $s = 0$, simply omit the final factor.). Let $r_i := q_i / p_i$ for $1 \leq i \leq m$, and if $s > 0$, let $r_{m+1} := s$. The scaling condition \eqref{genscaling} implies $k_1 r_1 + \cdots + k_m r_m = n$ when $s = 0$ and $k_1 r_1 + \cdots + k_m r_m + r_{m+1} = n$ when $s > 0$. Now the inequality $[[T]]_{x,\omega} \leq [[T]]$ permits the following pointwise estimate of $Q$ from below:
\[ \chi_{B_R}(x) \prod_{j=1}^m \left| T_j f_j(x) \right|^{q_j} \leq 2^{\sum_{j \in J_2} q_j } [[T]] Q(|f_1|^{p_1},\ldots,|f_m|^{p_m}) \]
when $s = 0$ and
\[ \chi_{B_R}(x) \prod_{j=1}^m \left| T_j f_j(x) \right|^{q_j} \leq 2^{\sum_{j \in J_2} q_j } (C_n R)^s [[T]] Q(|f_1|^{p_1},\ldots,|f_m|^{p_m},f_{m+1}) \]
when $s > 0$.  (Note that when $j \in J_2$, $f_j = |f_j|^{p_j}$, so there is no error made by replacing $f_j$ by $|f_j|^{p_j}$ inside the terms \eqref{justj2}).
By Theorem \ref{main1}, this implies (in both cases $s > 0$ and $s = 0$) that
\[ \begin{split}
\int_{B_R} &  \prod_{j=1}^m \left|T_j f_j(x) \right|^{q_j} dx \\ & \leq (C_n')^r  2^{\sum_{j \in J_2} q_j } (C_n R \deg \pi_{m+1})^{s} [[T]] \prod_{j=1}^{m} \left[ (\deg \pi_j) \int_{R^{n_j}} |f_{j}|^{p_j} \right]^{\frac{q_j}{p_j}} \end{split} \]
with the constant $C_n'$ being the one from Theorem \ref{main1}. Here it has been implicitly observed that $||f_{m+1}||_{L^1} = 1$.
This yields exactly the inequality \eqref{restrictresult}.
\end{proof}

\section{Applications}
\label{examplesec}

Note that in this section it is assumed that Theorem \ref{mainthm} has been fully proved; some readers may wish to read Section \ref{necessarysec} for the proof of the necessity of the testing condition \eqref{testing} and return here later.

\subsection{Simple corollaries of Theorems \ref{mainthm} and \ref{restrictstrongtype} for Radon-like transforms}

Even in the linear case, Theorem \ref{mainthm} has interesting and novel implications. In the spirit of a question of Stein \cite{stein1976}, it is possible to give an explicit criterion by which convolution with certain measures on affine varieties are $L^p$-improving for some pairs of exponents $p$ and $q$:
\begin{corollary}
Let $1 \leq k < n$ and suppose $\pi : \R^n \rightarrow \R^k$ is polynomial. Let $D \pi$ be the $k \times n$ Jacobian matrix of $\pi$, let $\Sigma \subset \R^n$ consist of those points in the zero set of $\pi$ at which $D \pi$ is full rank, and let $\mu$ be the measure on $\R^n$ given by
\[ \int f d \mu := \int_{\Sigma} f(y) \frac{w(y) \dH^{n-k}(y)}{|\det D \pi(y) (D \pi(y))^T|^{1/2}} \]
for any nonnegative continuous function $w$ on $\Sigma$ and any nonnegative Borel-measurable function $f$ on $\R^n$. Let $s$ be any positive real number. Then convolution with $\mu$ extends to a bounded map from $L^p(\R^n)$ to $L^{q}(\R^n)$ for all pairs $p,q$ satisfying
\begin{equation} \frac{1}{p} - \frac{1}{q} = \frac{(n-k)s}{n(s+1)} \ \text{ and } \ \left| \frac{1}{p} + \frac{1}{q} - 1 \right| \leq \frac{n-sk}{n(s+1)} \label{interval} \end{equation}
 if and only if
\begin{equation} \mathop{\sup_{M \in \R^{n \times n}}}_{ |\det M| = 1} \int_{\Sigma} \frac{(w(y))^{s} d \mu(y)}{|\det D \pi(y) M M^T (D \pi(y))^T|^{s/2}} < \infty. \label{seeoberlin} \end{equation}
\end{corollary}
\begin{proof}
To apply Theorem \ref{mainthm}, first select a defining function for the full Radon-like transform; the natural choice is to take $\pi(x,y) := \pi(x-y)$. In this case $||d_y \pi|| = ||d_x \pi||$, and by a slight modification of \eqref{equality1}, it must be the case that
\[ ||d_x \pi(x,y)||_\omega := \sqrt{ \det D \pi(x-y) M M^T (D \pi(x-y))^T }, \]
where $M$ is the matrix whose columns are the coordinates of $\omega_1,\ldots,\omega_n$, respectively.
Beyond Theorem \ref{mainthm} itself, the other necessary observation is that $||\mu * f||_q \leq C ||f||_p$ for all $f$ if and only if $||\mu * f||_{p'} \leq C ||f||_{q'}$ because the dual of convolution with $\mu$ is simply convolution conjugated by reflection $x \mapsto -x$ in $\R^n$. Setting $p$ so that $s = p'-1$ gives one of the two extreme points on the line segment \eqref{interval} and the other is obtained by duality. 
\end{proof}
The condition \eqref{seeoberlin} is very much in the same spirit as D.~Oberlin's curvature condition \cite{oberlin2000II}*{Corollary 3} but, when comparing \eqref{seeoberlin} to Oberlin's condition, one sees that \eqref{seeoberlin} is most reasonably understood as a sort of sublevel set inequality on the measure $\mu$ transported to the submanifold $\set{D \pi(y) \in \R^n}{ y \in \Sigma}$ rather than it being a condition ``directly'' on $\Sigma$ itself. To be clear about why this integral inequality is morally a strengthened sublevel set inequality, note that while \eqref{seeoberlin} is not explicitly computing the measure of a sublevel set, the layer cake formula
\[  \begin{split} \int_{\Sigma} & \frac{(w(y))^{s} d \mu(y)}{|\det D \pi(y) M M^T (D \pi(y))^T|^{s/2}} \\
& = s \int_0^\infty  \epsilon^{-s} \mu \left( \set{y \in \Sigma}{ |\det D \pi(y) M M^T (D \pi(y))^T|^{1/2} \leq \epsilon w(y) } \right) \frac{d \epsilon}{\epsilon} 
\end{split} \]
shows the very close relationship between \eqref{seeoberlin} and simpler sublevel set estimates of the sort appearing in Theorem \ref{restrictstrongtype} in the restricted strong-type cases. This is also to be expected when considering the interpolation theory of Lorentz spaces and its implications for the finiteness of \eqref{restrictresult}.

Another consequence of Theorems \ref{mainthm} and \ref{restrictstrongtype} is that restricted strong-type inequalities for Radon-like transforms with fractional-integration-like kernels are virtually automatic along the scaling line $q = np/k$ (though note that results of Secco \cite{secco1999} show that at least in some cases, full $L^p$--$L^q$ boundedness may also hold if one has additional information about the structure of the Radon-like transform).
\begin{corollary}[Fractional Integration]
Suppose $(\Omega,\pi,\Sigma)$ is a smooth incidence relation on $\R^n \times \R^{n'}$ of codimension $k$ such that the Radon-like transform 
\[ T f(x) := \int_{\li{x}} f(y) d \sigma(y) \]
admits a constant $C < \infty$ and an exponent $p \in [1,\infty)$ such that
\begin{equation} ||T f||_{L^{pn/k}(\R^n)} \leq C ||f||_{L^p(\R^{n'})} \label{thebound} \end{equation}
for all Borel functions $f$ on $\R^{n'}$.
Let $W$ be a nonnegative continuous function on $\Sigma$ such that
\[ \sup_{x \in \R^n} \int_{\li{x}} \chi_{W(x,y) \leq \epsilon} d \sigma(y) \leq C' \epsilon^s \]
for all $\epsilon > 0$, where $s$ is any fixed real number strictly greater than $1$ and $C'$ is finite. If $\pi(\cdot,y)$ is a polynomial function for each $y$ with bounded degree as a function of $y$, then the Radon-like transform
\[ \tilde T f(x) := \int_{\li{x}} f(y) \frac{d \sigma(y)}{W(x,y)} \]
admits a finite constant $C''$ such that
\begin{equation} ||\tilde T \chi_E ||_{L^{\tilde pn/k}(\R^n)} \leq C'' |E|^{\frac{1}{\tilde p}} \label{fracend} \end{equation}
for all Borel sets $E \subset \R^{n'}$, where $\tilde p := s p / (s-1)$.
\end{corollary}
\begin{proof}
By Theorem \ref{mainthm}, the bound \eqref{thebound} implies the existence of some constant $C_1$ such that
\[ \int_{\li{x}} \frac{d \sigma(y)}{||d_x \pi(x,y)||^{p'-1}_\omega} \leq C_1 \]
for all $x$ and all $\omega$ with $|\det \{\omega_i\}_{i=1}^n| = 1$. By Chebyshev's inequality, this implies the bound
\[ \sigma( \set{y \in \li{x}}{ ||d_x \pi(x,y)||_\omega \leq \delta}) \leq C_1 \delta^{p_1'-1} \]
for all $\delta > 0$, uniformly in $x$ and $\omega$.
Now let $w$ be any nonnegative continuous function on $\Sigma$ which is bounded above by $1/W(x,y)$ at every point.
For any nonnegative value of $\epsilon$,
\begin{align*}
w \sigma( & \set{y \in \li{x}}{ ||d_x \pi(x,y)||_\omega \leq \epsilon w(x,y)}) \\
& \leq \frac{\sigma}{W}\left( \set{y \in \li{x}}{ ||d_x \pi(x,y)||_\omega \leq \frac{\epsilon}{W(x,y)}} \right).
\end{align*}
Now for any fixed $\delta > 0$, every point $y \in \li{x}$ at which $||d_x \pi(x,y)||_\omega \leq \epsilon / W(x,y)$ must either satisfy $\delta \leq \epsilon / W(x,y)$ or $||d_x \pi(x,y)||_\omega \leq \delta$, so
\begin{align*}
w \sigma( & \set{y \in \li{x}}{ ||d_x \pi(x,y)||_\omega \leq \epsilon w(x,y)}) \\
& \leq \frac{\sigma}{W} \left( \set{y \in \li{x}}{ \delta \leq \frac{\epsilon}{W(x,y)}} \right) \\ & \qquad + \frac{\sigma}{W} \left( \set{y \in \li{x}}{ ||d_x \pi(x,y)||_\omega  \leq \delta \text{ and } \frac{\epsilon}{W(x,y)} \leq \delta } \right)  \\
& \leq \sum_{j=0}^\infty \frac{\sigma}{W} \left( \set{y \in \li{x}}{ 2^{-j-1} \frac{\epsilon}{\delta} < W(x,y) \leq 2^{-j} \frac{\epsilon}{\delta}} \right) \\
& \qquad + \frac{\delta}{\epsilon}  \sigma \left( \set{y \in \li{x}}{ ||d_x \pi(x,y)||_\omega \leq \delta } \right) \\
& \leq \sum_{j=0}^\infty 2^{j+1} \frac{\delta}{\epsilon} C' \left( 2^{-j} \frac{\epsilon}{\delta} \right)^s  + C_1 \delta^{p'} \epsilon^{-1} \leq C_2 (\epsilon^{s-1} \delta^{-(s-1)} + \delta^{p'} \epsilon^{-1})
\end{align*}
for some constant $C_2$ which is finite by virtue of the fact that $s > 1$. Note in particular that $C_2$ does not depend on $w$, $\omega$, $x$, or the parameters $\delta$ and $\epsilon$. Fixing $\delta := \epsilon^{s/(s+p'-1)}$ gives that
\[ \epsilon^{s-1} \delta^{-(s-1)} = \delta^{p'} \epsilon^{-1} = \epsilon^{\tilde p' - 1}, \]
so
\[ w \sigma(  \set{y \in \li{x}}{ ||d_x \pi(x,y)||_\omega \leq \epsilon w(x,y)}) \leq 2 C_2 \epsilon^{\tilde p'-1}. \]
By Theorem \ref{restrictstrongtype}, this uniform sublevel set inequality implies that every Borel set $E$ in $\R^{n'}$ satisfies
\[ \left( \int_{B_R} \left| \int_{\li{x}} \chi_E(y) w(x,y) d \sigma(y) \right|^{\frac{n\tilde p}{k}} dx \right)^{\frac{k}{n \tilde p}} \leq C'' |E|^\frac{1}{\tilde p} \]
for some constant $C''$ which is independent of $R$ and of the particular choice of $w$ for which $w(x,y) \leq 1/W(x,y)$. By Monotone Convergence as $R \rightarrow \infty$ and by fixing $w(x,y) = \min\{N W(x,y), 1/W(x,y)\}$ and letting $N \rightarrow \infty$, the inequality \eqref{fracend} must hold.
\end{proof}

The next two subsections are devoted to the proof of Theorem \ref{multiobj}. As mentioned in the introduction, Section \ref{detsec} establishes some fundamental inequalities for Gram determinants which are relevant to the proof, and the proof itself is carried out in Section \ref{intsec}.

\subsection{Some inequalities for Gram determinants} 
\label{detsec}

Let $H$ be any Hilbert space; for any vectors $v_1,\ldots,v_\ell \in H$, let $G(v_1,\ldots,v_\ell)$ be the associated Gram determinant, i.e.,
\begin{equation} G(v_1,\ldots,v_\ell) := \det \begin{bmatrix} \ang{v_1 , v_1} & \cdots & \ang{v_1 , v_\ell} \\ \vdots & \ddots & \vdots \\ \ang{v_\ell , v_1} & \cdots & \ang{v_\ell , v_\ell} \end{bmatrix}. \label{gram} \end{equation}
The matrix on the right-hand side of \eqref{gram} is always positive semidefinite  because
\[ \sum_{i=1}^{\ell}\sum_{i'=1}^\ell c^i c^{i'} \ang{v_i,v_{i'}} = \left| \left| \sum_{i=1}^\ell c^i v_i \right| \right|^2. \]
This identity guarantees that $G(v_1,\ldots,v_\ell)$ is never negative and that it vanishes if and only if $v_1,\ldots,v_\ell$ are linearly dependent.
The quantity $G(v_1,\ldots,v_\ell)$ can also be understood geometrically via the identity
\begin{equation} |G(v_1,\ldots,v_\ell)|^{-\frac{1}{2}} = \int_{\R^\ell} e^{- \pi || z^1 v_1 + \cdots z^\ell v_\ell||^2} dz; \label{setsize} \end{equation}
to prove this formula, first observe that when the matrix $G$ with entries $G_{ii'} := \ang{v_i,v_{i'}}$ is nonsingular, the fact that it is symmetric and positive-definite means that it has an inverse square root $G^{-1/2}$ satisfying the condition $\det G^{-1/2} = (\det G)^{-1/2} = (G(v_1,\ldots,v_\ell))^{-1/2}$. Making the change of variables $z \mapsto G^{-1/2} z$ gives
\[ \begin{split} \int_{\R^\ell} & e^{- \pi || z^1 v_1 + \cdots z^\ell v_\ell||^2} dz \\ & = |G(v_1,\ldots,v_\ell)|^{-1/2} \int_{\R^\ell} e^{-\pi (z^1)^2 - \cdots - \pi (z^\ell)^2} dz =  |G(v_1,\ldots,v_\ell)|^{-1/2}. \end{split} \]
On the other hand, if $G$ happens to be singular, then both sides of \eqref{setsize} have to be infinite. For the left-hand side, this is automatic; the right side must be infinite by virtue of the fact that when $u \in \R^\ell$ belongs to the kernel of $G$, the identity $||z^1 v_1 + \cdots + z^\ell v_\ell|| = ||(z^1 + t u^1) v_1 + \cdots + (z^\ell + t u^\ell) v_\ell||$ holds for all real $t$, which means that the integrand is constant in the direction of $u$ and therefore the integral is infinite by Fubini (because the integral on any hyperplane will always be nonzero).

The identity \eqref{setsize} has some important consequences. By making various linear changes of variables in the $z$ integral, one can easily verify that
 $G(v_1,\ldots,v_\ell)$ is invariant under permutations of the $v_i$ (verified by permuting indices of the coordinates of $z$) and that
\[ G(v_1,\ldots,v_{\ell-1}, v_\ell + c v_i) = G(v_1,\ldots,v_\ell) \]
for any constant $c$ and any $i \neq \ell$ (verified by the change of variables $z^i \mapsto z^i + c z^\ell$). Similarly, for any orthogonal projection $P$ on $H$,
\begin{equation} G(v_1,\ldots,v_\ell) \geq G(P v_1,\ldots, P v_\ell) \label{projectsmall} \end{equation}
by virtue of the fact that  $||\sum_{i=1}^\ell z^i P v_i|| \leq ||\sum_{i=1}^\ell z^i v_i||$ for all $v_i$ and $z$, so integral on the right-hand side of \eqref{setsize} cannot decrease when the $v_i$ are projected.

Using \eqref{projectsmall}, it is also possible to prove several identities and inequalities for Gram determinant when computed for collections of vectors that have some elements in common. To be more precise, let $u_1,\ldots,u_\ell \in H$ and suppose $P$ is projection onto the orthogonal complement of the span of $u_1,\ldots,u_\ell$. Then for any $v_1,\ldots,v_{\ell'}$, because $P v_{i'} = v_{i'} + \sum_{i=1}^\ell c_i u_i$ for some coefficients $c_i$, it must be the case that
\[ G(u_1,\ldots,u_\ell,v_1,\ldots,v_{\ell'}) = G(u_1,\ldots,u_\ell,P v_1,\ldots,P v_{\ell'}). \]
Applying the definition \eqref{gram} directly to $G(u_1,\ldots,u_\ell,P v_1,\ldots,Pv_{\ell'})$ expresses it as the determinant of a block-form matrix because $\ang{u_i, P v_{i'}} = 0$ for all $i,i'$. In particular, the Gram determinant factors as the product of the determinants of the blocks and therefore
\begin{equation}
 G(u_1,\ldots,u_\ell,v_1,\ldots,v_{\ell'}) = G(u_1,\ldots,u_\ell) G(P v_1,\ldots,P v_{\ell'}). \label{gramfactor} \end{equation}
If $P'$ is projection onto the orthogonal complement of $P v_1,\ldots,P v_{\ell'}$, then for any $w_1,\ldots,w_{\ell''} \in H$,
\begin{align*}
G & (u_1,\ldots,u_\ell,v_1,\ldots,v_{\ell'}) G(u_1,\ldots,u_\ell,w_1,\ldots,w_{\ell''})  \\
 & = (G(u_1,\ldots,u_\ell))^2 G( P v_1,\ldots,P v_{\ell'}) G(P w_1,\ldots, P w_{\ell''}) \\
 & \geq (G(u_1,\ldots,u_\ell))^2 G( P v_1,\ldots,P v_{\ell'}) G(P' P w_1,\ldots, P' P w_{\ell''})\\
 & = (G(u_1,\ldots,u_\ell))^2 G( P v_1,\ldots,P v_{\ell'}, P w_1,\ldots, P w_{\ell''}),
 \end{align*}
 where \eqref{gramfactor} is used to justify the second and last lines and the third line follows from \eqref{projectsmall}. Applying \eqref{gramfactor} again to the final line gives that
 \begin{equation}
 \begin{split}
  G & (u_1,\ldots,u_\ell,v_1,\ldots,v_{\ell'}) G(u_1,\ldots,u_\ell,w_1,\ldots,w_{\ell''}) \\ & \geq G(u_1,\ldots,u_\ell) G(u_1,\ldots,u_\ell,v_1,\ldots,v_{\ell'},w_1,\ldots,w_{\ell''}).
  \end{split} \label{splitfactor}
  \end{equation}
  A virtually identical argument (i.e., replacing each $w_i$ by $P' w_i$ and using \eqref{projectsmall} and \eqref{gramfactor}) shows that the analogue also holds when there are no common $u_i$, i.e., that
   \begin{equation}
 \begin{split}
  G & (v_1,\ldots,v_{\ell'}) G(w_1,\ldots,w_{\ell''}) \geq G(v_1,\ldots,v_{\ell'},w_1,\ldots,w_{\ell''}).
  \end{split} \label{splitfactor2}
  \end{equation}
  
There is one last family of inequalities for the Gram determinant which will be important in the proof of Theorem \ref{multiobj}.
For any fixed $n$-tuple of vectors $v_1,\ldots,v_n \in H$ and any $\ell \in \{1,\ldots,n\}$, let
  \begin{equation} I_\ell := \prod_{j=1}^n G(v_{j},\ldots,v_{j+\ell-1}), \label{cyclicdef} \end{equation}
  where the indices are interpreted periodically (i.e., $v_{n+1} := v_1$, $v_{n+2} := v_2$, etc.). These quantities $I_\ell$ and the various inequalities they satisfy will be of critical importance in the next section. The identity \eqref{splitfactor} implies when $\ell > 1$ that
  \[ G(v_{j},\ldots,v_{j+\ell-1}) G(v_{j+1},\ldots,v_{j+\ell}) \geq G(v_{j+1},\ldots,v_{j+ \ell - 1}) G(v_j,\ldots,v_{j+\ell})  \]
 for every $j$ (again, with indices understood periodically) so taking the product as $j$ ranges from $1$ to $n$ gives that
 $I_\ell^2 \geq I_{\ell-1} I_{\ell+1}$
for each $\ell  \in \{2,\ldots,n-1\}$. This means that the sequence $\{I_1,\ldots,I_n\}$ is log-concave. In particular, it must be the case that
\begin{equation} I_\ell \geq (I_{\ell-1})^{\frac{n-\ell}{n-\ell+1}} (I_n)^\frac{1}{n-\ell+1} = (I_{\ell-1})^{\frac{n-\ell}{n-\ell+1}} (G(v_1,\ldots,v_n))^\frac{n}{n-\ell+1} \label{adjacentG} \end{equation}
for each $\ell \in \{2,\ldots,n-1\}$ (because $I_n = (G(v_1,\ldots,v_n))^n$ by periodicity). Likewise by log concavity,
\begin{equation} I_\ell \geq (I_1)^{\frac{n-\ell}{n-1}} (I_n)^{\frac{\ell-1}{n-1}} \geq (G(v_1,\ldots,v_n))^{\ell} \label{globalG} \end{equation}
by virtue of the identity $I_n = (G(v_1,\ldots,v_n))^n$ and inequality $I_1 \geq G(v_1,\ldots,v_n)$ (which is proved by a repeated application of \eqref{splitfactor2} splitting $G(v_1,\ldots,v_n)$ into individual factors $G(v_i)$ for $i=1,\ldots,n$).
Then \eqref{adjacentG} combined with the identity $1 \geq (G(v_1,\ldots,v_n))^{\ell-1} I_{\ell-1}^{-1}$ (a consequence of \eqref{globalG}) gives
\begin{equation} I_\ell \geq (I_{\ell-1})^{s} (G(v_1,\ldots,v_n))^{s + (1-s) \ell} \text{ for all } s \leq \frac{n-\ell}{n-\ell+1} \label{allG} \end{equation}
for $2 \leq \ell \leq n$.
The upper limit of $s$ cannot be improved.  To see this, let $e_1,\ldots,e_n$ be mutually orthogonal unit vectors and define
$v_1 := N^{n-1} e_1$, $v_i := N^{n-1} e_1 + N^{-1} e_i$ for $i=2,\ldots,n$ for some large real number $N$.  Clearly $N^{n-1} \leq ||v_i|| \leq \sqrt{2} N^{n-1}$ for each $i$ and each $N \geq 1$, 
so $N^{2n(n-1)} \leq I_1 \leq 2^n N^{2n(n-1)}$. 
It is also easy to check that $G(v_1,\ldots,v_n) = G(v_1,v_2 - v_1,\ldots,v_n - v_1) = 1$ for each $N$ because this latter collection of vectors is mutually orthogonal. By \eqref{globalG}, it must be the case that $I_\ell \geq N^{2n(n-\ell)}$ for each $\ell = 1,\ldots,n$. But similarly, $G(v_j,\ldots,v_{j+\ell-1}) = G(v_j,v_{j+1} - v_j,\ldots, v_{j+\ell-1} - v_j) \leq ||v_j||^2 ||v_{j+1}-v_j||^2 \cdots ||v_{j+\ell-1} - v_j||^2 \leq 2 N^{2(n-1)} (2 N^{-2})^{\ell-1} = 2^\ell N^{2(n-\ell)}$, so $I_\ell \leq 2^{n \ell} N^{2n(n-\ell)}$ and therefore
\begin{equation}  \frac{(I_{\ell-1})^s (G(v_1,\ldots,v_n))^{s + (1-s)\ell}}{I_\ell} \geq \frac{N^{2sn(n-\ell+1)}}{2^{n \ell} N^{2 n (n-\ell)}} \rightarrow \infty \label{slimit} \end{equation}
as $N \rightarrow \infty$ if $s > (n-\ell)/(n-\ell+1)$.

\subsection{Proof of Theorem \ref{multiobj}}
\label{intsec}

\begin{proof}[Proof of Theorem \ref{multiobj}.] 
The integral appearing on the left-hand side of \eqref{examp1} can be put in the form \eqref{bddness}  by choosing
\[ \pi_j(x,y_j) := x^{j+\ell} - y^\ell_j + \sum_{i=1}^{\ell-1} |x^{j+i} - y^i_j|^2 \]
for each $j=1,\ldots,\ell$, with indices of $x$ understood as periodic of period $n$.  The first important calculation to carry out is to identify that the coarea measures $\sigma_j$ agree with the Lebesgue measure with respect to $t \in \R^{\ell-1}$ as it appears in \eqref{examp1}. For each $x \in \R^n$, the submanifold $\li{x}_j$ consists of those $y_j \in \R^{\ell}$ belonging to the paraboloid
\[ y^\ell_j = x^{j+\ell} + \sum_{i=1}^{\ell-1} |x^{j+i} - y^i_j|^2. \]
Viewing $y^{\ell}_j$ as a function of $y^{1}_j,\ldots,y^{\ell-1}_j$ gives a parametrization of the paraboloid as a graph, and the $(\ell-1)$-dimensional Hausdorff measure on this graph will equal
\[ \sqrt{1 + \sum_{i=1}^{\ell-1} \left|\frac{\partial y_j^\ell}{\partial y_j^i} \right|^2}  \, dy_j^1 \cdots dy^{\ell-1}_j  = \sqrt{ 1 + 4 \sum_{i=1}^{\ell-1} |y^i_j - x^{j+i}|^2} \, dy_j^1 \cdots dy^{\ell-1}_j. \]
For these maps $\pi_j$,
\[ D_{y_j} \pi_j (x,y_j) := \begin{bmatrix} 2 (y^1_j - x^{j+1}) & \cdots & 2(y^{\ell-1}_j - x^{j+\ell-1}) & -1 \end{bmatrix}, \]
which means that
\[ ||d_y \pi_j(x,y_j)|| = \sqrt{ 1 + 4 \sum_{i=1}^{\ell-1} |y^i_j - x^{j+i}|^2}. \]
Because $d \sigma_j = \dH^{\ell-1} / ||d_y \pi_j||$, it follows that $d \sigma_j = dy^1_j \cdots d y^{\ell-1}_j$ on $\li{x}_j$, and so a simple renaming of variables $y_j^1 = x^{j+1} + t^1,\ldots,y_j^{\ell-1} = x^{j+\ell-1} + t^{\ell-1}$ gives 
\[ \int_{\li{x}_j} f_j(y_j) d \sigma_j(y_j) = \int_{\R^{\ell-1}} f_j(x^{j+1} + t^1,\ldots,x^{j+\ell-1} + t^{\ell-1}, x^{j+\ell} + ||t||^2) dt \]
as is necessary for the application of Theorem \ref{mainthm}.

The matrix $D_x \pi_j(x,y_j)$ is a $1 \times n$ matrix with entry equal to $1$ in column $j + \ell$ and equal to $2 (x^{j+i'} - y^{i'}_j)$ in column $j+i'$ for $1 \leq i' \leq \ell-1$; the entries in all remaining columns are zero. Given $\{\omega_i\}_{i=1}^n$, it follows that
\[ ||d_x \pi_j(x,y_j)||_\omega^2 = \sum_{i=1}^n \left| \omega_i \cdot e_{j+\ell} + 2 \sum_{i'=1}^{\ell-1} (y_j^{i'}-x^{j+i'}) \omega_i \cdot e_{j+i'} \right|^2 \]
where $e_1,\ldots,e_n$ are the standard basis vectors of $\R^n$ and $\cdot$ is the usual inner product.
By Theorem \ref{mainthm} (and the substitution $t = 2 (y^{1}_j - x^{j+1},\ldots,y^{\ell-1}_j-x^{j+\ell-1})$), the constant $C_{p,\ell,n}$ in \eqref{examp1} is finite exactly when
\begin{equation}  \mathop{\sup_{\{\omega_i\}_{i=1}^n}}_{|\det \{\omega_i\}_{i=1}^n|=1} \prod_{j=1}^n \int_{\R^{\ell-1}} \left| \sum_{i=1}^n \left| \omega_{i} \cdot e_{j+\ell} + \sum_{i'=1}^{\ell-1} t^{i'} \omega_i \cdot e_{j+i'} \right|^2 \right|^{-\frac{p'-1}{2}} dt < \infty. \label{finiteprod} \end{equation}
The goal of the rest of the proof is consequently to determine the range of $p$ for which finiteness of \eqref{finiteprod} holds.

Consider for a moment a single term $j$ in the product \eqref{finiteprod}. Let $R_1,\ldots,R_\ell$ be vectors in $\R^n$ defined so that the coordinates of $R_{i}$ are $(e_{j+i} \cdot \omega_1,\ldots,e_{j+i} \cdot \omega_n)$ for each $i=1,\ldots,n$. By virtue of the identity
\[ \sum_{i=1}^n \left|  (e_{j+ \ell} \cdot \omega_i) + \sum_{i'=1}^{\ell-1} t^{i'} (e_{j+i'} \cdot \omega_i) \right|^2 = \left| \left| R_{\ell} + \sum_{i'=1}^{\ell-1} t^{i'} R_{i'} \right| \right|^2, \]
it is possible to choose any $t_* \in \R^{\ell-1}$ and write
\[ \sum_{i=1}^n \left|  (e_{j+ \ell} \cdot \omega_i) + \sum_{i'=1}^{\ell-1} t^{i'} (e_{j+i'} \cdot \omega_i) \right|^2 = \left| \left| R_{\ell} - \sum_{i'=1}^{\ell-1} t_*^{i'} R_{i'} + \sum_{i'=1}^{\ell-1} (t-t_*)^{i'} R_{i'} \right| \right|^2 \]
(where the reader is reminded that the notation $(t-t^*)^{i'}$ refers to the $i'$ coordinate of $t-t^*$).
For a suitable choice of $t_*$, it can always be arranged so that $R_{\ell} - t_*^1 R_1 - \cdots - t_*^{\ell-1} R_{\ell-1}$ is orthogonal to $R_1,\ldots,R_{\ell-1}$, giving that
\begin{align*}
\left| \left| R_{\ell} - \sum_{i'=1}^{\ell-1} t_*^{i'} R_{i'} + \sum_{i'=1}^{\ell-1} (t-t_*)^{i'} R_{i'} \right| \right|^2  & = || P R_{\ell}||^2 + \left| \left| \sum_{i'=1}^{\ell-1} (t-t_*)^{i'} R_{i'} \right| \right|^2 \\  = ||P R_{\ell}||^2& \left( 1 +  \left| \left| \frac{1}{||P R_{\ell}||} \sum_{i'=1}^{\ell-1} (t-t_*)^{i'} R_{i'} \right| \right|^2 \right),
\end{align*}
where $P$ is orthogonal projection onto the orthogonal complement of the span of $R_1,\ldots,R_{\ell-1}$. Now for any $p > 1$,
\begin{align*}
 \int_{\R^{\ell-1}} & \left| \left| R_{\ell} + \sum_{i'=1}^{\ell-1} t^{i'} R_{i'} \right| \right|^{-(p'-1)} dt \\
 & = ||P R_{\ell}||^{-(p'-1)} \int_{\R^{\ell-1}}  \left( 1 + \left| \left| \frac{1}{||P R_{\ell}||} \sum_{i'=1}^{\ell-1} (t-t_*)^{i'} R_{i'} \right| \right|^2 \right)^{-\frac{p'-1}{2}}dt \\
 & = ||P R_{\ell}||^{-(p'-\ell)} \int_{\R^{\ell-1}}  \left(1 + \left| \left| \sum_{i'=1}^{\ell-1} t^{i'} R_{i'} \right| \right|^2 \right)^{-\frac{p'-1}{2}}dt \\
 & = ||P R_{\ell}||^{-(p'-\ell)} \left(G(R_1,\ldots,R_{\ell-1})\right)^{-\frac{1}{2}} \int_{\R^{\ell-1}}  (1+||t||^2)^{-\frac{p'-1}{2}}dt \\
 & = \frac{(G(R_1,\ldots,R_{\ell-1}))^{(p'-\ell)/2}}{(G(R_1,\ldots,R_{\ell}))^{(p'-\ell)/2}} \left(G(R_1,\ldots,R_{\ell-1})\right)^{-\frac{1}{2}} \! \int_{\R^{\ell-1}}  (1+||t||^2)^{-\frac{p'-1}{2}}dt \\
 & =  \frac{(G(R_1,\ldots,R_{\ell-1}))^{(p'-\ell-1)/2}}{(G(R_1,\ldots,R_{\ell}))^{(p'-\ell)/2}} \int_{\R^{\ell-1}}  (1+||t||^2)^{-\frac{p'-1}{2}}dt.
 \end{align*}
 The integral quantity on this final line is independent of $R_1,\ldots,R_\ell$ and is finite if and only if $p'-1 > \ell-1$, i.e., $p < \ell/(\ell-1)$. 
  Taking a periodic product gives that
 \begin{equation} \prod_{j=1}^n \int_{\R^{\ell-1}} \left| \left| R_{j+\ell} + \sum_{i=1}^{\ell-1} t^i R_{j+i} \right| \right|^{-(p'-1)} dt = (C_p)^n \frac{(I_{\ell-1})^{(p'-\ell-1)/2}}{(I_{\ell})^{(p'-\ell)/2}}. \label{pprod} \end{equation}
 By \eqref{allG} and \eqref{slimit}, the ratio  ${(I_{\ell-1})^{(p'-\ell-1)/2}}/{(I_{\ell})^{(p'-\ell)/2}}$ is uniformly bounded for all $R_1,\ldots,R_n$ with $G(R_1,\ldots,R_n) = 1$ (which is true in this case because the matrix with rows $R_1,\ldots,R_n$ equals the matrix with columns $\omega_1,\ldots,\omega_n$) if and only if 
 \[ p' - \ell - 1 \leq \frac{n-\ell}{n-\ell+1} (p' - \ell), \]
 which occurs exactly when $p' \leq n+1$.
Thus it follows that
 \[
 \prod_{j=1}^n \int_{\R^{\ell-1}} \left| \sum_{i=1}^n \left| \omega_{i} \cdot e_{j+\ell} + \sum_{i'=1}^{\ell-1} t^{i'} \omega_i \cdot e_{j+i'} \right|^2 \right|^{-\frac{p'-1}{2}} dt \leq (C_p)^n
 \]
 uniformly in $\{\omega_i\}_{i=1}^n$ with $|\det \{\omega_i\}_{i=1}^n| = 1$ for a finite constant $C_p$ if and only if  $(n+1)/n \leq p < \ell/(\ell-1)$.
 
When $p = \ell/(\ell-1)$, one instead uses the inequality
\begin{align}
 & \left| \set{ t \in \R^{\ell-1}}{  \left| \left| R_{j+\ell} + \sum_{i=1}^{\ell-1} t^{i} R_{j+i} \right| \right| \leq \epsilon }  \right| \nonumber \\
 & \qquad \qquad = \left| \set{ t \in \R^{\ell-1}}{  ||P R_{j+\ell}||^2 + \left| \left| \sum_{i=1}^{\ell-1} t^{i} R_{j+i} \right| \right|^2 \leq \epsilon^2 } \right| \nonumber \\
 & \qquad \qquad \leq \left| \set{ t \in \R^{\ell-1}}{  \left| \left| \sum_{i=1}^{\ell-1} t^{i} R_{j+i} \right| \right| \leq \epsilon} \right| \nonumber \\
 & \qquad \qquad = C_\ell (G(R_{j+1},\ldots,R_{j+\ell-1}))^{-1/2} \epsilon^{\ell-1}. \label{rstype}
 \end{align}
 Taking a product over $j$ and using the inequality $I_\ell \geq (G(R_1,\ldots,R_n))^{\ell} = 1$ gives
 \[ \prod_{j=1}^n \sup_{\epsilon > 0} \epsilon^{-\ell+1} \sigma_j \left( \set{ y \in \li{x}_j}{ ||d_x \pi_j(x,y_j)||_\omega \leq \epsilon} \right) \leq (C_\ell)^n, \]
 which implies by Theorem \ref{restrictstrongtype} that \eqref{examp1} holds in the restricted strong-type sense when $p = \ell/ (\ell-1)$.
\end{proof}

\subsection{Mixed-norm inequalities from Theorems \ref{mainthm} and \ref{restrictstrongtype}}
\label{mixnorm}

Although not specifically phrased in terms of mixed-norm inequalities, the multilinear nature of Theorems \ref{mainthm} and \ref{restrictstrongtype} allows one to deduce certain mixed-norm inequalities for Radon-like transforms. For purposes of clarity (and to take advantage of the existing computations), we consider mixed norm estimates for the model operators
\[ T f(x,x') := \int_{\R^{\ell-1}} f(x + t, x' + ||t||^2) dt \]
for $x \in \R^{\ell-1}$ and $x' \in \R$.  Theorem \ref{restrictstrongtype} implies the following result.
\begin{theorem}
The operator $T$ satisfies a restricted strong-type $L^{\ell/(\ell-1)}(\R^{\ell-1} \times \R) \rightarrow L^{\infty}_{x'} ( L^{\ell}_x)$ inequality.
\end{theorem}
The endpoint mixed-norm inequality just stated does not appear to be widely observed in the literature (and although the exponents involved suggest that one might be able to prove this inequality by slicing the paraboloid into spheres of one lower dimension, no elementary argument of this sort appears to suffice when $\ell > 2$).
By interpolation with the standard $L^{(\ell+1)/\ell} \rightarrow L^{\ell+1}$ inequality, it follows that
\[ || T f ||_{L^{q}_{x'} (L^{p'}_x)} \leq C_p ||f||_{p} \ \forall f \in L^p (\R^{\ell-1} \times \R) \]
when $(\ell+1)/\ell \leq p < \ell / (\ell-1)$ and $1/q = 1 - (\ell/p')$ (where, as usual, $p$ and $p'$ are dual exponents). The restricted strong-type $L^{\ell/(\ell-1)}(\R^{\ell-1} \times \R) \rightarrow L^{\infty}_{x'} ( L^{\ell}_x)$ bound is particularly interesting because the full $L^{\ell/(\ell-1)}(\R^{\ell-1} \times \R) \rightarrow L^{\infty}_{x'} ( L^{\ell}_x)$ fails to hold.

\begin{proof}
Just as was computed in the previous section, this operator can be regarded as being of the form \eqref{tjdef} for a defining function $\pi((x,x'),(y,y')) := x' - y' + ||x-y||^2$ when $y \in \R^{\ell-1}$ and $y' \in \R$. Exactly as was the case for \eqref{rstype},
\[ \begin{split} \sup_{\epsilon > 0} \epsilon^{-\ell+1} & \sigma \left( \set{ (y,y') \in \li{(x,x')}}{ ||d_{(x,x')} \pi ((x,x'),(y,y'))||_\omega \leq \epsilon} \right) \\ & \leq C_\ell (G(R_1,\ldots,R_{\ell-1}))^{-1/2} \end{split} \]
where $R_1,\ldots,R_{\ell-1}$ are vectors in $\R^{\ell}$ with the property that $R_i := (\omega_1 \cdot e_i,\ldots,\omega_\ell \cdot e_i)$ for $i=1,\ldots,\ell$ with $e_1,\ldots,e_\ell$ being the standard basis of $\R^{\ell}$. Now suppose that $\pi'((x,x'),z) := z - x'$ for $z \in \R$. For this $\pi'$, $||d_{(x,x')} \pi'||_\omega = ||R_\ell|| = G(R_\ell)^{1/2}$. Therefore
\begin{align*}
\sup_{z'} & \frac{1}{||d_{(x,x')} \pi'((x,x'),z') ||_\omega} \\ & \cdot \sup_{\epsilon > 0} \epsilon^{-\ell+1} \sigma \left( \set{ (y,y') \in \li{(x,x')}}{ ||d_{(x,x')} \pi ((x,x'),(y,y'))||_\omega \leq \epsilon} \right) \\ & \leq C_\ell (G(R_1,\ldots,R_{\ell-1}))^{-1/2} (G(R_\ell))^{-1/2} \leq C_\ell (G(R_1,\ldots,R_\ell))^{-1/2} = C_\ell. 
\end{align*}
Because
\[ \ell = \frac{1 \cdot \ell}{ \ell / (\ell-1)} + 1, \]
Theorem \ref{restrictstrongtype} (with $(k_1,q_1,p_1) := (1,\ell,\ell/(\ell-1))$ and $(k_2,q_2,p_2) := (1,1,1)$ followed by the standard limiting argument $R \rightarrow \infty$) implies that
\[ \int_{\R} \int_{\R^{\ell-1}} \left| T \chi_E(x,x') \right|^{\ell} |g(x')| dx dx' \leq C |E|^{\ell-1} ||g||_{L^1(\R)} \]
for all Borel sets $E \subset \R^{\ell-1} \times \R$, all Borel functions $g$ on $\R$, and some constant $C$ depending only on $\ell$. Fixing $E$ momentarily and applying duality in $g(x')$ gives that
\[ \mathop{\operatorname{ess.sup}}_{x' \in \R}  \int_{\R^{\ell-1}} |T \chi_E(x,x')|^{\ell} dx \leq C |E|^{\ell-1}. \]
Raising both sides to the power $1/\ell$ gives that $T$ satisfies a restricted strong-type $L^{\ell/(\ell-1)} \rightarrow L^\infty_{x'} ( L^\ell_x)$ inequality.

Regarding sharpness, observe that the full $L^{\ell/(\ell-1)} \rightarrow L^\infty_{x'} ( L^\ell_x)$ bound cannot hold, for if it did, it would imply an inequality of the form
\[ \int_{\R} \int_{\R^{\ell-1}} \left| T f(x,x') \right|^{\ell} |g(x')| dx dx' \leq C ||f||_{L^{\ell/(\ell-1)}(\R^{\ell-1} \times \R)}^{\ell} ||g||_{L^1(\R)} \]
for all $f$ and $g$, which would subsequently imply finiteness of 
\[ \sup_{z'} \frac{1}{||d_{(x,x')} \pi'((x,x'),z') ||_\omega} \int_{\li{(x,x')}} \frac{d \sigma}{||d_{(x,x')} \pi||_\omega^{\ell-1}} \]
for all $(x,x')$ and $\omega := \{\omega_i\}_{i=1}^\ell$.
In this case, it has already been observed that the power $\ell-1$ is not large enough to achieve finiteness of the integral on the right-hand side. The fact that the restricted strong-type inequality holds while the full mixed-norm inequality fails suggests that one should not immediately assume that strong endpoint inequalities always hold. For example, there may be situations in which endpoint mixed-norm inequalities outside the range proved by Christ and Erdo\u{g}an \cite{ce2008} actually fail to hold.
\end{proof}

\section{Proof of Theorem \ref{mainthm}: Necessity of the testing condition}
\label{necessarysec} 

This section marks the return to the main goal of proving Theorem \ref{mainthm}; it remains only to establish \eqref{testing} under the assumption that \eqref{bddness} holds. The process is carried out by first computing some basic limits related to Knapp-type examples and then employing these computations to complete the proof. Throughout this section, no algebraic assumptions on the mappings $\pi_j$ are necessary or relevant. 

\subsection{Preliminary Knap-type calculations}

Given a smooth incidence relation $(\Omega,\pi,\Sigma)$ on $\R^{n} \times \R^{n'}$ of codimension $k$, the associated mapping $\pi$ may always be multiplied on the left by any invertible $k \times k$ matrix with smooth entries without changing the definition of the resulting incidence set $\Sigma$. Among all such equivalent choices for $\pi$, there is one which most efficiently captures the geometric properties of the associated Radon-like transform (or at least those geometric properties which matter for the present purposes).
Because the determinant $\det D_x \pi (D_x \pi)^T$ is nonvanishing on $\Sigma$, by restricting the domain of $\pi$ to some smaller set $\Omega' \subset \Omega$ as needed, it may be assumed that the matrix  $[ D_x \pi(x,y) (D_x \pi(x,y))^T ]^{-1/2}$ is uniquely defined (as a positive-definite symmetric matrix) and is a smooth function of $x$ and $y$ on some open set $\Omega \subset \Sigma$ and that $||d_x \pi(x,y)||$ and $||d_y \pi(x,y)||$ are both strictly positive on $\Omega$ as well. (An elementary proof of smoothness of $[ D_x \pi(x,y) (D_x \pi(x,y))^T ]^{-1/2}$ can be established using the Cauchy integral formula; see \cite{gressman2021}*{Appendix}).

For the remainder of this section, let 
\begin{equation} \overline{\pi}(x,y) := [ D_x \pi(x,y) (D_x \pi(x,y))^T ]^{-1/2} \pi(x,y). \label{pibar} \end{equation}
Because $\pi(x,y) = 0$ for any point $(x,y) \in \Sigma$, it must be the case that
\[ \frac{\partial \overline{\pi}}{\partial x^i}(x,y) = [ D_x \pi(x,y) (D_x \pi(x,y))^T ]^{-1/2} \frac{\partial \pi}{\partial x^i} (x,y) \]
at all points $(x,y) \in \Sigma$ for all $i = 1,\ldots,n$ by the product rule: any term in the product rule expansion of $\partial \overline{\pi} / \partial x^i$  in which a derivative falls on the matrix $[ D_x \pi(x,y) (D_x \pi(x,y))^T ]^{-1/2}$ must vanish identically on $\Sigma$ because it will be multiplied by an undifferentiated $\pi$ factor which necessarily vanishes on $\Sigma$. The analogous formula is true for derivatives with respect to $y^i$ as well. Consequently it must be the case that
\[ D_x \overline{\pi} =  [ D_x \pi (D_x \pi)^T ]^{-1/2} D_x \pi \text{ and } D_y \overline{\pi} =  [ D_x \pi (D_x \pi)^T ]^{-1/2} D_y \pi \]
at every point of $\Sigma$. It follows by \eqref{equality1} and \eqref{equality2} that
\begin{equation} ||d_x \overline{\pi}|| = 1  \text{ and } ||d_y \overline{\pi}|| = \frac{||d_y \pi||}{||d_x \pi||} \label{normalized} \end{equation}
at all points of $\Sigma$. Moreover, $D_x \overline{\pi} (D_x \overline{\pi})^T$ is exactly the $k \times k$ identity matrix at every point of $\Sigma$, which means that the rows of $D_x \overline{\pi}$ are orthonormal with respect to the standard inner product on $\R^n$.

To prove necessity of the testing condition \eqref{testing}, one must find suitable families of functions to which the multilinear Radon-Brascamp-Lieb transforms may be applied. The precise functions to be used are as follows:
fix any $x_0 \in \R^n$, and for each $j \in \{1,\ldots,m\}$ such that $p_j > 1$ and each $\delta$ sufficiently small, define
\begin{equation} f_{j,\delta}(y) := \delta^{-\frac{k_j}{p_j}} \chi_{||\overline{\pi}_j(x_0,y)|| < \delta} \left[ \frac{w_j(x_0,y)}{||d_x \pi_j(x_0,y)||} \right]^{p_j'-1} \eta_j(y) \label{bigpfn} \end{equation}
for any fixed, nonnegative, continuous $\eta_j$ with values in $[0,1]$ which is supported close to $\li{x_0}_j$ in the sense of Section \ref{appendix}.  If $j \in \{1,\ldots,m\}$ has $p_j = 1$, take instead
\begin{equation} f_{j,\delta}(y) := \delta^{-k_j} \chi_{||\overline{\pi}_j(x_0,y)|| < \delta} \eta_j(y) \label{ponefn} \end{equation}
for some nonnegative continuous $\eta_j$ with values in $[0,1]$ supported close to $\li{x_0}_j$.  The transforms $T_j$ will be applied to these functions and both sides of \eqref{bddness} will be examined in the limit $\delta \rightarrow 0^+$. The key phenomenon occurring in this limit is that the mass of the left-hand side integral becomes concentrated on the $\delta$-ball centered at $x_0$. This is a regime which corresponds to one of the typical families of Knapp-type examples and, in that sense, the principal innovation of the argument below is not the concept, but rather the level of detail.

There are two key computations regarding the limit $\delta \rightarrow 0^+$, both recorded in the proposition below.
\begin{proposition}
For a smooth incidence relation $(\Omega,\pi,\Sigma)$ on $\R^n \times \R^{n'}$ with codimension $k$, suppose $f$ is any continuous function on $\R^{n'}$ supported close to $\li{x_0}$.  Then for $\overline{\pi}$ as in \eqref{pibar},
\begin{equation} \lim_{\delta \rightarrow 0^+} \frac{1}{\delta^k} \int_{\R^{n'}} f(y) \chi_{||\overline{\pi}(x_0,y)|| < \delta} dy = c_k \int_{\li{x_0}} f(y) ||d_x \pi(x_0,y)|| d \sigma(y). \label{spaceint} \end{equation}
Additionally, the map
\[ \xi \mapsto \int_{\li{(x_0+\delta \xi)}} f(y) \chi_{||\overline{\pi}(x_0,y)|| < \delta} d \sigma(y) \]
is well-defined on $||\xi|| < c$ for any fixed $c < 1$ provided $\delta$ is sufficiently small; this map has the property that
\begin{equation}
\lim_{\delta \rightarrow 0^+} \int_{\li{(x_0 + \delta \xi)}} f(y) \chi_{||\overline{\pi}(x_0,y)|| < \delta} d \sigma(y) = \int_{\li{x_0}} f(y) d \sigma(y) \label{pointint}
\end{equation}
with convergence that is uniform for all $||\xi|| < c$.
\end{proposition}
\begin{proof}
To prove \eqref{spaceint}, one needs only to use the coarea formula for $\overline{\pi}$:
\begin{equation}
  \int_{\R^{n'}} f(y) \chi_{||\overline{\pi}(x_0,y)|| < \delta} dy  = \int_{||s|| < \delta} \left[ \int_{y \, : \, \overline{\pi}(x_0,y) = s} f(y) \frac{\dH^{n'-k}(y)}{||d_y \overline{\pi}(x_0,y)||} \right] ds \label{coareaid1}
 \end{equation}
 The inner integral on the right-hand side is continuous in $s$ at $s = 0$. (This is a consequence of the upcoming Corollary \ref{integralcorollary} in Section \ref{appendix} applied to the family of maps $\pi_s (y) := \pi(x_0,y) - s$.)
Therefore
 \begin{align*}
 \lim_{s \rightarrow 0} \int_{y \, : \, \overline{\pi}(x_0,y) = s} f(y) \frac{\dH^{n'-k}(y)}{||d_y \overline{\pi}(x_0,y)||} &
 =  \int_{\Sigma_{\overline{\pi}(x_0,\cdot)}} f(y) \frac{\dH^{n'-k}(y)}{||d_y \overline{\pi}(x_0,y)||} \\
 & = \int_{\li{x_0}} f(y) ||d_x \pi(x_0,y)|| d \sigma(y),
 \end{align*}
 where the final line uses the computation \eqref{normalized} to relate the coarea measure for $\overline{\pi}(x_0,\cdot)$ to the corresponding coarea measure for $\pi(x_0,\cdot)$. By continuity at $s=0$, the limit of the right-hand side of \eqref{coareaid1} multiplied by $\delta^{-k}$ exists  and   \[  \lim_{\delta \rightarrow 0^+} \frac{1}{\delta^k} \int_{\R^{n'}} f(y) \chi_{||\overline{\pi}(x_0,y)|| < \delta} dy = c_k \int_{\li{x_0}} f(y) ||d_x \pi(x_0,y)|| d \sigma(y), \]
 where $c_k$ is the volume of the $k$-dimensional unit ball.

As for \eqref{pointint},
by Taylor's Theorem, for any $y$ at which $\overline{\pi}(x_0+\delta \xi,y) = 0$, 
\begin{align*}
\overline{\pi}(x_0, y) = \overline{\pi}(x_0+\delta \xi, y) - \left. D_x \overline{\pi} \right|_{(x_0 + \delta \xi, y)} (\delta \xi) + O(\delta^2)
\end{align*}
with the implicit constant in the $O(\delta^2)$ term being uniform for all $y$ belonging to any fixed compact set (in this case, the support of $f$) and all $||\xi|| < c < 1$. Now $\overline{\pi}(x_0+\delta \xi, y) = 0$ and the rows of $\left. D_x \overline{\pi} \right|_{(x_0 + \delta \xi, y)}$ are orthonormal, so it follows that
$|| \overline{\pi}(x_0, y)|| \leq c \delta + O(\delta^2)$. This means that for all $\delta$ sufficiently small, the characteristic function $\chi_{||\overline{\pi}(x_0,y))|| < \delta}$ will be identically $1$ on the support of $f$ inside the integral on the left-hand side of \eqref{pointint}.
Consequently
\[ \int_{\li{(x_0 + \delta \xi)}} f(y) \chi_{||\overline{\pi}(x_0,y)|| < \delta} d \sigma(y) = \int_{\li{(x_0 + \delta \xi)}} f(y) d \sigma(y) \]
for all sufficiently small $\delta$. By \eqref{limformula} from Corollary \ref{integralcorollary}, it follows that this expression converges uniformly in $\xi$ to 
\[ \int_{\li{x_0}} f(y) d \sigma(y) \]
as $\delta \rightarrow 0^+$.
\end{proof}

\subsection{Conclusion of the proof of Theorem \ref{mainthm}}

We are now ready to complete the proof of Theorem \ref{mainthm} by showing that the testing condition \eqref{testing} must hold. 
\begin{proof}[Proof of the necessity of \eqref{testing}]
Recall the definitions \eqref{bigpfn} and \eqref{ponefn} of the testing functions $f_{j,\delta}$ to be used.
By the limit computation \eqref{spaceint}, when $p_j > 1$, it must be the case that
\begin{align}
\lim_{\delta \rightarrow 0^+}&  || f_{j,\delta}||_{p_j} \nonumber \\ & = \left[ \lim_{\delta \rightarrow 0^+} \int_{\R^{n_j}} \left|\delta^{-\frac{k_j}{p_j}} \chi_{||\overline{\pi}_j(x_0,y_j)|| < \delta} \left[ \frac{w_j(x_0,y_j)}{||d_x \pi_j(x_0,y_j)||} \right]^{p_j'-1} \eta_j(y_j) \right|^{p_j} dy_j \right]^{\frac{1}{p_j}} \nonumber \\
 & = \left[ c_{k_j} \int_{\li{x_0}_j} \frac{(w_j(x_0,y_j))^{p'_j}}{||d_x \pi_j(x_0,y_j)||^{p'_j-1}} |\eta_j(y_j)|^{p_j} d \sigma_j(y_j) \right]^{\frac{1}{p_j}} \nonumber \\
& \leq \left[ c_{k_j} \int_{\li{x_0}_j} \frac{(w_j(x_0,y_j))^{p'_j}}{||d_x \pi_j(x_0,y_j)||^{p'_j-1}} \eta_j(y_j) d \sigma_j(y_j) \right]^{\frac{1}{p_j}}. \label{normcalc}
\end{align}
Here the identity $(p_j' - 1)p_j = p_j'$ is used to compute exponents inside the integral; this explains why the final exponent on $w_j$ is $p_j'$. The exponent of $||d_x \pi_j||$ becomes $p_j'-1$ after accounting for the extra factor of $||d_x \pi_j||$ arising from \eqref{spaceint}. Lastly $|\eta_j|^{p_j} \leq \eta_j$ because $\eta_j$ is nonnegative and no greater than one (and $p_j > 1$).
 If $p_j = 1$, then one instead has
\begin{equation}
\begin{split}
\lim_{\delta \rightarrow 0^+}  || f_{j,\delta}||_{1} & = \lim_{\delta \rightarrow 0^+} \int_{\R^{n_j}} \delta^{-k_j} \chi_{||\overline{\pi}_j(x_0,y_j)|| < \delta}  \eta_j(y_j) dy_j  \\
 & = c_{k_j} \int_{\li{x_0}_j} ||d_x \pi_j(x_0,y_j)|| \eta_j(y_j) d \sigma_j(y_j)
 \end{split} \label{normcalc1}
\end{equation}
also by \eqref{spaceint}.

Now consider the integral
\begin{equation}
\int_{B_{c \delta}(x_0)} \prod_{j=1}^m \left| T_j f_{j,\delta} (x) \right|^{q_j} dx  =  \int_{B_c(0)} \prod_{j=1}^m  \left| \delta^{\frac{k_j}{p_j}} T_j f_{j,\delta} (x_0 + \delta \xi ) \right|^{q_j} d\xi \label{integralside}
\end{equation}
(here $B_{c \delta}(x_0)$ indicates the Euclidean ball at $x_0$ with radius $c \delta$), where the change of variables $x \mapsto x_0 + \delta \xi$ is applied to reach the right-hand side, and the Jacobian factor $\delta^n$ is absorbed into the product over $j$ using the fact that $k_1 q_1 /p_1 + \cdots + k_m q_m / p_m = n$.
By \eqref{pointint}, when $p_j > 1$,
\begin{align*}
 \lim_{\delta \rightarrow 0^+} \delta^{\frac{k_j}{p_j}} T_j f_{j,\delta}(x_0 + \delta \xi) & = \lim_{\delta \rightarrow 0^+} \delta^{\frac{k_j}{p_j}} \int_{\li{(x_0 + \delta \xi)}_j} f_{j,\delta} (y)w_j(x_0 + \delta \xi,y_j) d \sigma_j(y_j)
 \\
 & = \int_{\li{x_0}_j} \frac{(w_j(x_0,y))^{p'_j}}{||d_x \pi_j(x_0,y_j)||^{p'_j-1}} \eta_j(y_j) d \sigma_j(y_j) 
 \end{align*}
with uniform convergence for all $||\xi|| < c < 1$ (where we have also used the fact that $w_j(x_0 + \delta \xi,y_j)$ converges uniformly to $w_j(x_0,y_j)$ on the support of $\eta_j$ and that the integrals
\[ \int_{\li{(x_0+\delta \xi)}_j} \frac{(w_j(x_0,y))^{p'_j-1}}{||d_x \pi_j(x_0,y_j)||^{p'_j-1}} \eta_j(y_j) d \sigma_j(y_j) \]
are uniformly bounded in $\xi$ for all small $\delta$ and $||\xi|| < c$, which is a consequence of Corollary \eqref{integralcorollary} in Section \ref{appendix} as well).  If instead $p_j = 1$,
 \[ \lim_{\delta \rightarrow 0^+} \delta^{k_j} T_j f_{j,\delta}(x_0 + \delta \xi) = \int_{\li{x_0}_j} w_j(x_0,y_j) \eta_j(y_j) d \sigma_j(y_j). \]
Therefore the limit as $\delta \rightarrow 0^+$ of \eqref{integralside} exists and 
\begin{align}
\lim_{\delta \rightarrow 0^+}  \int_{B_{c \delta}(x_0)} \prod_{j=1}^m & \left| T_j f_{j,\delta} (x) \right|^{q_j} dx \label{thelimit} \\ & = |B_c(0)|
\prod_{j \, : \, p_j = 1} \left| \int_{\li{x_0}_j} w_j(x_0,y_j) \eta_j(y_j) d \sigma(y_j) \right|^{q_j}  \nonumber \\ & \cdot \prod_{j  \, : \, p_j > 1}  \left| \int_{\li{x_0}_j} \frac{(w_j(x_0,y_j))^{p'_j} \eta_j(y_j)}{||d_x \pi_j(x_0,y_j)||^{p'_j-1}}  d \sigma_j(y_j) \right|^{q_j}. \nonumber
\end{align}
By \eqref{normcalc} and \eqref{normcalc1},
\begin{align*}
\lim_{\delta \rightarrow 0^+} \prod_{j=1}^m ||f_{j,\delta}||_{p_j}^{q_j} & \leq \prod_{j \, : \, p_j = 1} \left[ c_{k_j} \int_{\li{x_0}_j} ||d_x \pi_j(x_0,y_j)|| \eta_j(y_j) d \sigma_j(y_j) \right]^{q_j} \\ & \qquad \cdot \prod_{j \, :  \, p_j > 1}  \left[ c_{k_j}  \int_{\li{x_0}_j} \frac{(w_j(x_0,y_j))^{p'_j}}{||d_x \pi_j(x_0,y_j)||^{p'_j-1}} \eta_j(y_j) d \sigma_j(y_j) \right]^{\frac{q_j}{p_j}}.
\end{align*}
Combining this inequality with \eqref{thelimit}, the boundedness condition \eqref{bddness} therefore implies that
\begin{align*}
\prod_{j \, : \, p_j = 1} & \left[ \frac{\int_{\li{x_0}_j} w_j(x_0,y_j) \eta_j(y_j) d \sigma_j(y_j)}{\int_{\li{x_0}_j} ||d_x \pi_j(x_0,y_j)|| \eta_j(y_j) d \sigma_j(y_j)} \right]^{q_j} \\
& \cdot \prod_{j \, : \, p_j > 1}  \left[ \int_{\li{x_0}_j} \frac{(w_j(x_0,y_j))^{p'_j}}{||d_x \pi_j(x_0,y_j)||^{p'_j-1}} \eta_j(y_j) d \sigma_j(y_j) \right]^{\frac{q_j}{p_j'}} \leq c ||T||
\end{align*}
for some constant $c$ depending only on exponents and dimensions. 

The next steps of the argument involve taking limits to remove the cutoff functions $\eta_j$. For those indices $j$ having $p_j = 1$, fix any point $z_j \in \li{x_0}$ and let $\eta_j$ be replaced by $\eta_j(y_j) \max \{ 1 - R ||y_j-z_j||,0\}$ where $\eta_j$ is the cutoff function supported near $z_j$ given by the second part of Proposition \ref{partitionprop} in the Appendix. The numerator and denominator of the ratio
\[ \frac{\int_{\li{x_0}_j} w_j(x_0,y_j) \eta_j(y_j) \max \{ 1 - R ||y_j-z_j||,0\} d \sigma_j(y_j)}{\int_{\li{x_0}_j} ||d_x \pi_j(x_0,y_j)|| \eta_j(y_j) \max \{ 1 - R ||y_j-z_j||,0\} d \sigma_j(y_j)} \]
do not vanish for any $R > 0$, but as $R \rightarrow \infty$, the support of the integrand is contained in the $1/R$-neighborhood of $z_j$, and so by continuity of both integrands, the limit as $R \rightarrow \infty$ exists and equals 
 \[ \frac{w_j(x_0,z_j)}{||d \pi_j (x_0,z_j)||}. \]
Now taking a supremum over each $z_j$ gives
\begin{align*}
\prod_{j \, : \, p_j = 1} & \left[ \sup_{z_j \in \li{x_0}_j}   \frac{w_j(x_0,z_j)}{||d_x \pi_j(x_0,z_j)||} \right]^{q_j} \\ & \cdot  \prod_{j \, : \, p_j > 1}  \left[ \int_{\li{x_0}_j} \! \! \! \frac{(w_j(x_0,y_j))^{p'_j} \eta_j(y_j) d \sigma_j(y_j)}{||d_x \pi_j(x_0,y_j)||^{p'_j-1}}  \right]^{\frac{q_j}{p_j'}} \leq c ||T||.
\end{align*}
For the remaining indices $j$, let $\eta_j$ be replaced by terms of a sequence of cutoff functions tending to $1$ everywhere on $\li{x_0}_j$. The existence of such a sequence will be shown in Section \ref{appendix} via Proposition \ref{partitionprop}. The sequence constructed there monotonically increases to $1$ everywhere on $\li{x_0}_j$. By Monotone Convergence, it follows that
\begin{align*}
\prod_{j \, : \, p_j = 1} \left| \sup_{z_j \in \li{x_0}_j}  \frac{w_j(x_0,z_j)}{||d_x \pi_j(x_0,z_j)||} \right|^{q_j} \! \! \prod_{j \, : \, p_j > 1}  \! \left[ \int_{\li{x_0}_j} \frac{(w_j(x_0,y_j))^{p'_j} d \sigma_j(y_j)}{||d_x \pi_j(x_0,y_j)||^{p'_j-1}}  \right]^{\frac{q_j}{p_j'}} \! \! \leq c ||T||.
\end{align*}
This is exactly \eqref{testing} when $\omega$ is taken to be the standard $n$-tuple of coordinate vectors. The leap from this case to the full condition \eqref{testing} is accomplished by applying this inequality just derived to the family of defining functions $\pi_{j,M} (x,y_j) := \pi_j(x_0 + M(x-x_0),y_j)$ (with associated weight function $w_{j,M}(x,y_j) := w_j(x_0 + M(x-x_0), y_j$) for any $n \times n$ matrix $M$ of determinant $\pm1$. Because the determinant of $M$ is magnitude one, one trivially has by \eqref{bddness} and a change of variables that 
\[\int \prod_{j=1}^m |T_{j,M} f_j (x)|^{q_j} dx  \leq ||T|| \prod_{j=1}^m ||f_j||_{p_j}^{q_j}  \]
uniformly in $M$ for any nonnegative Borel measurable functions $f_j$.  Therefore
\begin{align*}
\prod_{j \, : \, p_j = 1} & \left| \sup_{z_j \in \li{x_0}_j}  \frac{w_{j,M}(x_0,z_j)}{||d_x \pi_{j,M}(x_0,z_j)||} \right|^{q_j} \\ & \cdot  \prod_{j \, : \, p_j > 1}  \! \left[ \int_{\li{x_0}_j} \frac{(w_{j,M}(x_0,y_j))^{p'_j} d \sigma_j(y_j)}{||d_x \pi_{j,M}(x_0,y_j)||^{p'_j-1}}  \right]^{\frac{q_j}{p_j'}} \! \! \leq c ||T||.
\end{align*}
Now $w_{j,M}(x_0,y_j) = w_j(x_0,y_j)$ for each $y_j \in \li{x_0}_j$. The proof is therefore finished once it is observed that $||d_x \pi_{j,M}(x_0,y_j)|| = ||d_x \pi_j(x_0,y_j)||_\omega$ when $\omega_1,\ldots,\omega_n$ are the columns of $M$. Because one can construct such an $M$ for any $\{\omega_i\}_{i=1}^n$ with $|\det \{\omega_i\}_{i=1}^n|=1$, the testing condition \eqref{testing} must hold.
\end{proof}

To close this section, a final, brief justification of an earlier remark is in order. By \eqref{integralside}, if exponents $p_j$ and $q_j$ were chosen so that
\[ \sum_{j=1}^m \frac{k_j q_j}{p_j} > n, \]
then one would instead have that
\begin{equation} \delta^{n - \sum_{j} \frac{k_j q_j}{p_j}} \int_{B_{c \delta}(x_0)} \prod_{j=1}^m |T_j f_{j,\delta} (x)|^{q_j} dx \label{altlim} \end{equation}
has a limit as $\delta \rightarrow 0^+$, the value of which is the same as the expression on the right-hand side of \eqref{thelimit}. If the product of the weights $\prod_{j=1}^m w_j(x,y_j)$ is anywhere nonzero, the limit of \eqref{altlim} can arranged to be nonzero. However, the product $\prod_{j=1}^m ||f_{j,\delta}||_{p_j}^{q_j}$ still remains finite as $\delta \rightarrow 0^+$ while a finite positive limit for \eqref{altlim} implies that the left-hand side of \eqref{bddness} tends to infinity, so one concludes that no such inequality \eqref{bddness} can hold in this case.

\section{Appendix}
\label{appendix}

This last section is devoted to resolving some technical issues which arise from the fact that the nondegeneracy assumption made for smooth incidence relations (namely, that $\Sigma$ contains only those points of $\Omega$ at which $||d_x \pi(x,y)||$ and $||d_y \pi(x,y)||$ are nonzero) is of a qualitative rather than quantitative nature. Opting for a qualitative approach (i.e., making no assumptions about positive lower bounds for $||d_x \pi(x,y)||$ and $||d_y \pi(x,y)||$ and setting no minimum distance between points of $\Sigma$ and points on the boundary of $\Omega$ or points at which $||d_x \pi(x,y)||$ or $||d_y \pi(x,y)||$ vanish) is in some sense necessitated by the fact that Theorem \ref{mainthm} is invariant under all volume-preserving affine coordinate changes in $x$ and $y$, a small fact which nevertheless plays an important role to conclude Theorem \ref{mainthm}'s proof near the end of Section \ref{necessarysec}. As a consequence, it is possible in principle for the submanifolds $\li{x}$ to have limit points at which certain undesirable things happen (e.g., points on the boundary of $\Omega$ or points at which $||d_y \pi||$ vanishes for which every neighborhood has infinite coarea measure).

The purpose of this section is to establish that there exist natural ways to localize around ``good'' compact sets via continuous functions of compact support such that no undesirable anomalies are encountered. Specifically, it will be established that the manifolds $\li{x}$ and their measures can be exhausted by such good compact sets, and that classical results like the Implicit Function Theorem hold in an appropriate sense around such good compact sets.

The first step is to make precise exactly what these good compact sets are. For the most part, the results in this section are sensitive only to one-sided behavior of defining functions, so as was done in Section \ref{computesubsec}, one of the two sides will be suppressed when possible.

Suppose $\Omega \subset \R^n$ is open and $\pi : \Omega \rightarrow \R^k$ is smooth. Let $E \subset \Omega$ be a compact set such that $\pi(x) = 0$ and $||d_x \pi(x)|| \neq 0$ for every $x \in E$. Any compact set $K$ containing $E$ will be called $\pi$-close to $E$ when $K$ is contained in $\Omega$ and when $\inf_{x \in K} ||d_x \pi(x)|| > 0$. An open neighborhood $U$ of $E$ will be called $\pi$-close to $E$ when its closure is compact and $\pi$-close to $E$.  Finally, a continuous function $\eta$ on $\R^n$ will be said to be supported close to $\Sigma_\pi := \set{x \in \Omega}{\pi(x) = 0 \text{ and } ||d_x \pi(x)|| > 0}$ when it is compactly supported, $\Sigma_\pi \cap \supp \eta$ is compact and $\supp \eta$ is $\pi$-close to $\supp \eta \cap \Sigma_\pi$.

The following proposition establishes that continuous cutoff functions exist which are supported close to $\Sigma_\pi$ and either exhaust all of $\Sigma_\pi$ or concentrate around any single $x \in \Sigma_\pi$.
\begin{proposition}
\label{partitionprop} For any open set $\Omega \subset \R^n$ and any smooth function $\pi : \Omega \rightarrow \R^{k}$, 
\begin{enumerate}
\item There is a sequence $\{\eta_i\}_{i=1}^\infty$ of continuous functions supported close to $\Sigma_\pi$ with values in $[0,1]$ such that 
\begin{equation} \int_{\Sigma_\pi} \eta_i d \sigma_\pi < \infty \label{sequencegood} \end{equation}
for each $i$ and such that $\eta_i(x)$ converges monotonically to $1$ as $i \rightarrow \infty$ for all $x \in \Sigma_\pi$. Here $\sigma_\pi$ is the coarea measure $\dH^{n-k}/||d_x \pi||$.
\item For any $x \in \Sigma_\pi$, there is a continuous function $\eta$ supported close to $\Sigma_\pi$ taking only values in $[0,1]$ having the property that
\begin{equation} 0 < \int_{\Sigma_\pi} \eta f d \sigma_\pi < \infty \label{pointgood} \end{equation}
for any nonnegative continuous function $f$ on $\Sigma_\pi$ which is not equal to zero at $x$.
\end{enumerate}
\end{proposition}
\begin{proof}
Consider the sequence $\{\eta_i\}_{i=1}^\infty$ from \eqref{sequencegood}. For each $x \in \Sigma_\pi$, $||d_x \pi(x)||$ is nonzero, so the Implicit Function Theorem guarantees the existence of a (Euclidean) ball $B_\delta(x)$ such that $\H^{n-k}(B_\delta(x) \cap \Sigma_\pi)$ is finite and nonzero, $\overline{B_\delta(x)} \cap \Sigma_{\pi}$ is closed, and $\inf_{x' \in \overline{B_\delta(x)}} ||d_x \pi(x')|| > 0$. Without loss of generality, it may be assumed that $\overline{B_\delta(x)} \subset \Omega$ and $\sigma_\pi (B_\delta(x))$ is finite and nonzero. This ball will for the moment be called the IFT ball at $x$. Let $\mathcal B$ be the collection of all balls in $\R^n$ with rational radius and center point having all rational coordinates. For each $B \in \mathcal B$, let $x_B$ be a point of $\Sigma_\pi$ such that the IFT ball at $x_B$ contains $B$; if no such point exists, $x_B$ remains simply undefined for this $B$. Note that it is \textit{not} necessary for $x_B$ to belong to $B$, only that the IFT ball associated to $x_B$ is sufficiently large that it contains all of $B$. The set of all points $x_B$ defined in this way is countable because $\mathcal B$ is. We claim that the union
\begin{equation} \bigcup \set{B \in \mathcal B}{ x_B \text{ is defined}} \label{defballs} \end{equation}
contains $\Sigma_\pi$. To see this, let $x' \in \Sigma_\pi$. This $x'$ has its own IFT ball $B_{\delta'}(x')$, and for any rational number $\delta'' < \delta'$, there is a rational point $x'' \in B_{\delta'}(x')$ such that $x' \in B_{\delta''}(x'') \subset B_{\delta'}(x')$. Because this ball $B'':= B_{\delta''}(x'')$ belongs to the rational collection $\mathcal B$ and is contained in the IFT ball $B_{\delta'}(x')$, the point $x_{B''}$ is defined for $B''$.  Therefore the entire ball $B''$ belongs to the union \eqref{defballs} and consequently $x'$ belongs to \eqref{defballs} as well.

Because \eqref{defballs} covers $\Sigma_\pi$ and because every $B$ in the union \eqref{defballs} is contained in an IFT ball, 
there must exist a countable collection $\{B_{\delta_j} (x_j)\}_{j=1}^\infty$ of IFT balls such that $\bigcup_{i=1}^\infty B_{\delta_j}(x_j)$ contains $\Sigma_\pi$. For each integer $N$, let 
\[ \varphi_{j,N} (x) := \begin{cases} \min \{ 1, N( \delta_j - ||x - x_j||) \} & x \in B_{\delta_j}(x_j), \\ 0 & x \in \R^{n} \setminus B_{\delta_j}(x_j). \end{cases} \]
For each $j$, $\varphi_{j,N}$ is a continuous function supported on $\overline{B_{\delta_j}(x_j)}$, and $\varphi_{j,N}(x)$ is a nondecreasing function of $N$ which tends to $1$ at all points of $B_{\delta_j} (x_j)$ and is zero everywhere else. Now define
\[ \eta_i(x) := \max_{j \in \{1,\ldots,i\}} \varphi_{j,i}(x) \]
for each $i$. It remains only to show that this sequence accomplishes the requirements of the first part of this proposition.
 First, the support of $\eta_i$ is contained in the union $\bigcup_{j=1}^i \overline{B_{\delta_j}(x_j)}$ and consequently the intersection of the support of $\eta_i$ and $\Sigma_\pi$ is compact. The union $\bigcup_{j=1}^i \overline{B_{\delta_j}(x_j)}$ is also $\pi$-close to $\Sigma_\pi \cap \bigcup_{j=1}^i \overline{B_{\delta_j}(x_j)}$ because $||d_x \pi(x)||$ is bounded below by a positive constant on each ball. Because $\eta_i \leq 1$ everywhere, the integral of $\eta_i$ with respect to $\sigma_\pi$ is bounded above by $\sigma_\pi(B_{\delta_1}(x_1)) + \cdots + \sigma_\pi(B_{\delta_i}(x_i))$, which is finite. Lastly, $\eta_i(x)$ is clearly nondecreasing as a function of $i$ for every $x$, and every $x \in \Sigma_\pi$ is contained in some ball $B_{\delta_j} (x_j)$ for some $j$, which means that $\eta_i(x) \rightarrow 1$ as $i \rightarrow \infty$ because $\varphi_{j,i}(x) \rightarrow 1$ as $i \rightarrow \infty$.

For the second part of the proposition, one can simply take $\eta$ to be $\varphi_{j,1}$ for the value of $j$ such that $x \in B_{\delta_j}(x_j)$. The integral
\[ \int_{\Sigma_\pi} \varphi_{j,1}(x') f(x') d \sigma_\pi(x') \]
is never zero because the integrand always strictly bounded below on some small ball centered at $x$ and the $(n-k)$-dimensional Hausdorff measure of $\Sigma_\pi$ intersected with any small ball centered at $x$ is positive, meaning that the measure of that ball with respect to $\sigma$ is positive.
\end{proof} 

The next step is to establish a number of local continuity/uniformity results for smooth perturbations of the map $\pi$. The main result in this direction is Lemma \ref{paramlemma}, but first an auxiliary proposition is necessary. This proposition can be thought of as a topological extension of the classical Implicit Function Theorem.
\begin{proposition} \label{CIFT}
Let $E,G \subset \R^n$ and $F \subset \R^{\ell}$ be compact sets.  Suppose that $\Phi(x,y)$ is a smooth map from some neighborhood of $E \times F$ into $\R^n$ such that 
\begin{itemize}
\item For each $y \in F$, the map $x \mapsto \Phi(x,y)$ is one-to-one on $E$.
\item For each $y \in F$ and $z \in G$, there is an $x \in E$ such that $\Phi(x,y) = z$.
\item For each $(x,y) \in E \times F$,  $D_x \Phi(x,y)$ is rank $n$.
\end{itemize}
Then there exist neighborhoods $U$ of $E$ and $V$ of $F$, each with compact closure, such that $\overline{U} \times \overline{V}$ is contained in the domain of $\Phi$ and
\begin{itemize}
\item For each $y \in \overline{V}$, the map $x \mapsto \Phi(x,y)$ is one-to-one on $\overline{U}$.
\item For each $y \in \overline{V}$ and $z \in G$, there is an $x \in U$ such that $\Phi(x,y) = z$.
\item For each $(x,y) \in \overline{U} \times \overline{V}$,  $D_x \Phi(x,y)$ is rank $n$.
\end{itemize}
\end{proposition}
\begin{proof}
When $E$, $F$, and $G$ are all singleton sets, this proposition is essentially the Implicit Function Theorem (as in that case, the assumptions are merely that points $x,y,z$ exist with $\Phi(x,y) = z$ and $D_x \Phi(x,y)$ is full rank, and the conclusion are that $\Phi(x',y')$ is one-to-one for all $x'$ in some neighborhood of $x$, provided $y'$ is sufficiently close to $y$, and that the equation $\Phi(x',y') = z$ has a solution $x'$ in the given neighborhood of $x$ for each $y'$ close to $y$). The bulk of the work to establish this proposition is to show that the local information provided by the Implicit Function Theorem can be consistently glued together on (presumably larger) compact sets.

Let $j$ be any positive integer and define $U_j$ to be the $1/j$-neighborhood of $E$ and $V_j$ be the $1/j$-neighborhood of $F$ with respect to Euclidean distance. Compactness of $E \times F$ guarantees that for all sufficiently large $j$, the closure of $U_j \times V_j$ will be contained in the open set on which $\Phi$ is defined and that $||d_x \Phi(x,y)|| > 0$ on $\overline{U_j \times V_j}$ as well. Suppose that among the indices $j$ satisfying these constraints, there is none such that $\Phi(\cdot,y)$ is one-to-one on $\overline{U_j}$ for all $y \in \overline{V_j}$. This implies the existence of twin sequences $(x_j,y_j)$ and $(x_j',y_j)$ belonging to $\overline{U_j} \times \overline{V_j}$ for each large $j$ such that $\Phi(x_j,y_j) = \Phi(x_j',y_j)$ but $x_j \neq x'_{j}$.  Let $\xi_j$, $\xi_j'$, and $\eta_j$ be points in $E$, $E$, and $F$, respectively such that $||x_j - \xi_j|| \leq 1/j$, $||x_j' - \xi'_j|| \leq 1/j$, and $||y_j - \eta_j|| \leq 1/j$ for each large $j$.
Compactness of $E$ and $F$ implies that one may pass to a subsequence ${j_i}$ such that $\xi_{j_i}$, $\xi'_{j_i}$, and $\eta_{j_i}$ converge as $i \rightarrow \infty$ to some points $\xi \in E,\xi' \in E$, and $\eta \in F$, respectively. Consequently $(x_{j_i},y_{j_i}) \rightarrow (\xi,\eta)$ and $(x'_{j_i},y_{j_i}) \rightarrow (\xi',\eta)$ as $i \rightarrow \infty$. Continuity of $\Phi$ also guarantees that $\Phi(\xi,\eta) = \lim_{i \rightarrow \infty} \Phi(x_{j_i},y_{j_i}) = \lim_{i \rightarrow \infty} \Phi(x'_{j_i},y_{j_i}) = \Phi(\xi',\eta)$, but by assumption on $\Phi$, $\Phi(\cdot,\eta)$ is one-to-one on $E$, implying that $\xi = \xi'$. Thus $(x_{j_i},y_{j_i})$ and $(x'_{j_i},y_{j_i})$ both converge to $(\xi,\eta)$ as $i \rightarrow \infty$. But now because $D_x \Phi(\xi,\eta)$ is nonsingular, the Implicit Function Theorem implies the existence of an open neighborhood $A \times B$ of $(\xi,\eta)$ such that $\Phi(\cdot,\eta)$ is one-to-one on $A$ for all $\eta \in B$. Since $(x_{j_i},y_{j_i})$ and $(x'_{j_i},y_{j_i})$ both belong to $A \times B$ for all $i$ sufficiently large, the equality $\Phi(x_{j_i},y_{j_i}) = \Phi(x'_{j_i},y_{j_i})$ forces $x_{j_i} = x'_{j_i}$ for all $i$ sufficiently large, causing a contradiction.

Note the following mild self-improvement of the result just obtained: because the sets $\overline{U_j}$ and $\overline{V_j}$ are decreasing as a function of $j$, it follows that for all sufficiently large $j,j'$ with $j' \geq j$, $\Phi(\cdot,y)$ must be one-to-one on $\overline{U_j}$ for all $y \in \overline{V_{j'}}$ and that $D_x \Phi$ must be full rank on $\overline{U_j} \times \overline{V_{j'}}$.
Suppose, for some such large fixed $j$, that there is no $j' \geq j$
 such that $G \subset \Phi({U_j},y)$ for all $y \in \overline{V_{j'}}$. This would imply the existence of a sequence of points $(y_{j'},z_{j'}) \in \overline{V_{j'}} \times G$ such that $\Phi({U_j},y_{j'})$ does not contain $z_{j'}$. Passing to a subsequence as above implies that there are points $\eta \in F$ and $\zeta \in G$ such that $(y_{j'_i},z_{j'_i}) \rightarrow (\eta,\zeta)$ as $i \rightarrow \infty$. By assumption, there exists an $x \in E$ such that $\Phi(x,\eta) = \zeta$. By the Implicit Function Theorem applied to $\Phi$ at the point $(x,\eta)$, there must exist a continuous function $y,z \mapsto \xi(y,z)$ defined on some neighborhood $B \times C$ of $(\eta,\zeta)$ such that $\xi(\eta,\zeta) = x$ and $\Phi(\xi(y,z),y) = z$ for all $(y,z) \in B \times C$. For all $i$ sufficiently large, $(y_{j'_i},z_{j'_i})$ belongs to $B \times C$ and consequently $x_{j'_i} := \xi(y_{j'_i},z_{j'_i})$ is well-defined for sufficiently large $i$ and satisfies $\Phi(x_{j'_i},y_{j'_i}) = z_{j'_i}$. Because $\xi$ is continuous, $x_{j'_i}$ converges to $x \in E$ as $i \rightarrow \infty$, meaning in particular that $x_{j'_i}$ must indeed belong to ${U_j}$ for $i$ large enough and consequently that $z_{j'_i} \in \Phi(U_j,y_{j'_i})$ after all. Therefore taking $V := V_{j'}$ for some sufficiently large $j' \geq j$ will achieve the desired conclusions of the proposition.
\end{proof} 

Proposition \ref{CIFT} will now be applied to establish Lemma \ref{paramlemma} below, which proves a number of important stability results for the level sets $\Sigma_\pi$ when the map $\pi$ is smoothly perturbed. Throughout the lemma, the parameter $p$ represents the smooth perturbation parameter while $x$ continues to represent the ``spatial'' variable in which the perturbed level sets reside. In particular, given a smooth map $\pi$ from some open subset of $\R^n$ into $\R^k$, the notation $\Sigma_\pi$ continues to refer to those points $x$ in the domain of $\pi$ for which $\pi(x) = 0$ and $||d_x \pi(x)|| \neq 0$.

\begin{lemma}
Suppose that $\pi$ is a smooth map defined on some open subset of $\R^{n} \times \R^{\ell}$ with values in $\R^k$; for each $p \in \R^{\ell}$, let $\Omega_p \subset \R^{n}$ be the set of points $x \in \R^{n}$ such that $(x,p)$ belongs to the domain of $\pi$ and let $\pi_p : \Omega_p \rightarrow \R^k$ be given by $\pi_p(x) := \pi(x,p)$ for each $x \in \Omega_p$. Assuming that $\Omega_0$ is nonempty, let $E \subset \Sigma_{\pi_0}$ be compact and let $K \subset \R^\ell$ be any compact set which is $\pi_0$-close to $E$. There exist an open set $\Omega$ containing $K$ which is $\pi_0$-close to $E$, a $\delta > 0$, and a smooth $\R^n$-valued map $\psi$ defined on a neighborhood of $\overline{\Omega \times B_{\delta}(0)} \subset \R^{n} \times \R^{\ell}$ such that \label{paramlemma}
\begin{itemize}
\item The domain of $\psi$ is contained in the domain of $\pi$ and $K \subset \Omega_p$ for all $||p|| \leq \delta$.
\item For each $||p|| \leq \delta$, the map $\psi_p : \overline{\Omega} \rightarrow \R^n$ given by $\psi_p(x) := \psi(x,p)$ is one-to-one and $D_x \psi_p(x)$ is nonsingular at every $x \in \overline{\Omega}$. Moreover, for each $||p|| \leq \delta$, $ K \subset \psi_p(\Omega) \subset \psi_{p}(\overline{\Omega}) \subset \Omega_p$ and
\[ \pi_p(\psi_p(x)) = \pi_0(x) \text{ for all } x \in \overline{\Omega}. \]
\item For all $x \in \overline{\Omega}$, $\psi_0(x) = x$. As $p \rightarrow 0$, $\psi_p(x)$ converges uniformly on $\overline{\Omega}$ to $x$ and $D_x \psi_p(x)$ converges uniformly on $\overline{\Omega}$ to the $n \times n$ identity matrix for all $x \in \overline{\Omega}$.
\item The set $\Sigma_{0,\Omega} := \Sigma_{\pi_0} \cap \Omega$ is an embedded $(n-k)$-dimensional submanifold of $\R^n$ with finite $(n-k)$-dimensional Hausdorff measure and $||d_x \pi_0||$ is bounded between two positive constants for all $x \in \Sigma_{0,\Omega}$. The $\pi_0$-coarea measure of $\Sigma_{0,\Omega}$ is finite as well.
\item For all $||p|| \leq \delta$, the map $\psi_p$ restricted to $\Sigma_{0,\Omega}$ parametrizes an $(n-k)$-dimensional embedded submanifold on which $\pi_p$ is identically zero. The submanifold contains all points $\Sigma_{p,K} := \set{ x \in K}{\pi_p (x) = 0}$.  The quantity $||d_x \pi_p(\psi_p(x))||$ is bounded uniformly between positive finite constants independent of $p$ on $\Sigma_{0,\Omega}$ and the Hausdorff and $\pi_p$-coarea measures of $\psi_p(\Sigma_{0,\Omega})$ are uniformly bounded for all $||p|| \leq \delta$. The measure $w_p \dH^{n-k}$ on $\Sigma_{0,\Omega}$ which pushes forward to Hausdorff measure $\dH^{n-k}$ on $\psi_p(\Sigma_{0,\Omega})$ has density $w_p$ which converges uniformly to $1$ on $\Sigma_{0,\Omega}$ as $p \rightarrow 0$.
\end{itemize}
\end{lemma}
\begin{proof} 
Let $E$ and $K$ be as indicated. Because $K$ is $\pi_0$-close to $E$, it is in particular true that $||d_x \pi_0(x)||$ must be well-defined and strictly positive on some neighborhood of $K$. Consequently, any point $x \in K$ at which $\pi_0(x) = 0$ automatically belongs to $\Sigma_{\pi_0}$ (i.e., it has $||d_x \pi_0(x)|| > 0)$. Continuity of $\pi_0$ and compactness of $K$ guarantee that $K \cap \Sigma_{\pi_0}$ is compact, and one may invoke the Implicit Function Theorem to cover it by finitely many balls $B$ of finite radius such that that $K \cap \Sigma_{\pi_0} \cap B$ is an embedded submanifold of dimension $n-k$ in $\R^n$ with finite $(n-k)$-dimensional Hausdorff measure as well. Thus, geometrically, the level sets of $\pi_0$ are straightforward on $K$ (or any set which is $\pi_0$-close to $E$).

Now consider the map
\[ \Phi(x,\zeta,p) := \pi(x + (D_x \pi(x,0))^T \zeta,p) - \pi(x,0),  \]
which is well-defined on some open subset containing $K \times \{0\} \times \{0\} \subset \R^{n} \times \R^{k} \times \R^{\ell}$. At any point $(x,0,0)$ for $x \in K$, the Jacobian matrix $D_\zeta \Phi(x,0,0)$ equals $D_x \pi_0(x) (D_x \pi_0(x))^T$, which is nonsingular because $D_x \pi_0(x)$ is full rank at all points of $K$.  Additionally, for all $(x,p) \in K \times \{0\}$, $\Phi(x,0,p) = 0$. Proposition \ref{CIFT} will now be applied in a slightly counterintuitive way: $\Phi(x,\zeta,p)$ will be regarded as a function of $\zeta \in \R^k$ with parameters $(x,p) \in \R^{n+\ell}$. The set $E$ from Proposition \ref{CIFT} will simply be $\{0\} \subset \R^k$, and likewise $G := \{0\} \subset \R^k$. The set $F$ will be $K \times \{0\}$.
By Proposition \ref{CIFT}, there exists a neighborhood of $K \times \{0\} \subset \R^{n} \times \R^{\ell}$, which without loss of generality may be written $U_1 \times B_{\delta_1}(0)$ for some neighborhood $U_1$ of $K$ and some $\delta_1 > 0$, and an $\epsilon > 0$ such that 
$\zeta \mapsto \Phi(x,\zeta,p)$ is one-to-one for all $(x,p) \in \overline{U_1} \times \overline{B_{\delta_1}(0)}$ when $\zeta$ belongs the ball $\overline{B_\epsilon(0)}$. Moreover $\Phi(x,\zeta,p) = 0$ has a unique solution $\zeta \in B_{\epsilon}(0)$ for all $(x,p) \in \overline{U_1} \times \overline{B_{\delta_1}(0)}$. Thus there is a map $\zeta : \overline{U_1} \times \overline{B_{\delta_1}(0)} \rightarrow B_{\epsilon}(0)$ for which $\Phi(x,\zeta(x,p),p) = 0$. Because $D_\zeta \Phi$ is full rank for all $(x,\zeta,p) \in \overline{U_1} \times \overline{B_\epsilon(0)} \times \overline{B_{\delta_1}(0)}$, the Implicit Function Theorem guarantees that this map $\zeta(x,p)$ is smooth on $U_1 \times B_{\delta_1}(0)$. Since $\zeta \mapsto \Phi(x,\zeta,p)$ is one-to-one for any $(x,p) \in \overline{U_1} \times \overline{B_{\delta_1}(0)}$ and $\Phi(x,0,0) = 0$ for all $x \in \overline{U_1}$, it follows that $\zeta(x,0) = 0$ for all $x \in \overline{U_1}$. Shrinking $U_1$ as necessary, it may be assumed that $U_1$ is $\pi_0$-close to $E$. 

Now consider the smooth function $\psi(x,p) := x + (D_x \pi(x,0))^T \zeta(x,p)$. This map is well-defined on $\overline{U_1} \times \overline{B_{\delta_1}(0)}$ and smooth on the interior. When $p = 0$, $\psi(x,0)$ is the identity map on $\overline{U_1}$, so in particular $D_x \psi(x,0)$ is the identity matrix on all of $K$. It is also trivially true that all $z \in K$ admit an $x \in K$ such that $\psi(x,0) = z$ (namely, $x = z$).  Let us also temporarily restrict the domain of $\psi$ to only those pairs $(x,p) \in U_1 \times B_{\delta_1}(0)$ which are at a distance at least $\delta_0 > 0$ to the boundary of the domain of $\pi$ and have such $||d_x \pi(x,p)|| > c$ and $||d_x \pi(\psi_p(x),p)|| > c$ for some fixed $c > 0$. This can clearly be done in such a way that the set $K \times \{0\}$ still belongs entirely to the domain of $\psi$ because $K$ is compact and $||d_x \pi(x,0)|| = ||d_x \pi(\psi_0(x),0)||$ is never equal to zero for any $x \in K$. Now apply Proposition \ref{CIFT} again: let $E$ and $G$ both equal the set $K \subset \R^n$ and let $F$ equal $\{0\} \subset \R^\ell$ and apply to the map $\phi(x,p)$. It follows that there exists an open set $\Omega$ containing $K$ and a $\delta > 0$ such that $\psi$ is defined on $\overline{\Omega} \times \overline{B_{\delta}(0)}$ and has the property that $D_x \psi(x,p)$ is nonsingular on $\overline{\Omega} \times \overline{B_{\delta}(0)}$, that $\psi_p(x) := \psi(x,p)$ is one-to-one on $\overline{\Omega}$ for each $||p|| \leq \delta$, and $K \subset \psi_p(\Omega)$ for all $||p|| \leq \delta$. Because the domain of $\psi$ was temporarily restricted before the application of the proposition, $\Omega$ is $\pi_0$-close to $E$, the closure of the domain of $\psi$ is contained in the domain of $\pi$, and $||d_x \pi_p(x)|| \geq c$ for all $(x,p) \in \overline{\Omega} \times \overline{B_\delta(0)}$. (And also note that $\psi$ is defined and smooth on $U_1 \times B_{\delta_1}(0)$, which is an open neighborhood of $\overline{\Omega} \times \overline{B_\delta(0)}$.) 

It remains only work through the promised bullet points:
\begin{itemize}
\item Taking the domain of $\psi$ to be $U_1 \times B_{\delta_1}(0)$, it is immediate by the first application of Proposition \ref{CIFT} that this product set is contained in the domain of $\pi$. Moreover $K \subset U_1$, so $K \times B_{\delta_1}(0)$ is contained in the domain of $\pi$, meaning $K \subset \Omega_p$ for all $||p|| < \delta_1$.
\item It is a direct consequence of the second application of Proposition \ref{CIFT} that $\psi_p$ is one-to-one on $\overline{\Omega}$ for all $||p|| \leq \delta$ and that the Jacobian $D_x \psi_p(x)$ is never singular on $\overline{\Omega} \times \overline{B_\delta(0)}$. It is also immediate that $K \subset \psi_p(\Omega)$ for each $||p|| \leq \delta$. It is less obvious that $\psi_p(\overline{\Omega}) \subset \Omega_p$ for each $||p|| \leq \delta$. To see why, observe that $\zeta(x,p)$ is defined on all of $U_1 \times B_{\delta_1}(0)$, which necessarily requires that $\psi(x,p) \in \Omega_p$ for any $x \in U_1$ and $||p|| < \delta_1$ (by virtue of the definitions of $\zeta(x,p)$ and $\psi(x,p)$).
Because $\overline{\Omega} \times \overline{B_{\delta}(0)}$ is contained in $U_1 \times B_{\delta_1}(0)$, it must be the case that $\psi_p(\overline{\Omega}) \subset \Omega_p$ for all $||p|| \leq \delta$. Finally, the identity $\pi_p(\psi_p(x)) = \pi_0(x)$ follows directly from the definitions of $\zeta(x,p)$ and $\psi(x,p)$.
\item The observation $\zeta(x,0) = 0$ for all $x \in \overline{U_1}$ guarantees that $\psi_0(x) = x$ for all $x \in U_1$. Smoothness of $\psi(x,p)$ in both variables and compactness of $\overline{\Omega} \subset U_1$ give the desired uniform convergence properties as $p \rightarrow 0$.
\item Because $\overline{\Omega}$ is $\pi_0$-close to $E$, it is necessary (just as was shown for $K$ itself at the beginning of the proof of this lemma) for $\Sigma_{0,\Omega}$ to be an embedded submanifold with finite $(n-k)$-dimensional Hausdorff measure. Because $\overline{\Omega}$ is compact and $||d_x \pi_0(x)||$ is strictly positive there, $||d_x \pi_0(x)||$ must be uniformly bounded above and below on $\overline{\Omega}$ by positive constants, which implies comparability of the Hausdorff and $\pi_0$-coarea measures on $\Sigma_{0,\Omega}$.
\item The image $\psi_p(\Sigma_{0,\Omega})$ is an embedded submanifold because $\psi_p$ is one-to-one and has everywhere nonsingular Jacobian on some open set containing the (compact) closure of $\Sigma_{0,\Omega}$. On this submanifold, $\pi_p$ is identically zero simply because $\pi_p(\psi_p(x)) = \pi_0(x)$ and $\pi_0(x)$ vanishes on $\Sigma_{0,\Omega}$. The image $\psi_p(\Sigma_{0,\Omega})$ contains all points in $K$ at which $\pi_p = 0$ by virtue of the fact that $K$ belongs to the domain of $\psi^{-1}_p$ for all $||p|| \leq \delta$: if $\pi_p(x) = 0$ for some $x \in K$ and some $||p|| \leq \delta$,  $0 = \pi_p(x) = \pi_p( \psi_p( \psi_p^{-1}(x))) = 0  = \pi_0(\psi_p^{-1}(x))$, so $\psi_p^{-1}(x)$ is the point in $\Omega$ whose image via $\psi_p$ is $x \in K$. Because of the temporary restriction of the domain of $\psi$ imposed before the second application of Proposition \ref{CIFT}, it must be the case that $||d_x \pi(x,p)||$ and $||d_x \pi(\psi_p(x),p)||$ are bounded uniformly above and below by positive constants on all of $\overline{\Omega} \times \overline{B_{\delta}(0)}$, which guarantees comparability of $(n-k)$-dimensional Hausdorff measure and the $\pi_p$-coarea measure on the submanifold $\psi_p(\Sigma_{0,\Omega})$.
Finiteness of the Hausdorff measure follows from the fact that the norms of $||D_x \psi_p||$ and $||D_x \psi_p^{-1}||$ are necessarily uniformly bounded above and below on all of $\overline{\Omega}$ by continuity of $D_x \psi_p$ and the fact that it is everywhere full rank.  The density $w_p(x)$ on $\Sigma_{0,\Omega}$ that pushes forward to Hausdorff measure must equal $\sqrt{ \det (D_x \psi_p(v_i) \cdot D_x \psi_p(v_j))_{i,j \in \{1,\ldots,n-k\}}}$ for any $n-k$ orthonormal vectors $v_1,\ldots,v_{n-k}$ tangent to $\Sigma_{0,\Omega}$ at the point $x \in \Sigma_{0,\Omega}$. Since $D_x \psi_p$ converges uniformly to the identity on $\overline{\Omega}$, this density must converge uniformly to $1$ as $p \rightarrow 0$.
\end{itemize}
This finishes the proof of the lemma.
\end{proof} 

A key corollary of Lemma \ref{paramlemma} is that integrals on $\Sigma_{\pi_p}$ are continuous functions of $p$ at $p=0$ when the integrands are continuous and supported close to $\Sigma_{\pi_0}$.
\begin{corollary}
If $\pi_p$ is a smooth family of maps from open sets in $\R^n$ into $\R^k$ (such that the set in $\R^{n} \times \R^{\ell}$ of points $(x,p)$ in the domain of of $\pi_p(x)$ is open and $\pi$ is smooth as a function of $(x,p)$ there) and $\eta$ is a continuous function on $\R^n$ which is supported close to $\Sigma_{\pi_{0}}$, then \label{integralcorollary} 
\begin{equation} \lim_{p \rightarrow {0}} \int_{\Sigma_{\pi_p}} \eta d \sigma_{\pi_p} = \int_{\Sigma_{\pi_0}} \eta d \sigma_{\pi_0}. \label{limformula} \end{equation}
\end{corollary}
\begin{proof}
Apply Lemma \ref{paramlemma} with $K$ equal to the support of $\eta$ and $E := \supp \eta \cap \Sigma_{\pi_0}$. For all $||p|| \leq \delta$, the integral 
\[ \int_{\Sigma_{\pi_p}} \eta d \sigma_{\pi_p} \]
must be finite because $\eta$ is bounded and the intersection of the support of $\eta$ with $\Sigma_{\pi_p}$ is exactly the set $\Sigma_{p,K}$ from Lemma \ref{paramlemma}, which has finite $\pi_p$-coarea measure for all sufficiently small $p$.  For all such $p$, one can parametrize $\Sigma_{\pi_p} \cap K$ in terms of $\psi_p$ to obtain the identity
\begin{align*}
\int_{\Sigma_{\pi_p}} \eta d \sigma_{\pi_p} & = \int_{\Sigma_{\pi_p}} \eta \frac{\dH^{n-k}}{||d_x \pi_p||}
= \int_{\Sigma_{0,\Omega}} \eta(\psi_p(x)) w_p(x) \frac{\dH^{n-k}(x)}{||d_x \pi_p(\psi_p(x))||}.
\end{align*}
Now for every $||p|| \leq \delta$, the function
\[ \frac{\eta(\psi_p(x)) w_p(x)}{||d_x \pi_p(\psi_p(x))||} \]
is bounded on as a function of $x \in \Sigma_{0,\Omega}$ and $p$, and it moreover converges pointwise to $\eta(x) / ||d_x \pi_0(x)||$ for all $x$ in its support as $p \rightarrow 0$. Finiteness of the $(n-k)$-dimensional Hausdorff measure of $\Sigma_{0,\Omega}$ allows one to use the Lebesgue Dominated Convergence Theorem to immediately conclude \eqref{limformula} for any sequence of parameters $p$ tending to $0$.
\end{proof}

The final technical result of this section is the Fubini-type identity for coarea measure which was used in Section \ref{proofsec1} to prove Theorem \ref{main1}.
\begin{proposition}
Suppose $(\Omega,\pi,\Sigma)$ is a smooth incidence relation on $\R^n \times \R^{n'}$ of codimension $k$.
If $F$ is any nonnegative Borel measurable function on $\Sigma$, then
\begin{equation}
\int_{\R^{n}} \left[ \int_{\li{x}} F d \sigma \right] dx = \int_{\R^{n'}} \left[ \int_{\ri{y}} F d \sigma \right] dy. \label{mainfubini}
\end{equation}
\end{proposition}
\begin{proof}
Let $g$ be any function of compact support contained in $\Omega$ such that both $||d_x \pi(x,y)||$ and $||d_y \pi(x,y)||$ are both strictly positive on the support of $g$.
The coarea formula (e.g., Federer \cite{federer1969} Theorem 3.2.12 when $k < \min\{n,n'\}$ and Theorem 3.2.5 when $k = \min\{n,n'\}$) dictates that for any continuous $\varphi$ on $\R^{k}$,
\begin{align*}
\int \varphi(\pi(x,y)) g(x,y) dx dy & = \int \varphi(u) \left[ \int_{\R^{n'}} \int_{x \, : \, \pi(x,y) = u} g(x,y) \frac{\dH^{n-k}(x)}{||d_x \pi(x,y)||} dy \right] du \\
& = \int \varphi(u) \left[ \int_{\R^{n}} \int_{y \, : \, \pi(x,y) = u} g(x,y) \frac{\dH^{n'-k}(y)}{||d_y \pi(x,y)||} dx \right] du \\
& = \int \varphi(u) \left[ \int_{(x,y)  \, : \, \pi(x,y) = u} g(x,y) \frac{\dH^{n + n'-k}(x,y)}{||d_{x,y} \pi(x,y)||} \right] du. 
\end{align*}
Here the identity \eqref{genmat} has been implicitly used, as the coarea formula is typically written with $\sqrt{ \det (D_x \pi(x,y)) (D_x \pi(x,y))^T}$ in the place of $||d_x \pi(x,y)||$, etc. Likewise, in the final equality, the coarea formula is applied treating both $x$ and $y$ variables as components of a single point $(x,y) \in \R^{n+n'}$. Because $||d_{x,y} \pi(x,y)||^2 \geq ||d_x \pi(x,y)||^2 + ||d_y \pi(x,y)||^2$, it must be the case that $||d_{x,y} \pi(x,y)||$ is bounded below on the support of $g$.

For each $u \in \R^k$ and each $y \in \R^{n'}$, the function $g(\cdot,y)$ is supported close to the smooth zero set of the map $y \mapsto \pi(x,y) - u$ because the lower bound for $||d_x \pi||$ on the support of $g$ guarantees that the intersection of this support with the zero set of $\pi(x,y) - u$ will also be compact. Consequently, Corollary \ref{integralcorollary} implies that the map
\[ (y,u) \mapsto \int_{x \, : \, \pi(x,y) = u} g(x,y) \frac{\dH^{n-k}(x)}{||d_x \pi(x,y)||} \]
is continuous. The support of this function in the variables $(y,u)$ is also necessarily compact, and so as a consequence 
\[ u \mapsto \int_{\R^{n'}} \int_{x \, : \, \pi(x,y) = u} g(x,y) \frac{\dH^{n-k}(x)}{||d_x \pi(x,y)||} dy \]
is a continuous function of $u$. Likewise
\[ \int_{\R^{n}} \int_{y \, : \, \pi(x,y) = u} g(x,y) \frac{\dH^{n'-k}(y)}{||d_y \pi(x,y)||} dx \]
is a continuous function of $u$, as is
\[ u \mapsto  \int_{(x,y)  \, : \, \pi(x,y) = u} g(x,y) \frac{\dH^{n + n'-k}(x,y)}{||d_{x,y} \pi(x,y)||}. \]
 By choosing an appropriate sequence of functions $\varphi$ concentrating around the point $u=0$, it follows that
\begin{equation} 
\begin{split}
\int_{\R^{n'}} \left[ \int_{\ri{y}} g(x,y) \frac{\dH^{n-k}(x)}{|| d_x \pi(x,y)||} \right] dy & = \int_{\R^{n}} \left[ \int_{\li{x}} g(x,y) \frac{\dH^{n'-k}(y)}{|| d_y \pi(x,y)||} \right] dx \\
& = \int_{\Sigma} g(x,y) \frac{\dH^{n+n'-k}(x,y)}{||d_{x,y} \pi(x,y)||}.
\end{split}
 \label{fubini} \end{equation}
 The goal now is to extend \eqref{fubini} to successively larger classes of functions $g$.

Fix any compact set $K \subset \Omega$ on which $||d_x \pi(x,y)||$ and $||d_y \pi(x,y)||$ never vanish. By compactness of $K$ and the Implicit Function Theorem, there is a neighborhood $U$ of $K \cap \Sigma$ such that $U \cap \Sigma$ is an embedded submanifold of $\R^{n+n'}$ of codimension $k$ with finite Hausdorff measure.  Let $E \subset K \cap \Sigma$ have $(n+n'-k)$-dimensional Hausdorff measure equal to zero. For any positive $\epsilon$ and $\delta$, we may cover $E$ by countably many Euclidean balls $B_j$ of radius at most $\delta$ such that $\sum_j \dH^{n+n'-k}( B_j \cap E) < \epsilon$. By taking $\delta$ sufficiently small, it may be assumed for each $j$ that \eqref{fubini} holds when $g := \varphi_j$ for some continuous nonnegative function $\varphi_j$ which is identically $1$ on $B_j$ and identically zero on the complement of the double of $B_j$. By Monotone Convergence, the identity \eqref{fubini} holds also for $g := \sum_j \varphi_j$. Since $g$ dominates the characteristic function of $E$, letting $\delta \rightarrow 0^+$ implies that \eqref{fubini} also holds whenever $g$ is the characteristic function of a $(n+n'-k)$-dimensional null set contained in $K \cap \Sigma$. An important consequence is that any such null set $E$ has the property that $\li{x} \cap E$ is an $(n'-k)$-dimensional null set for almost every $x \in \R^n$ and similarly $\ri{y} \cap E$ is an $(n-k)$-dimensional null set for almost every $y \in \R^{n'}$.

Now continuous functions on the submanifold $U \cap \Sigma$ are dense in the integrable functions with respect to $(n+n'-k)$-dimensional Hausdorff measure on $U \cap \Sigma$. Likewise, all continuous functions supported on the submanifold $U \cap \Sigma$ have extensions to $\R^{n} \times \R^{n'}$ supported on some fixed compact set $\tilde K$ containing $K$ and having the property that $||d_x \pi(x,y)||$ and $||d_y \pi(x,y)||$ never vanish on $\tilde K$. Thus given any Borel function $g$ on $U \cap \Sigma$ which is integrable with respect to $(n+k'-k)$-dimensional Hausdorff measure on $\Sigma$, there must be a sequence $\{g_j\}_{j=1}^\infty $ be a sequence of continuous functions on $\tilde K$ such that
\[ \int_{\Sigma} |g(x,y) - g_j(x,y)| \frac{\dH^{n+n'-k}(x,y)}{||d_{x,y} \pi(x,y)||} \leq 2^{-j}. \]
Exponential convergence implies that
\begin{equation} \int_{\Sigma} \sum_{j=1}^\infty |g(x,y) - g_j(x,y)| \frac{\dH^{n+n'-k}(x,y)}{||d_{x,y} \pi(x,y)||} < \infty \label{converge} \end{equation}
which in turn means that the sum over $j$ inside \eqref{converge} is finite on $U \cap \Sigma$ except possibly for some $(n+n'-k)$-dimensional null set. This means that $g_j$ must converge to $j$ almost everywhere on $U \cap \Sigma$, and consequently it means that for almost every $x$, $g_j(x,y)$ converges to $g(x,y)$ almost everywhere on $\li{x}$ (and likewise $g_j$ converges to $g$ almost everywhere on $\ri{y}$ for almost every $y$). By Dominated Convergence (using the dominating function $|g_1(x,y)| + \sum_{j > 1} |g_{j}(x,y) - g_{j-1}(x,y)|$) one can apply \eqref{fubini} to each function in the sequence $g_j$ and pass to the limit as $j \rightarrow \infty$ to conclude that \eqref{fubini} holds for $g$ itself. If $g$ is nonnegative but not integrable on $\Sigma$, the identity \eqref{fubini} can still be seen to be true for $g$ by bounding it below by a sequence of Borel-measurable simple functions whose integrals tend to $\infty$.

Finally, taking a sequence of sets $K_m$ defined to be those points $(x,y) \in \Omega$ at which $||x|| + ||y|| \leq m, \dist((x,y),\Omega^c) \geq 1/m, ||d_x \pi(x,y)|| \geq 1/m$, and  $||d_y \pi(x,y)|| \geq 1/m$ gives an increasing sequence of compact sets on which $||d_x \pi(x,y)||$ and $||d_y \pi(x,y)||$ never vanish. The union of these sets is exactly the set of points in $\Omega$ at which $||d_x \pi(x,y)||$ and $||d_y \pi(x,y)||$ are nonzero, and so
\[ \lim_{m \rightarrow \infty}  g(x,y) \chi_{K_m}(x,y) = g(x,y) \chi_{(x,y) \in \Omega, ||d_x \pi(x,y)||, ||d_y \pi(x,y)|| > 0}. \]
Because \eqref{fubini} is already known to hold for the function $g(x,y) \chi_{K_m}(x,y)$, it follows by Monotone Convergence that \eqref{fubini} must hold for any nonnegative Borel-measurable function $g$.
\end{proof}

\bibliography{mybib}

\end{document}